\documentclass[11pt]{article} 
\usepackage[a4paper, total={6in, 9in}]{geometry}
\usepackage{amssymb}
\usepackage{amsmath}
\usepackage{todonotes}
\usepackage{setspace}
\usepackage{slashed} 
\usepackage{enumitem}
\usepackage{changepage}
\usepackage{perpage}
\usepackage{amsfonts}
\usepackage{amssymb}
\usepackage{tikz-cd}
\usepackage{adjustbox}
\usepackage{mathrsfs}
\usepackage{amsbsy}
\usepackage{xcolor}
\usepackage{soul}
\usepackage{stmaryrd}
\usepackage{extarrows}
\usepackage{varwidth}
\usepackage{array}
\usepackage{tikz}
\usepackage{tikz-cd}
\usetikzlibrary{matrix}
\tikzcdset{cramped}
\usepackage{gensymb}
\usepackage{bm}
\usepackage{graphicx}
\usepackage{tensor}
\usepackage{hyperref}
\usepackage{mathtools}
\usepackage{braket}
\usepackage{bbm} 
\usepackage[symbol]{footmisc}

\usepackage[square,sort,comma,numbers]{natbib}
\usepackage{amsthm}
\theoremstyle{plain}
\newtheorem{theorem}{Theorem}[section]
\newtheorem{lemma}[theorem]{Lemma}
\newtheorem{proposition}[theorem]{Proposition}
\newtheorem{corollary}[theorem]{Corollary}
\newtheorem{question}[theorem]{Question}

\newtheorem{conjecture}[theorem]{Conjecture}
\theoremstyle{definition}
\newtheorem{definition}[theorem]{Definition}
\newtheorem{example}[theorem]{Example}
\newtheorem{remark}[theorem]{Remark}

\newtheorem{prop/def}[theorem]{Proposition/Definition}

\numberwithin{figure}{section}
\numberwithin{table}{section}
\numberwithin{equation}{section}

\makeatletter
\newcommand*{\transpose}{%
  {\mathpalette\@transpose{}}%
}
\newcommand*{\@transpose}[2]{%
  \raisebox{\depth}{$\m@th#1\intercal$}%
}
\makeatother

\newcommand*\Ext{\underline{\textnormal{Hom}}}
\newcommand*\A{\mathbb{A}}
\newcommand*\BB{\mathbb{B}}
\newcommand*\CC{\mathbb{C}}

\newcommand*\EE{\mathbb{E}}

\newcommand*\GG{\mathbb{G}}

\newcommand*\LL{\mathbb{L}}

\newcommand*\N{\mathbb{N}}

\newcommand*\Q{\mathbb{Q}}

\newcommand*\ZZ{\mathbb{Z}}

\newcommand{\mA}{{\mathcal A}}
\newcommand{\mB}{{\mathcal B}}
\newcommand{\mC}{{\mathcal C}}

\newcommand{\mE}{{\mathcal E}}
\newcommand{\mF}{{\mathcal F}}

\newcommand{\mL}{{\mathcal L}}
\newcommand{\mM}{{\mathcal M}}
\newcommand{\mN}{{\mathcal N}}
\newcommand{\mP}{{\mathcal P}}

\newcommand{\mS}{{\mathcal S}}

\newcommand{\mV}{{\mathcal V}}

\newcommand*\coker{\textnormal{coker}}

\newcommand*\inv{\textnormal{inv}}
\newcommand*\vir{\textnormal{vir}}
\newcommand*\uvir{\textnormal{uvir}}
\newcommand*\ch{\textnormal{ch}}

\newcommand*\dg{\textnormal{dg}}

\newcommand*\mO{\mathcal{O}}
\newcommand{\pl}{{\textnormal{pl}}}
\newcommand{\ses}{{\textnormal{ss}}}
\newcommand{\pa}{{\textnormal{pa}}}
\newcommand{\ttop}{\textnormal{top}}

\newcommand*\Perf{\textnormal{Perf}_{\CC}}
\newcommand*\Hilb{\textnormal{Hilb}^n(X)}
\newcommand*\Sym{\textnormal{Sym}}
\newcommand*\Quot{\textnormal{Quot}_S(\mathbb{C}^N,n)}

\newcommand\numberthis{\addtocounter{equation}{1}\tag{\theequation}}

\newcommand\Item[1][]{%
  \ifx\relax#1\relax  \item \else \item[#1] \fi
  \abovedisplayskip=0pt\abovedisplayshortskip=0pt~\vspace*{-\baselineskip}}
  \renewcommand{\text}[1]{\textnormal{#1}}

\newcommand{\ov}[1]{\overline{#1}}
\MakePerPage{footnote}

\title{Hilbert Schemes of Points on Calabi--Yau 4-Folds via Wall-Crossing}
\author{Arkadij Bojko \thanks{Address: Institute of Mathematics, Academia Sinica,
6F, Astronomy-Mathematics Building,
No. 1, Sec. 4, Roosevelt Road,
Taipei 10617, Taiwan; Email: abojko@gate.sinica.edu.tw
}}
\date{}

\begin{document}

\maketitle
\begin{abstract}
 Gross--Joyce--Tanaka \cite{GJT} proposed a wall-crossing conjecture for Calabi--Yau fourfolds.  Assuming it, we prove the conjecture of Cao--Kool \cite{CK1} for 0-dimensional sheaf-counting invariants on projective Calabi--Yau 4-folds. From it, we extract the full topological information contained in the virtual fundamental classes of Hilbert schemes of points which turns out to be equivalent to the data of all descendent integrals.   
 As a consequence, we can express many generating series of invariants in terms of explicit universal power series.
\begin{enumerate}[label=\roman*)]
    \item  On $\CC^4$, Nekrasov proposed invariants with a conjectured closed form \cite{NMF}.  We show that an analog of his formula holds for compact Calabi--Yau 4-folds satisfying the wall-crossing conjecture.
    \item   We notice a relationship to corresponding generating series for Quot schemes on elliptic surfaces which are also governed by a wall-crossing formula. This leads to a Segre--Verlinde correspondence for Calabi--Yau fourfolds. 
\end{enumerate}
\end{abstract}
%\begin{center}
%\begin{tabular}{p{5.2cm} p{7.5cm}}
%\textbf{Subject codes (primary):} & 14D23,14N10\\
%\textbf{Subject codes (secondary):} & 17B69, 55R40, 05A15 
%\end{tabular}
%\end{center}
\tableofcontents
\setstretch{1.2}
\section{Introduction}
\label{sec:intro}
Following the philosophy of Donaldson \cite{Do83}, who defined invariants counting \textit{anti-self-dual connections} on a real 4-manifold, Donaldson--Thomas proposed a holomorphic version of this construction for a Calabi--Yau fourfold $X$ in \cite{DT96}. To give a rigorous formulation of their ideas, a new approach to sheaf counting was pioneered by Borisov--Joyce \cite{BJ} (using derived differential geometry) and Oh-Thomas \cite{OT}\footnote{See also the work of Cao--Leung  \protect\cite{CL} which focused on some special cases.}. They constructed new virtual fundamental classes (from now on VFCs for short) for moduli spaces of sheaves on Calabi--Yau 4-folds, which unlike their analogs obtained using Behrend--Fantechi \cite{BF}, are not canonically determined by their deformation and obstruction theory. 

The additional input needed to use the machinery of \cite{BJ} and \cite{OT} is a choice of orientation on the moduli space, the existence of which was proved in \cite{CGJ} for projective Calabi--Yau 4-folds and by the author \cite{bojko} for quasi-projective ones. It should be noted that this leads to complications that need to be addressed.

In this paper, we focus on the compact case. Some results in this setting have been obtained by Park \cite{Huy} extending the ones of Cao--Kool \cite{CK1}, but the majority of computations have been done for local 4-folds. Because for now, the construction of the deformation invariant VFC requires the Calabi--Yau condition even for ideal sheaves of points, one is unable to use algebraic cobordism to reduce everything to toric computations as in \cite{Li, LP}. Instead, we will use the conjectural wall-crossing along the lines of Gross--Joyce--Tanaka \cite{GJT}. 

Parallel results using toric computations have been obtained by Martijn Kool and Jørgen Rennemo \cite{KRdraft} on $\CC^4$ in the process of writing this paper. Unlike the Donaldson--Thomas invariants of Hilbert schemes of points for 3-folds, where this would describe the corresponding invariants of all compact 3-folds, there is currently no relation between the local and compact setting for Calabi--Yau 4-folds. Additionally, the computation in \cite{KRdraft} uses localization to the fixed point locus of $\textnormal{Hilb}^n(\CC^4)$ with respect to the Calabi--Yau torus $\GG_m^3$. Then the problem of orientations reduces to choices of signs at each isolated fixed point which are determined in \cite{KRdraft} using a global description of $\textnormal{Hilb}^n(\CC^4)$ as the vanishing locus of an isotropic section. However, when gluing copies of $\CC^4$ to larger toric Calabi--Yau 4-folds, there is currently no indication whether the signs in \cite{KRdraft} should be modified by some global sign for each copy of $\CC^4$.

\paragraph{Wall-crossing conjecture for Calabi--Yau 4-folds}
We use the term Calabi--Yau 4-fold for a smooth projective 4-fold $X$ with a trivial canonical bundle $K_X$ and $H^2(\mathcal{O}_X) = 0$, as this is an assumption, we will need to obtain all of our results. The general setup of wall-crossing for Calabi--Yau fourfolds explained in §\ref{VAinAG} does not need the last condition. We also always assume that $X$ is connected as everything can be easily generalized to multiple components.

We use $K^0(X)$ to denote the complex topological K-theory of $X$. Let $\alpha\in K^0(X)$  and $M_{\alpha,L}$ be a projective moduli scheme of perfect complexes with fixed determinant $L$ in class $\alpha$  satisfying some stability condition, then Oh--Thomas \cite{OT} or Borisov--Joyce \cite{BJ} give us a VFC $$[M_{\alpha,L}]^{\textnormal{vir}}\in H_{\chi(\mathcal{O}_X)-\chi(\alpha,\alpha)}(M_{\alpha,L},\Q)$$ when $\chi(\alpha,\alpha)$ is even. Working with singular homology instead of chow homology and thus halving the degree of the cycle is the result of taking ``half of the obstruction theory" and the degree $\chi(\mathcal{O}_X)-\chi(\alpha,\alpha)$ was interpreted in Borisov--Joyce \cite{BJ} as the real virtual dimension of $M_{\alpha,L}$. 

These fundamental classes were conjectured to satisfy universal wall-crossing formulae by Gross--Joyce--Tanaka \cite{GJT}. This leads to a wall-crossing conjecture for Joyce--Song stable pairs as in Joyce--Song \cite{JoyceSong} (see Conjecture \ref{conjecture WC}, also \cite[Def. 4.3]{GJT}). We present here the first application by computing virtual fundamental classes of Hilbert schemes of points and the related invariants. Our goal is two-fold:
\begin{enumerate}[label=\roman*)]
\item The main novelty of this work as the first successful application of the wall-crossing machinery is the description of the full topological data of $[\Hilb]^{\vir}$. More explicitly, using the minimal input provided to us by Park \cite[Cor. 0.3]{Huy} (recalled in Theorem \ref{thm:park}), we recover an \underline{explicit} description of all the descendent integrals on $\Hilb$ in Theorem \ref{theoremhilb} under the additional condition $H^1(\mO_X)=0$. 
    \item Out of the above data, we are able to compute closed formulae for the generating series of many new invariants and study their symmetries; all this without the extra assumption $H^1(\mO_X)=0$. Here we distinguish two sub-points:
    \begin{enumerate}
        \item    We reduce existing conjectures in literature to the wall-crossing conjecture. These are the Conjecture \ref{conjecture CK} of Cao--Kool and its K-theoretic refinement, the analog of which was conjectured by Nekrasov for $\CC^4$.  
        \item We obtain a relation between universal generating series for Calabi--Yau fourfolds and for surfaces satisfying $c_1^2 = 0$. A consequence of this is for example the Segre--Verlinde correspondence for Calabi--Yau fourfolds. This relation is not unexpected, as DT$_4$ invariants are meant to be the holomorphic version of Donaldson invariants on real fourfolds, in particular on complex surfaces where they were studied by \cite{AJLOP, MOPhigher}.
    \end{enumerate}
 
\end{enumerate}

In the introduction, we choose to focus more on the particular generating series and their symmetries over the more general result mentioned in the first point. A concise summary of the presently used methods appears in §\cite[§1.2]{bojkoquot} where its applications to Quot schemes is the main point. 

\paragraph{Tautological insertions}
One natural insertion on $\Hilb$, considered for surfaces by Göttsche \cite{Gottschechern} and by Cao--Kool  \cite{CK1}, Park \cite{Huy}, Cao--Qu \cite{CaoQu} and Nekrasov\footnote{In the toric setting.} \cite{NMF} for Calabi--Yau 4-folds is the Chern class $c(L^{[n]})$ of the vector bundle 
\begin{equation}
\label{Ln}
    L^{[n]} = \pi_{2\, *}\big(\mathcal{F}_{n}\otimes \pi_X^*(L)\big)\,,
\end{equation}
where $X\xleftarrow{\pi_X} X\times \Hilb \xrightarrow{\pi_2}\Hilb$ are the projections, $L$ is a line bundle on $X$, and $\mathcal{O}\to \mathcal{F}_n$ is the universal complex on $\Hilb$.

For the generating series of invariants 
\begin{equation} 
\label{tautdef}
I(L;q)=1+\sum_{n>0}I_n(L)q^n =1+\sum_{n>0} \int_{[\Hilb]^{\textnormal{vir}}}c_n(L^{[n]})q^n
\end{equation} Cao--Kool \cite{CK1} conjecture the following:
\begin{conjecture}[Cao--Kool \cite{CK1}]
\label{conjecture CK}
Let $X$ be a projective Calabi--Yau 4-fold and  $L$ a line bundle on $X$ then 
\begin{equation}
\label{taut=mm}
  I(L;q) =M(-q)^{\int_Xc_1(L)\cdot c_3(X)}  
\end{equation}
for \underline{some choice of orientations}. Here $M(q) = \prod_{i=1}^\infty(1-(q)^i)^{-i}$ is the Mac-Mahon function. 
\end{conjecture}

Park \cite[Cor. 0.3]{Huy} proves that if $L=\mathcal{O}(D)$ for a smooth connected divisor $D\subset X$, then this conjecture holds (see Theorem \ref{thm:park}). Using  Conjecture \ref{conjecture WC} we reduce Conjecture \ref{conjecture CK} for any line bundle $L$ to their result. It turns out that the correct orientations of Conjecture \ref{conjecture CK} are obtained by fixing canonical orientations for K-theory classes of $\mathcal{O}_X$ and $\mathcal{O}_x$, where $x\in X$ is a $\CC$-point, and using the compatibility under sums in the sense of Cao--Gross--Joyce \cite[Thm. 1.15]{CGJ} to extend these. 

We call these orientations \textit{point-canonical}. They are precisely the orientations of Park\cite{Huy} for any smooth divisor $D$ such that $D\cdot c_3(X)\neq 0$. In particular, they are induced by the canonical orientations of Behrend--Fantechi \cite{BF} under the virtual pullback.

To work with general invariants that follow, it will be useful to define the \textit{universal transformation} on power-series $f$ with constant term $f(0)=1$:
$$
U\big(f(q)\big) = \prod_{n>0}\prod_{k=1}^nf(-e^{\frac{2\pi i k}{n}}q)^{-n}\,.
$$
Note that this transformation plays an important role in the follow-up paper \cite{bojkoquot} as it appears in wall-crossing for Calabi--Yau fourfolds in general. 
\paragraph{Segre series of all ranks}
\label{segreintro paragraph}
For a surface $S$ and a line bundle $L\to S$ the Segre series
$$
R(S,L;q) = \int_{\textnormal{Hilb}^n(S)}s_{2n}\big(L^{[n]}\big)q^n\,
$$
defined in terms of the degree $2n$ Segre class $s_{2n}(-)$ appeared in Tyurin \cite{Tyurin} in relation to Donaldson invariants. Its precise form was conjectured by Lehn \cite{Lehn} and proved by Marian--Oprea--Pandharipande \cite{MOP1} for K3 surfaces and the general case in \cite{MOP2}. One can replace $L$ in \eqref{Ln} by any class $\alpha$ in the Grothendieck group of algebraic vector bundles $G^0(X)$. These more general invariants were studied by Marian--Oprea--Pandharipande in \cite{MOPhigher} because of their relation to Verlinde numbers and strange duality.  We study analogous questions in our case with a more general relation between the invariants of surfaces and fourfolds given in \eqref{eqTra}.

Let $\vec{\alpha}=(\alpha_1,\ldots, \alpha_M)$ for $\alpha_1,\ldots \alpha_M\in G^0(X)$ and $\vec{t} = (t_1,\ldots,t_M)$. We define the \textit{generalized $DT_4$-Segre series} for Calabi--Yau 4-folds by
\begin{equation}
\label{segre}
R(\vec{\alpha}, \vec{t} ;q) =1+\sum_{n>0}q^n\int_{\big[\Hilb\big]^{\textnormal{vir}}}s_{t_1}\big(\alpha_1^{[n]}\big)\ldots s_{t_M}\big(\alpha_M^{[n]}\big) 
\end{equation}
in terms of the Segre classes.
The corresponding series for virtual fundamental classes of Quot schemes on surfaces were studied by Oprea--Pandharipande \cite{OP1}. When $M=1$, we will write
$$
R(\alpha;q) = 1+\sum_{n>0}q^n\int_{\big[\Hilb\big]^{\textnormal{vir}}}s_n\big(\alpha^{[n]}\big)\,.
$$
To express their generating series, recall that \textit{Fuss-Catalan numbers} are given by
 \begin{equation}
 \label{fusscatalan}
  C_{n,a} = \frac{1}{an+1}{an+1\choose n}
\end{equation}
for any positive integer $a$. They were defined by Fuss \cite{fuss}\footnote{They were rigorously studied in \protect\cite{chu, GKP, sagan, sand, OP1}. } and appeared also in the work of Oprea--Pandharipande \cite{OP1}. 
 We denote their generating series by
\begin{equation}
\label{fusscatalan series}
  \mathscr{B}_a(q) = \sum_{n\geq 0}  C_{n,a}q^n\,.
 \end{equation}
\begin{theorem}
\label{theoremsegreintro}
Let $\alpha,\alpha_1,\ldots,\alpha_M \in G^0(X)$, $a=\textnormal{rk}(\alpha), a_i=\textnormal{rk}(\alpha_i)$, then assuming Conjecture \ref{conjecture WC} we have
$$
R(\vec{\alpha},\vec{t}; q) = U\Big[(1+t_1z)^{c_1(\alpha_1)\cdot c_3(X)}\cdots (1+t_Mz)^{c_1(\alpha_M)\cdot c_3(X)}\Big]\,
$$
for point-canonical orientations. Here $z$ is the unique solution to $$z(1+t_1z)^{a_1}\cdots(1+t_Mz)^{a_M} = q\,.$$ Moreover, in the same setting we have the explicit expressions 
\begin{equation}
\label{segreseriesexplicit}
R(\alpha;q) = 
\left\{\begin{array}{ll}
    \displaystyle  U\big[\mathscr{B}_{a+1}(-q)^{-c_1(\alpha)\cdot c_3(X)}\big]  & \textnormal{for }a\geq 0 \\
    &\\
  \displaystyle    U\big[ \mathscr{B}_{-a}(q)^{c_1(\alpha)\cdot c_3(X)}\big] &\textnormal{for }a< 0 \\
\end{array} 
\right.
 \,. 
  \end{equation}
\end{theorem}

As a Corollary, we obtain
$$
R(L;q) =   U\Big[\frac{(1+\sqrt{1+4q})}{2}\Big]^{c_1(L)\cdot c_3(X)}\,.
$$

\paragraph{Nekrasov genus
}
\label{Nekrasovintro paragraph}
K-theoretic invariants for Calabi--Yau 3-folds using twisted virtual structure sheaves were introduced by Nekrasov--Okounkov \cite{NO} to study the relation between DT$_3$ invariants and curve-counting in Calabi--Yau five-folds. They were defined in terms of twisting the usual virtual structure sheaf of $\mO^{\vir}$ constructed by Fantechi--Göttsche \cite{GF} by a square root of the virtual canonical bundle given at each sheaf $E$ in the moduli space by $\det\big(\Ext^\bullet(E,E)\big)$. Extending their ideas to Calabi--Yau fourfolds, Oh--Thomas \cite{OT} define twisted virtual structure sheaves $\hat{\mathcal{O}}^{\textnormal{vir}}$ on $M_{\alpha,L}$ which no longer originate naturally from some untwisted version.

For any $A\in G^0(\Hilb)$, one can compute the \textit{twisted virtual Euler characteristic}
\begin{equation}
\label{chivir}
    \hat{\chi}^{\textnormal{vir}}\big(\Hilb,A\big) = \chi\big(\Hilb,\hat{\mathcal{O}}^{\textnormal{vir}}\otimes A\big) \,.
\end{equation}
Set
$$\mathscr{N}_y\big(\alpha^{[n]}\big)=\Lambda^\bullet_{y^{-1}}\big(\alpha^{[n]}\big)\otimes \textnormal{det}^{-\frac{1}{2}}\big(\alpha^{[n]}\cdot y^{-1}\big)\in G^0(M_{\alpha,L},\Q)(\!(y^{-\frac{1}{2}})\!) \,,$$
where the square-root exists because we are working over rational coefficients. One should think of $y$ as the equivariant parameter for the trivial $\CC^*$ action on $\mathcal{O}_X$. For any $\alpha_1,\ldots,\alpha_M\in G^0(X)$, we define
\begin{equation}
\label{intronekrasov}
    K(\vec{\alpha},\vec{y};q) = 1+\sum_{n>0}\hat{\chi}^{\textnormal{vir}}\big(\Hilb, \mathscr{N}_{y_1}(\alpha^{[n]}_1)\otimes\cdots\otimes \mathscr{N}_{y_M}(\alpha^{[n]}_M)\big)\,.
\end{equation}

Segre series from §\ref{segreintro paragraph} can be obtained as a \textit{classical limit} of these invariants (see Proposition \ref{proplimit}). We again use $K(\alpha,y;q)$ to denote the $N=1$ case. 
 
 The case $\alpha=L$ for a line bundle $L$  was introduced by Nekrasov \cite[§4.2.4]{NMF}, Nekrasov--Piazzalunga \cite{NP} and were further studied by Cao--Kool--Monavari \cite{CKM} in relation to the equivariant DT/PT correspondence. An equivariant version of these invariants can be defined for any toric Calabi--Yau 4-fold using localization of Oh--Thomas \cite[§7]{OT} and the action of the 3-dimensional torus $T =\{t\in (\CC^*)^4:t_1t_2t_3t_4=1\}$.  Nekrasov \cite{NMF} conjectured the following formula for $\CC^4$:
\begin{equation}
\label{nekrasov}
  K_{\textnormal{Nek}}(L_y,t; q) =\textnormal{Exp}\bigg[\chi\Big(\CC^4,q\frac{(T\CC^4-T^*\CC^4)(L_y^{\frac{1}{2}}-L_{y}^{-\frac{1}{2}}}{(1-qL_y^{\frac{1}{2}})(1-qL_y^{-\frac{1}{2}})}\Big)\bigg]\,,
\end{equation}
where $\chi$ denotes the equivariant Euler characteristic on $\CC^4$, $L_y$ is the line bundle on $\CC^4$ with weight $y$ of the $\CC^*$ action and $\textnormal{Exp}[f(y,q)] = \textnormal{exp}\Big[\sum_{n>0}\frac{f(y^n,q^n)}{n}\Big]$ is the plethystic exponential of a power-series $f$. For $X = \textnormal{Tot}_{\mathbb{P}^1}(\mathcal{O}(-1)\oplus \mathcal{O}(-1)\oplus \mathcal{O})$ a similar conjecture was given by Cao--Kool--Monavari \cite[Conjecture 0.16]{CKM}. 

Replacing  $\CC^4$ with a compact $X$ in \eqref{nekrasov}, we obtain a prediction for compact Calabi--Yau 4-folds:
\begin{equation}
\label{prediction}
     K(L,y; q) =\textnormal{Exp}\bigg[\chi\Big(X,q\frac{(TX-T^*X)(L^{\frac{1}{2}}y^{\frac{1}{2}}-L^{-\frac{1}{2}}y^{-\frac{1}{2}}}{(1-qL^{\frac{1}{2}}y^{\frac{1}{2}})(1-qL^{-\frac{1}{2}}y^{-\frac{1}{2}})}\Big)\bigg]\,.
\end{equation}
We thank Noah Arbesfeld for pointing out this direct relation to our previous version of the formula.

\begin{theorem}
\label{theorem nekrasov intro}
If Conjecture \ref{conjecture WC} holds, then for all $\alpha_1,\ldots,\alpha_M$ with $a_i=\textnormal{rk}(\alpha_i)$,  $\sum_{i}a_i=2b+1$ and point-canonical orientations, we have
\begin{align*}
K(\vec{\alpha},\vec{y};q)&=\prod_{i=1}^NU\bigg[\frac{(y_i-1)^2u}{(y_i-u)^2}\bigg]^{\frac{1}{2}c_1(\alpha_i)\cdot c_3(X)}
\end{align*}
where $u$ is the unique solution to 
$$
q=\frac{(u-1)u^b}{\prod_{j=1}^N (y^{\frac{1}{2}}-y^{-\frac{1}{2}}u)^{a_i}}\,.
$$
In particular, we recover \eqref{prediction} when $N=1$, $\alpha_1=\alpha$, $y_1=y$, $a_1=1$.
\end{theorem}
During the writing of this paper, Martijn Kool and Jørgen Rennemo \cite{KRdraft} announced a proof of \eqref{nekrasov}, which in particular proves also the equivariant version of Conjecture \ref{conjecture CK} on $\CC^4$. This gives further motivation to our belief that an argument replacing that of algebraic cobordism in \cite{LP,Li} should exist for DT$_4$ invariants. 

\paragraph{Segre--Verlinde correspondence}
\label{untwistedintro paragraph}
The invariants considered in the previous paragraph are special to 4-folds. Here we address K-theoretic invariants which are the analog of \textit{Verlinde series} for surfaces.

Verlinde series for a smooth projective surface $S$, a line bundle $L$ on it, and an integer $r$ is given by
\begin{equation}
    V(S,L,r;q)=1+\sum_{n>0}\chi\big(\textnormal{Hilb}^n(S),L_{(n)}\otimes E^r\big)q^n\,.
\end{equation}
Here $E=\textnormal{det}\big(\mathcal{O}_S^{[n]}\big)$, and $L_{(n)}$ is constructed by starting from $L^{\boxtimes n}$ on $S^{\times n}$ which descends to a line bundle $L^{(n)}$ on the symmetric product $S^{(n)}$. One then defines $L_{(n)}$ as the pullback of $L^{(n)}$ along the resolution of singularities $S^{[n]}\to S^{(n)}$.  Verlinde series were defined and studied by Ellingsrud--Göttsche--Lehn \cite{EGL}. By Göttsche \cite[Rem. 5.3]{Gottscherefined} it follows that  when  $r=\textnormal{rk}(\alpha)-1$ and $L=\textnormal{det}(\alpha)$ they can be expressed as 
$$
V(S,L,r;q)=V(S,\alpha;q)= 1+\sum_{n>0}\chi\Big(\textnormal{Hilb}^n(S),\textnormal{det}\big(\alpha^{[n]}\big)\Big)q^n\,.
$$
Their virtual analogs for Quot schemes were studied in \cite{GKvirtual} and \cite{AJLOP}.

Motivated by the above two equivalent definitions of $V(S,\alpha;q)$ together with the relation to higher rank Nekrasov genus (see Remark \ref{Verlinde-Nekrasov remark}), we define \textit{square root DT$_4$ Verlinde series}  
$$
V^{\frac{1}{2}}(\alpha;q) =1+\sum_{n>0}\hat{\chi}^{\textnormal{vir}}\Big(\Hilb, \textnormal{det}^{\frac{1}{2}}\big(L^{[n]}_\alpha\big)\otimes E^{r+1} \Big)q^n\,,
$$
where $L_\alpha =\textnormal{det}(\alpha)$, $r=\textnormal{rk}(\alpha)-1$, and $E=\textnormal{det}(\mathcal{O}_X^{[n]})$. Simultaneously, we define the \textit{DT$_4$ Verlinde series} for Calabi--Yau 4-folds:
$$
V(\alpha;q)=1+\sum_{n>0}\hat{\chi}^{\textnormal{vir}}\Big(\Hilb, \textnormal{det}\big(\alpha^{[n]}\big)\otimes E^{\frac{1}{2}}\Big)q^n\,.
$$
Assuming the wall-crossing conjecture, these are related as their name would suggest:  $$V(\alpha;q) = \big(V^{\frac{1}{2}}(\alpha;q)\big)^2\,.$$
This statement is obtained as a consequence of direct computation of both series, so it would be interesting to give an independent reason for it. In \cite[§5.2]{BH}, we study the equivariant version of the above series in the case when $X=\CC^4$. We do so using the description of $\textnormal{Hilb}^n(\CC^4)$ as a vanishing locus of an isotropic section of a vector bundle on a moduli scheme of representations of a framed four-loop quiver. Using this geometric picture that was originally presented by Kool--Rennemo in \cite{KRdraft}, we show that the above factor $E^{\frac{1}{2}}$ is the most natural correction needed to obtain integer invariants. 
\begin{remark}
   The observations made in §\ref{sect verlinde} and \cite[§5.2]{BH} suggested that for Hilbert schemes of points the \textit{untwisted virtual structure sheaf}
   $$
   \mO^{\uvir} = \hat{\mO}^{\vir}\otimes E^{\frac{1}{2}}
   $$
   is a particularly useful modification of $\hat{\mO}^{\vir}$ that guarantees integer invariants. 
\end{remark}

One of the most notable properties of the Verlinde and Segre series on surfaces motivated by strange duality was proposed by Johnson \cite{Johnson}. An explicit formulation was given by Marian--Oprea--Pandharipande \cite{MOP2} as a change of variables $z=f(q)$, $w=g(q)$ giving 
$$
   V(S,\alpha;z) = R(S,-\alpha;w)\,. 
$$
For Quot schemes of $n$ points $\textnormal{Quot}_{S}(\mathbb{C}^N,n)$ on surfaces, it was observed by Arbesfeld--Johnson--Lim--Oprea--Pandharipande \cite{AJLOP} that this correspondence takes the simple form 
$$
V(\alpha;q)=R(\alpha;(-1)^Nq)\,.
$$
We recover an analogous formula for fourfolds.
\begin{theorem}
\label{verlindeintro theorem}
If Conjecture \ref{conjecture WC} holds, then for any $\alpha\in G^0(X)$, we have for any choices of orientations
$$
\quad V(\alpha;q) = R(\alpha;-q)\,.
$$
\end{theorem}
We obtain this result by computation, which leads to us asking the following.
\begin{question}
Is there a geometric interpretation of the Segre--Verlinde correspondence for Calabi--Yau 4-folds similar to the one of \cite{Johnson}?
\end{question}

\paragraph{4D-2D correspondence}

We discuss now a stronger correspondence between invariants on surfaces and Calabi--Yau 4-folds. To a reader familiar with the results of \cite{OP1} and \cite{AJLOP}, this will be a generalization of some of the above results including the Segre--Verlinde correspondence.

Let $S$ be a surface, then $\textnormal{Hilb}^n(S)$ carries a virtual fundamental class $\big[\textnormal{Hilb}^n(S)\big]^{\textnormal{vir}}=[\textnormal{Hilb}^n(S)]\cap c_n\Big(\big(K_S^{[n]}\big)^\vee\Big)$ constructed in \cite[Lem. 1]{MOP1}. Let $f,h$ be multiplicative genera (see §\ref{multigen paragraph}). If $S$ satisfies $c_1(S)^2 = 0$, then by \cite[§2.2]{AJLOP}, there is a universal series $A(q)$ depending on $f,h$ and $a=\textnormal{rk}(\beta)$, such that
$$
 1+\sum_{n>0}q^n\int_{\big[\textnormal{Hilb}^n(S)\big]^{\textnormal{vir}}}f(\beta^{[n]})h\big(T^{\textnormal{vir}}_{\textnormal{Hilb}^n(S)}\big) =A^{c_1(\beta)\cdot c_1(S)}\,,
$$
where $T^{\textnormal{vir}}_Y$ denotes the virtual tangent bundle of $Y$ dual to a given obstruction theory and $\beta\in G^0(S)$.
Assuming Conjecture \ref{conjecture WC}, we show that for point-canonical orientations, $\alpha \in G^0(X)$ with $\textnormal{rk}(\alpha)=a$ and $h(z) = g(z)g(-z)$ we get
\begin{equation}
\label{eqTra}
 1+\sum_{n>0}q^n\int_{\big[\Hilb\big]^{\textnormal{vir}}}f(\alpha^{[n]})g\big(T^{\textnormal{vir}}_{\Hilb}\big) = U\big(A(q)^{c_1(\alpha)\cdot c_3(X)}\big)\,.
\end{equation}
We consider a single series $f$ for the purpose of the introduction. For a general statement see Theorem \ref{4d2d1d}.
Taking $g$ to be trivial and $f$ to be the Segre class, we relate Theorem \ref{segreintro paragraph} to \cite[Thm. 3, Thm. 6, Thm. 14]{OP1}. A further analogous statement in $K$-theory is stated in \ref{4d2d1d} leading to a natural question.
\begin{question}
What is the geometric interpretation of the universal transformation $U$ and the correspondence between the universal series for Calabi--Yau 4-folds and surfaces satisfying $c_1(S)^2 = 0$?
\end{question}
A less geometric answer is given in Proposition \ref{mega theorem surfaces}, where we recover the universal generating series of AJLOP \cite{AJLOP} for surfaces satisfying $c_1(S)^2=0$ from an \underline{explicit} expression for its virtual fundamental class obtained using the same wall-crossing methods. This explains the relation between the two generating series as a consequence of the underlying theories being governed by the same wall-crossing formula.

In our future work \cite{bojkoquot}, we pursue this question further by studying (virtual) fundamental classes of Quot schemes $\big[\textnormal{Quot}_Y(F,n)\big]^{\textnormal{vir}}$, where $F\to Y$ is a torsion-free sheaf and $Y$ is a curve, surface or a Calabi--Yau fourfold. All three theories in the three different dimensions can be approached by the same universal wall-crossing formula, so we obtain a unification of the corresponding invariants which is best summarized in \cite[§1.4]{bojkoquot}.

\section*{Acknowledgments} 
 We would like to thank Dominic Joyce for supervising this project and giving multiple helpful ideas. We are grateful to Martijn Kool for recommending to study K-theoretic invariants. We would also like to express our thanks to Noah Arbesfeld, Yalong Cao, Sergej Monavari, Jeongseok Oh, and Rahul Pandharipande. 

 The author was supported by a Clarendon Fund Scholarship at the University of Oxford. 
\section*{Notation}
\begin{small}
	\begin{center}
		\begin{tabular}{p{4cm} p{10cm}}
$(V_*,\ket{0}, T,Y)$& The data of a vertex algebra on its underlying $\ZZ$-graded vector spaces $V_*$. It is given by the vacuum vector $\ket{0}$, the translation operator $T:V_*\to V_{*+2}$,  and the state-field correspondence $Y$.\\
$[z^n]\{f(z)\}$& The $n$'th coefficient of a Laurent series $f$ in $z$.\\
$\epsilon_{\alpha,\beta}$&Some signs forming a group 2-cocycle.\\
$(-)^{\ttop}$& Topological realization functor of Blanc \cite[§3]{Blanc}.\\
$K^0(-)$&Topological complex K-theory of topological spaces and (higher) stacks. \\
$\llbracket \mE\rrbracket$& Topological complex K-theory class associated with a perfect complex $\mE$.\\
$V^*,(E^\bullet)^\vee,\alpha^\vee$& The duals of a vector bundle $V$, a perfect complex $E^\bullet$ and a K-theory class $\alpha$.\\
$\textbf{Ho}(\textbf{Top}), \textbf{Ho}(\textbf{HSt})$& The homotopy categories of the model categories of topological spaces $\textbf{Top}$ and higher stacks $\textbf{HSt}$.\\
$K_0(\mA)$&The Grothendieck group of an abelian/triangulated category $\mA$. \\
$\hat{H}_*(\mC),\check{H}_*(\mC)$& A vertex algebra on the homology of $\mC$ and its associated Lie algebra.\\
$\hat{H}^F_*(\mC),\check{H}^F_*(\mC)$& The $F$-twisted vertex algebra on the homology of $\mC$ and its associated Lie algebra.\\
$\mM_X$& The higher stack of perfect complexes on $X$.\\
$\mC_X$& The classifying space of topological complex K-theory on $X$.\\
$\mathfrak{U}$,$\mathfrak{E}$& The universal K-theory classes on $BU\times \ZZ$ and $X\times \mC_X$, respectively. 
\\
$\theta_{\mC}$&The K-theory class $\theta_{\mathcal{C}} = \pi_{2,3\,*}\big(\pi^*_{1,2}(\mathfrak{E})\cdot\pi^*_{1,3}(\mathfrak{E})^\vee\big)$ on $\mC_X\times \mC_X$ where $\pi_{2,3}$ is the projection from $X\times \mC_X\times \mC_X$ to the last two factors. 
\\
$\Gamma$& The natural map $\mM^{\ttop}_X\to \mC_X$ mapping perfect complexes on $X$ to their K-theory classes.\\
$p$& The K-theory class of a sky-scraper sheaf on a connected $X$.\\
$\tau$&Gieseker stability.\\
\\
$\mA_q$& The abelian category of $\tau$-semistable coherent sheaves with reduced Gieseker polynomial $q$.\\
$\mB_{q,k}$, $\mN_{q,k}$& The exact category of pairs $V\otimes \mO_X(-k)\to E$ for $k\gg 0$ and $E$ in $\mA_q$ together with its moduli stack.\\
$\Theta$& The dual of the Ext-complex on $\mM_X\times \mM_X$ and its restriction to different substacks.\\
$\Theta^{\pa}$& The Ext-complex on $\mN_{q,k}\times \mN_{q,k}$.\\
$\chi$& The Euler pairing $K^0(X)\oplus K^0(X)\to \ZZ$.\\
$(-)^{\pa}$& Denotes the analog of an existing notation for sheaves applied to pairs in $\mB_{q,k}$ (for example the symmetric pairing $\chi^{\pa}$ on objects of $\mB_{q,k}$).\\
$\mP_X, \mP_{\alpha,d}$& The topological space $\mC_X\times BU\times \ZZ$ and its connected component labeled by $(\alpha,d)\in K^0(X)\oplus \ZZ$\\
\end{tabular}
	\end{center}
	\end{small}
\begin{small}
	\begin{center}
		\begin{tabular}{p{4cm} p{10cm}}
  $\mathscr{I}_n,\mF_n$& The universal pair $\mathscr{I}_n = (\mO\to \mF_n)$ on $X\times \Hilb$.\\
  $\Omega$& The map $\mN_{q,k}^{\ttop}\to \mP_X$ used to compare the respective vertex algebras constructed on both spaces.\\
$\theta_{\mP}$, $\tilde{\chi}$, $\tilde{\epsilon}$& The self-dual K-theory class, the symmetric pairing and the group two-cocycle used in the construction of $\widehat{H}_*(\mP_X)$. \\
 $\mathfrak{E}_{\mP}$& The universal K-theory class on $X\times \mP_X$.\\
$-\backslash -$& Slant product (acting from the left) in cohomology and complex K-theory over rational numbers .\\
$e^{\alpha,d}\otimes 1$& The unit element in $H^0(\mP_{\alpha,d})$.\\
$e^{(\alpha,d)}\otimes\mu_{\sigma,i}$& Natural generators of $H^*(\mP_{\alpha,d})$ obtained as slant products.\\
$e^{(\alpha,d)}\otimes u_{\sigma,i}$& Dual generators of $H_*(\mP_{\alpha,d})$ as a polynomial ring.\\
$I_n(L), I(L;q)$&The invariant $\int_{[\Hilb]^{\textnormal{vir}}}c_{n}\big(L^{[n]}\big)$ and its generating series.\\
$M(q)$& The MacMahon function.\\
$\mL^{[-,-]}$& The vector bundle on $\mN_1\times \mN_1$ extending the usual tautological vector bundle $L^{[n]}$ on $\Hilb$.\\
$\hat{H}^L_*(-), \check{H}^L_*(-)$& The $L$-twisted versions of vertex algebras $\hat{H}_*(-)$ and Lie algebras $\check{H}_*(-)$.\\
$\mathscr{H}_n,\mathscr{M}_{np}$& The classes in $\check{H}_*(\mP_X)$ obtained by pushing forward $[\Hilb]^{\vir}_{\mN}$ and $[M_{np}]^{\inv}$.\\
$\mathscr{N}_{np}$& A polynomial in the variables $u_{v,k}$ such that $e^{np,0}\otimes \mathscr{N}_{np}$ is a lift of $\mM_{np}$ to $\hat{H}(\mP_X)$.\\
$\mathscr{G}_f,f$& The multiplicative Hizrebruch-genus determined by the power-series $f$ and the shortened notation for it.\\
$\mathfrak{T}(\alpha)$&The weight-zero K-theory class on $\mP_X\times \mP_X$ extending $\alpha^{[n]}$ from $\Hilb$.\\
$\{-\}$& The operation $f(z)\mapsto f(z)f(-z)$ on the power-series of multiplicative genera. 
\\
$C_{n,a},\mB_a(q)$& The Fuss-Catalan numbers and their generating series.\\
$\mN_y(-)$& Nekrasov genus.\\
$\hat{\mO}^{\vir},\hat{\chi}^{\vir}$& The twisted virtual structure sheaf and the twisted virtual Euler characteristic.\\
$\mO^{\uvir},\chi^{\uvir}$& The untwisted virtual structure sheaf and the untwisted virtual Euler characteristic.\\
$\mB_N, \mN^N_1$& The abelian category and its moduli stack of pairs $V\otimes \CC^N\otimes \mO_S\to F$ for a zero-dimensional sheaf $F$ and a vector space $V$. \\
$\mathscr{Q}_{N,n}$& The pushforward of $[\Quot]^{\vir}_{\mN}$ to $\check{H}_*(\mP_S)$.
\end{tabular}
	\end{center}
	\end{small}
\section{Vertex algebras in algebraic geometry and topology}
Vertex algebras constructed on the homology of moduli stacks appeared in \cite{JoyceWC}. Their main ingredients will be recalled in §\ref{sec:axiomsgVA}, but it is essential to keep in mind that the class of objects parametrized by the moduli stack should carry some additive structure. This is always the case if they form triangulated or abelian categories. In cases, when the homology of the stacks is explicitly described, Joyce's construction is known to produce lattice vertex algebras.

In this section, we give a quick recollection of the definition of vertex algebras and the algebro-geometric vertex algebra construction of Joyce \cite{Joycehall} which allowed Gross--Joyce--Tanaka \cite{GJT} to formulate their wall-crossing conjectures. We then state the specific version of their conjecture that we will apply to compute integrals over $\Hilb$. This formulation is based on the wall-crossing between sheaves and Joyce--Song pairs in \cite[§5.4]{JoyceSong} and it relies on constructing a vertex algebra on the moduli stack $\mN_{q,k}$ of pairs of the from $V\otimes \mO_X(-k)\to F$.

In the second half of the section, we rephrase everything in a more topological language by introducing $\mP_X$ -- the topological space of pairs. Using a natural map $\mN_{q,k}\to \mP_X$, we compare the the vertex algebra structures on both sides. While $H_*(\mN_{q,k})$ is difficult to describe, the vertex algebra on $H_*(\mP_X)$ has an explicit form presented in §\ref{sec:ecplicitVApa}. This will be the main tool for obtaining closed formulae for all descendent integrals in §\ref{VFC section}.

Because wall-crossing takes place in the Lie algebra associated to the vertex algebra as some quotient, we explain how to integrate insertions over classes in such Lie algebras. The appropriate insertions are called \textit{weight zero} and they form a subalgebra of the cohomology of the underlying stack dual to the Lie algebra. Top Chern classes of vector bundles satisfying Definition \ref{E-twisted} are are a particular example of such insertions that behaves well under wall-crossing. For this reason, the invariants \eqref{tautdef} for a line bundle $L$ present a convenient set of data to wall-cross from. 

\subsection{Vertex algebras}
\label{VAinAG}
Let us recall first the definition of vertex algebras focusing on \textit{graded super-lattice vertex algebras} which will be the only explicit example that we will be working with. For background literature, we recommend  \cite{Borcherds, KacVA, BZVA, FLM, LLVA, gebert}, with Borcherds \cite{Borcherds} being the most concise. 

\begin{definition}
A \textit{$\mathbb{Z}$-graded vertex algebra} over a field $\mathbb{Q}$ is a collection of data $(V_*,T,\ket{0}, Y)$, where $V_*$ is a $\ZZ$-graded vector space, the \textit{translation operator} $T: V_*\to V_{*+2}$ is graded linear, \textit{the vacuum vector} $\ket{0}\in V_0$, and the \textit{the state-field correspondence} $Y: V_*\to \textnormal{End}(V_*)\llbracket z,z^{-1}\rrbracket$ is graded linear after setting $\textnormal{deg}(z)=-2$.  For any $u,v\in V_*$, this data is required to satisfy the following axioms:
\begin{enumerate}
    \item \textit{(vacuum)} $T\ket{0} = 0$\,, $Y(|0\rangle, z)=\textnormal{id}$\,, $Y(u,z)|0\rangle\in u+zV_*\llbracket z\rrbracket$\,,
    \item \textit{(translation covariance)} $[T,Y(u,z)] = \frac{d}{dz} Y(u,z)$\,,
    \item \textit{(locality)} there is an $N\gg 0$ such that 
    $$
    (z_1-z_2)^N[Y(u,z_1),Y(v,z_2)] = 0\,,
    $$
    where the supercommutator is defined on $\textnormal{End}(V_*)$ by
    $$
    [A,B] =A\circ B -  (-1)^{\deg(A)\deg(B)}B\circ A\,.
    $$
\end{enumerate}
\renewcommand{\theenumi}{\roman{enumi}}
\end{definition}
By Borcherds \cite{Borcherds}, the graded vector-space $V_{*+2}/T(V_*)$ carries a graded Lie algebra structure determined by
\begin{equation}
\label{lie algebra}
  [\bar{u},\bar{v}]=[z^{-1}]\big\{Y(u,z)v\big\}  
\end{equation}
where we use $[z^n]\{f(z)\}$ for the $n$'th coefficient of the power-series $f(z)$ and will continue to do so. 

Let $A^{\pm}$ be abelian groups and $$\chi^{+}: A^{+}\times A^{+}\to \ZZ\,,\qquad \chi^{-}: A^{-}\times A^{-}\to \ZZ$$ be symmetric, respectively anti-symmetric bi-additive maps. Let us denote $\mathfrak{h}^{\pm} = A^{\pm}\otimes_{\ZZ} \mathbb{Q}$ and fix a basis of $B^{\pm}$ of $\mathfrak{h}^{\pm}$. For $(A^+,\chi^+)$ and a choice of a \textit{group 2-cocycle} $\epsilon: A^+\times A^+\to \ZZ_2$ satisfying
\begin{align*}
\label{signs}
 \epsilon_{\alpha,\beta} &= (-1)^{\chi^+(\alpha,\beta)+\chi^+(\alpha,\alpha)\chi^+(\beta,\beta)}\epsilon_{\beta,\alpha}   \,,\quad  \forall \,\,\alpha,\beta\in A^+\\
     \epsilon_{\alpha,0} &= \epsilon_{0,\alpha}=1\,,\qquad
    \epsilon_{\alpha,\beta}\epsilon_{\alpha+\beta,\gamma}=\epsilon_{\beta,\gamma}\epsilon_{\alpha,\beta+\gamma}\,,
    \numberthis
\end{align*}
(the second line is precisely the condition for the map to be a group 2-cocycle) there is a natural graded vertex algebra on
\begin{equation}
\label{latticeVA}
  \mathbb{Q}[A^+]\otimes_{\mathbb{Q}} \textnormal{Sym}_{\mathbb{Q}}[u_{v,i},v \in B^+, i>0]\cong \mathbb{Q}[A^+]\otimes_{\mathbb{Q}} \textnormal{Sym}_{\Q}\big(\mathfrak{h}^+\otimes t^2\Q[t^2]\big)\,,  
\end{equation}
where $\Q[G]$ denotes the group algebra spanned by the elements of a group $G$, $\Sym_{\Q}[-]$ is a free polynomial algebra in the variables appearing inside of the square brackets, and $\Sym_{\Q}(C)$ is the free polynomial algebra generated by a basis of the vector space $C$.
The isomorphism takes $u_{v,i}\mapsto v\otimes t^{2i}$  and $t$ is of degree $1$. This vertex algebra was called the \textit{generalized lattice vertex algebra} by Gross \cite[Def. 3.5]{gross} (see \cite[§6.4]{LLVA}, \cite[§5.4]{KacVA}).
For given $(A^{-},\chi^-)$, Kac \cite[§3.6]{KacVA} describes a natural $\ZZ$-graded vertex algebra on 
\begin{equation}
\label{super}
  \Lambda_{\mathbb{Q}}[u_{v,i}, v\in B^-,i>0]\cong \Lambda_{\Q}\big(\mathfrak{h}^-\otimes t\Q[t^2]\big)\,, 
\end{equation}
where $\Lambda_{\Q}[-]$ is a free polynomial super-algebra in the odd variables contained inside of the square bracket and $\Lambda(C)$ is again freely generated by a basis of an odd degree vector space $C$. The isomorphism is given by $u_{v,i}\mapsto v\otimes t^{2i-1}$ and $t$ is of degree 1. 

Suppose we have vertex algebras $(V_*,T_V,\ket{0}_V, Y_V)$ and $(W_*,T_W,\ket{0}_W,Y_W)$, then there is a graded Vertex algebra on their tensor product, with state-field correspondence
$$
Y_{V_*\otimes W_*}(v\otimes w,z)(u\otimes t) = (-1)^{\textnormal{deg}(u)\textnormal{deg}(w)}Y_{V_*}(v,z)u\otimes Y_{W_*}(w,z)t\,.
$$
 The \textit{generalized superlattice vertex algebra} for $(A^+\oplus A^-,\chi^\bullet)$ where $\chi^{\bullet} = \chi^+\oplus \chi^-$ is then given by the tensor product of \eqref{latticeVA} and \eqref{super}.

From the definition of the super-lattice vertex algebra $(V_*,T,\ket{0},Y)$ associated to $(A^+\oplus A^-,\chi^\bullet)$ we can easily deduce the fields $Y(u_{v,i},z)$ on the generators of:
\begin{align*}
V_*&\cong \Q[A^+]\otimes_{\Q}\textnormal{SSym}_{\Q}[ u_{v,i}, v\in B,i>0 ]\\
&\cong\mathbb{Q}[A^+]\otimes _{\mathbb{Q}}\textnormal{Sym}_{\Q}\big(A^+\otimes_{\ZZ}t^2\Q[t^2]\big)\otimes_{\Q}\Lambda_{\Q}\big(A^-\otimes_{\ZZ}t\Q[t^2]\big)\,.    
\end{align*}
Let $\alpha\in A^+$ be such that its image in $\mathfrak{h}^+$ can be written as $\sum_{v\in B^+}\alpha_v v$. We use $e^{\alpha}$ to denote the corresponding element in $\Q[A^+]$.  For $K\in \textnormal{SSym}_{\Q}[ u_{v,i}, v\in B,i>0 ]$, any $\beta\in A^+$ and $v\in \mathfrak{h}^{\pm}$ we have
	\begin{align*}
	\label{fields}
	Y(e^{0}\otimes u_{v,1},z)e^{\beta}\otimes K &= e^\beta\otimes \Big\{\sum_{k>0}u_{v,k}\cdot Kz^{k-1} \\
	+ &\sum_{k>0}\sum_{w\in B}k^{(1-\delta_v)}\chi^{\bullet}(v,w)\frac{dK}{d u_{w,k}}z^{-k-1+\delta_v} + \chi^\bullet(v,\beta)Kz^{-1}\Big\}\,,\\
	Y(e^{\alpha}\otimes 1,z)e^\beta\otimes K &= \epsilon_{\alpha,\beta}z^{\chi^+(\alpha,\beta)}e^{\alpha+\beta}\otimes \textnormal{exp}\Big[\sum_{k>0}\sum_{v\in B^+}\frac{\alpha_v}{k}u_{v,k}z^k\Big]\\
	&\textnormal{exp}\Big[-\sum_{k>0}\sum_{v\in B^+}\chi^+(\alpha,v)\frac{d}{du_{v,k}}z^{-k}\Big]K\,,
	\numberthis
	\end{align*}
 where 
 $$
 \delta_v =\begin{cases}
     0&\textnormal{if}\quad v\in \mathfrak{h}^+\\
    1&\textnormal{if}\quad v\in \mathfrak{h}^-
 \end{cases}\,.
 $$
Note that by the reconstruction lemma  \cite[Thm. 5.7.1]{LLVA}, Ben-Zvi \cite[Thm. 4.4.1]{BZVA} and Kac \cite[Thm. 4.5]{KacVA} these formulae are sufficient for determining all fields. 

\subsection{Axioms of vertex algebras on homology of stacks}
\label{sec:axiomsgVA}
Here, we recall Joyce's definition of vertex algebras from \cite{Joycehall} where they are defined for a given moduli stack $\mM$ of some additive category $\mA$. The associated Lie algebras are used to formulate wall-crossing formulae based on \cite{GJT} in the later sections.

For a higher stack $\mathcal{S}$ (see \cite{Toen1, TVHAGDAG, TV2}), we denote by $H_*(\mathcal{S}) = H_*(\mathcal{S}^{\textnormal{top}})$, $H^*(\mathcal{S}) = H^*(\mathcal{S}^{\textnormal{top}})$ its rational Betti (co)homology as in Joyce \cite{Joycehall}, Gross \cite{gross} using the topological realization functor $(-)^{\textnormal{top}}$ from Blanc \cite[§3]{Blanc}. For any topological space $T$, we treat $H_*(T)$ as a direct sum and $H^*(T)$ as a product ranging over all degrees. Following May \cite[§24.1]{May}, define the topological K-theory of $\mathcal{S}$ to be
$$
K^0(\mathcal{S}) =[\mathcal{S}^{\textnormal{top}}, BU\times\ZZ]\,,
$$
where $[X,Y] =\pi_{0}\big(\textnormal{Map}_{C^0}(X,Y) \big)$ after denoting $\textnormal{Map}_{C^0}(-,-)$ the mapping space of continuous maps between two topological spaces. Because the stack $\Perf$ defined in \cite[§1.3.7, Def. 1.3.7.5]{TV2} is the classifying stack of perfect complexes, there is for any perfect complex $\mathcal{E}$ on $\mS$ a unique map $\phi_{\mathcal{E}}: \mathcal{S}\to \textnormal{Perf}_{\CC}$ in $\textbf{Ho}(\textbf{HSt})$. Using Blanc \cite[§4.1]{Blanc} stating that
\begin{equation}
\label{eq:PerfBUZ}
\big(\textnormal{Perf}_{\CC}\big)^{\textnormal{top}}  = BU\times \ZZ
\end{equation}
this leads to
$$
\llbracket \mathcal{E}\rrbracket: \mathcal{S}^{\textnormal{top}}\longrightarrow BU\times \ZZ
$$
in $\textbf{Ho}(\textbf{Top})$. We then have a well-defined map assigning to each perfect complex $\mathcal{E}$ its class $\llbracket \mathcal{E}\rrbracket\in K^0(\mathcal{S})$.  

By \cite[p. 207]{May}, the cohomology of $BU\times\ZZ$ is given by
$$
H^*(BU\times \ZZ) \cong \Q[\ZZ]\otimes_{\Q}\Q\llbracket \beta_1,\beta_2,\ldots \rrbracket\,,
$$
where $\beta_i=\textnormal{ch}_i(\mathfrak{U})$, $\mathfrak{U}$ is the universal K-theory class, and $\Q\llbracket -\rrbracket$ denotes the ring of power-series over $\Q$ in the variables contained in the double bracket. Similarly to \cite{Joycehall}, we define $\textnormal{ch}_i(\mathcal{E}) = \llbracket \mathcal{E}\rrbracket^*(\beta_i)$ and the Chern classes by the Newton identities for symmetric polynomials:
  \begin{equation}
  \label{chern classes}
      \sum_{n\geq 0}c_n(\mathcal{E})q^n=\textnormal{exp}\Big[\sum_{n=1}^\infty(-1)^{n+1}(n-1)!\textnormal{ch}_n(\mathcal{E})q^n \Big]\,.
  \end{equation}
  As $BU\times \ZZ$ is a ring space \cite[§24.1]{May}, the set $K^0(\mathcal{S})$ carries a natural ring structure. Moreover, by similar arguments as in \cite[§24.1]{May}, one also has a map $(-)^\vee: BU\times \ZZ\to BU\times \ZZ$ inducing a map $(-)^\vee: K^0(\mathcal{S})\to K^0(\mathcal{S})$. When $\mathcal{S}$ is replaced with a compact CW-complex $X$, this becomes the standard complex K-theory $K^0(X)$ and $(-)^\vee$ corresponds to taking duals. If $V$ is a vector bundle, then I will write $V^*$ for its dual instead. Later on, we will also use the higher K-theory groups $K^1(X)$.
  
  For a class of topological spaces/higher stacks $\{Z_i\}_{i\in I}$ labelled by an index set $I$ and a  perfect complex/K-theory/cohomology class $\kappa$ on a product  $Z_i\times Z_j$ for $i,j\in I$, we use the notation $(\kappa)_{i,j}$ to denote $\pi_{i,j}^*(\kappa)$, where 
$$\pi_{i,j}: \prod_{k\in I} Z_k \longrightarrow Z_i\times Z_j$$
is a projection to the $i,j$ components.
\begin{definition}[Joyce \cite{Joycehall}]
\label{definition hall}
Let $(\mathcal{A},K(\mathcal{A}),\mathcal{M}, \Phi, \mu, \Theta, \epsilon)$ be data satisfying:
\begin{itemize}
    \item $\mathcal{A}$ is an abelian category or derived category.
    \item Let $K_0(\mathcal{A})$ be the Grothendieck group of $\mA$ and $K_0(\mathcal{A})\to K(\mathcal{A})$ be a map of abelian groups. For each $E\in\textnormal{Ob}(\mathcal{A})$ denote $\llbracket E\rrbracket\in K(\mathcal{A})$ the image of its class.
    \item $\mathcal{M}$ a moduli stack of objects in $\mathcal{A}$ (see \cite[§7]{JoyI}) with an action $\Phi:[*/\GG_m]\times \mathcal{M}\to \mathcal{M}$ corresponding to multiplication by $\lambda\textnormal{id}$ of $\textnormal{Aut}(E)$ for any $E\in \textnormal{Ob}(\mathcal{A})$ and a map $\mu:\mathcal{M}\times \mathcal{M}\to \mathcal{M}$ corresponding to direct sum\footnote{For any nice enough moduli stack, the maps $\Phi, \mu$ are an immediate consequence of $\mA$ being an abelian category.}.
    \item For each $\alpha\in K(\mathcal{A})$, write $\mathcal{M}_\alpha$ for the substack of objects $E$ satisfying $\llbracket E\rrbracket=\alpha$. We require $\mM_{\alpha}$ to be open and closed in $\mM$. 
    \item $\Theta$ is a perfect complex on $\mathcal{M}\times \mathcal{M}$ satisfying $\sigma^*(\Theta)\cong \Theta^\vee[2n]$ for some $n\in\ZZ$ where $\sigma:\mathcal{M}\times \mathcal{M}\to \mathcal{M}\times \mathcal{M}$ interchanges factors and 
    \begin{align*}
    \label{sumcond}
    (\mu\times \textnormal{id}_{\mathcal{M}})^*(\Theta)\cong  \pi_{1,3}^*(\Theta)\oplus \pi_{2,3}^*(\Theta)\,,&\qquad 
    (\textnormal{id}_{\mathcal{M}}\times \mu)^*(\Theta)\cong \pi_{1,2}^*(\Theta)\oplus \pi_{1,3}^*(\Theta)\\
    (\Phi\times \textnormal{id}_{\mathcal{M}})^*(\Theta) \cong \mathcal{V}_{1}\boxtimes \Theta\,,&\qquad
    (\textnormal{id}_{\mathcal{M}}\times \Phi)^*(\Theta) \cong  \mathcal{V}_1 ^*\boxtimes \Theta\,,
    \numberthis
    \end{align*}
    where $\mathcal{V}_1$ is the universal line bundle on $[*/\GG_m]$.
    One also writes $\Theta_{\alpha,\beta}=\Theta|_{\mathcal{M}_\alpha\times \mathcal{M}_\beta}$ and requires that $\textnormal{rk}(\Theta_{\alpha,\beta})$ is constant. Setting 
    $$
    \chi(\alpha,\beta) =  \textnormal{rk}(\Theta_{\alpha,\beta})\,,
    $$
  we obtain a bi-additive symmetric form $\chi: K(\mathcal{A})\times K(\mathcal{A})\to \ZZ$. 
    \item A group  2-cocycle $\epsilon: K(\mathcal{A})\times K(\mathcal{A})\to \ZZ_2$  satisfying \eqref{signs} with respect to $\chi^+=\chi$.
\end{itemize}
Let $\hat{H}_*(\mathcal{M})$ be the homology with shifted grading given by $\hat{H}_n(\mathcal{M}_\alpha)=H_{n-\chi(\alpha,\alpha)}(\mathcal{M}_\alpha)$ for each $\alpha\in K(X)$, then using the above data one constructs a vertex algebra $(\hat{H}_*(\mathcal{M}), \ket{0}, e^{zT}, Y)$ over the $\Q$ vector space $\hat{H}_*(\mathcal{M})$. It is defined by:
\begin{itemize}
    \item $\ket{0} =0_*(*)$, where $0:*\to \mathcal{M}$ is the inclusion of the zero object,
    \item $T(u) = \Phi_*(t\boxtimes u)$ for all $u\in \hat{H}_*(\mathcal{M})$ where $t\in H_2([*/\GG_m]) =H_2(\mathbb{C}\mathbb{P}^{\infty})$ is the generator of homology given by inclusion $\mathbb{C}\mathbb{P}^1\subset\mathbb{C}\mathbb{P}^{\infty}$.
    \item The state to field correspondence $Y$ is given by
\begin{align*}
    Y(u,z)v = \epsilon_{\alpha,\beta} (-1)^{a\chi(\beta,\beta)}z^{\chi(\alpha,\beta)}\mu^{\textnormal{top}}_*(e^{zT}\otimes \textnormal{id})\big((u\boxtimes v)\cap c_{z^{-1}}(\Theta_{\alpha,\beta})\big)\,.
\end{align*}
for all $u\in \hat{H}_a(\mathcal{M}_\alpha)$, $v\in \hat{H}_b(\mathcal{M}_\beta)$.
\end{itemize}
\end{definition}

\begin{remark}
\label{rem:ob}
 The complex $\Theta$ has a geometric interpretation which is useful to keep in mind when constructing these vertex algebras in relation to wall-crossing. For this, we fix $\mM_0$ to be a scheme of perfect complexes on $Y$ with fixed determinants. We assume that it admits an obstruction theory 
		$$
		\EE = \textnormal{RHom}_{\mM_0}(\mE,\mE)_0^\vee[-1] \longrightarrow \LL_{\mM_0}\,,
		$$
		where $\mE$ is the universal perfect complex on $Y\times \mM_0$ and $(-)_0$ denotes the traceless part. If
		$$
		\mu: \mM_0\times \mM_0\longrightarrow \mM_0
		$$
		were an embedding induced by taking direct sums of perfect complexes, then the virtual co-normal bundle would be given by
  \footnote{We work in the next equation with the factors $Y\times \mM_0\times \mM_0$ in this specified order.}
		$$
		\textnormal{RHom}_{\mM_0\times \mM_0}\big(\mE_{1,2},\mE_{1,3}\big)^\vee[-1]\oplus \textnormal{RHom}_{\mM_0\times \mM_0}\big(\mE_{1,3},\mE_{1,2}\big)^\vee[-1]\,.
		$$
		The first summand is equal to $\Theta[-1]$ when $Y=X$ is a Calabi--Yau fourfold. If we pull it back along the diagonal $\Delta: \mM_0\to \mM_0\times \mM_0$, then we recover $\EE$ after taking its traceless part. We can also map $\mM_0$ to the stack $\mM$ where the determinants are not fixed. The vertex algebras on $\mM_0$ and $\mM$ are compatible under the pushforward in homology along this map. This implies that we may construct vertex algebras without worrying about fixed determinants and only specify that we need traceless obstruction theories when construction virtual fundamental classes that we want to study. 
  
  The general set-up for Quot schemes in §\ref{sec:4d2d} is similar, but we need to adapt the class $\Theta$ to the pair obstruction theory when $Y$ is a surface or a curve.
	\end{remark}
The wall-crossing formulae in Joyce \cite{JoyceWC}, Gross--Joyce--Tanaka \cite{GJT} are expressed in terms of a Lie algebra defined by Borcherds \cite{Borcherds}. Let $(\hat{H}_*(\mM), \ket{0},e^{zT}, Y)$ be the vertex algebra from Definition \ref{definition hall} and define $$\check{H}_*(\mM) = \hat{H}_{*+2}(\mM)/T\big(\hat{H}_{*}(\mM)\big)$$ with the projection 
$\Pi_{*+2}:\hat{H}_{*+2}(\mM)\longrightarrow \check{H}_*(\mM).$
By \eqref{lie algebra}, the quotient $\check{H}_*(\mM)$ has a natural Lie algebra structure.

\subsection{Point-canonical orientations}
When working with sheaves on Calabi--Yau fourfolds, the signs $\epsilon_{\alpha,\beta}$ from Definition \ref{definition hall} are meant to relate orientations on different components of $\mM$ under taking direct sums. In this subsection, we state this in a rigorous way while recalling the existence results for said orientations.

Let $\mathcal{M}_X$ be the $\infty$-stack of perfect complexes on the Calabi--Yau fourfold $X$ as in Toën--Vaquié \cite{TVaq}. Let $\mathcal{E}$ be the universal perfect complex on $X\times \mathcal{M}_X$ and $\mathcal{E}\textnormal{xt} = \underline{\textnormal{Hom}}_{\mathcal{M}_X}(\mathcal{E}^\vee\otimes \mathcal{E})$, then there is a natural isomorphism
\begin{equation}
\label{serredualityext}
   \mathcal{E}\textnormal{xt}\cong \mathcal{E}\textnormal{xt}^\vee[-4]. 
\end{equation}

There is an orientation $\ZZ_2$-bundle $O^\omega\to \mathcal{M}_X$ whose sheaf of sections   consists of local trivializations $o:\textnormal{det}(\mathcal{E}xt)\to \mathcal{O}$ satisfying 
$$
o\otimes o=i^\omega: \textnormal{det}(\mathcal{E}xt)\otimes \textnormal{det}(\mathcal{E}xt)\to \mathcal{O}\,,
$$
where $i^\omega:  \textnormal{det}(\mathcal{E}xt)\otimes \textnormal{det}(\mathcal{E}xt)\to \mathcal{O}$ is the result of taking determinants of \eqref{serredualityext}. We follow the sign conventions of \cite{BJ, CGJ, JTU} which are suitable for wall-crossing unlike the ones proposed in \cite{OT}. For any quasi-projective Calabi--Yau 4-fold $X$ this $\ZZ_2$-bundle is trivializable and Oh--Thomas \cite{OT} and Borisov--Joyce \cite{BJ} use this to construct virtual fundamental classes depending on the choice of trivializations which are called orientations. More can be said about the orientations and their compatibility under direct sums. For this recall also that using the notation $\mathcal{C}_X=\textnormal{Map}_{C^0}(X^{\textnormal{an}},BU\times \ZZ)$, there is a natural map $$\Gamma:\mathcal{M}^{\textnormal{top}}_X\to \mathcal{C}_X$$ induced by $X^{\textnormal{an}}\times \mathcal{M}_X^{\textnormal{top}}\to (\textnormal{Perf}_{\CC}^{\textnormal{top}})=BU\times \ZZ$.

\begin{theorem}[{Cao--Gross--Joyce \cite[Thm. 1.15 (b)]{CGJ}}]
\label{theorem CGJ}
Let $X$ be a projective Calabi--Yau 4-fold, then there exists a natural trivializable $\ZZ_2$-bundle $O^{\textnormal{dg}}\to \mathcal{C}_X$ with a natural isomorphism
$$
\Gamma^*(O^{\textnormal{dg}})\cong (O^\omega)^{\textnormal{top}}\,.
$$
Let $\mC_\alpha$ be the connected component labelled by $\alpha\in \pi_0(\mathcal{C}_X) = K^0(X)$, and define $\mathcal{M}_\alpha$ such that $\mathcal{M}^{\textnormal{top}}_\alpha = \Gamma^{-1}(\mathcal{C}_\alpha)$. Further set $O^{\textnormal{dg}}_\alpha = O^{\textnormal{dg}}|_{\mathcal{C}_\alpha}$ and $O^{\omega}_\alpha = O^{\omega}|_{\mathcal{M}_\alpha}$, then there are natural isomorphisms
\begin{equation}
\label{eq:orientphis}
\phi^{\textnormal{dg}}_{\alpha,\beta}: O^{\textnormal{dg}}_{\alpha}\boxtimes O^{\textnormal{dg}}_{\beta}\to \mu^*_{\alpha,\beta}(O^{\textnormal{dg}}_{\alpha+\beta})\,,\qquad \phi^{\omega}_{\alpha,\beta}: O^{\omega}_{\alpha}\boxtimes O^{\omega}_{\beta}\to \mu^*_{\alpha,\beta}(O^{\omega}_{\alpha+\beta})\,,
\end{equation}
Let $o^{\textnormal{dg}}_\alpha$ be a choice of trivialization of $ O^{\textnormal{dg}}_{\alpha}$ for all $\alpha\in K^0(X)$, then define $o^\omega_{\alpha}= \Gamma^*(o^{\textnormal{dg}}_\alpha)$ the trivializations of $O^{\omega}_\alpha$.  There exist unique signs $\epsilon_{\alpha,\beta}$ for all $\alpha,\beta\in K^0(X)$ satisfying \eqref{signs}, such that
$$
\phi^{\omega}_{\alpha,\beta}(o^\omega_\alpha\boxtimes o^\omega_\beta) = \epsilon_{\alpha,\beta}o^\omega_{\alpha+\beta}
$$
if one chooses $o^{\textnormal{dg}}_{0}$ such that $\phi^{\textnormal{dg}}(o^{\textnormal{dg}}_0\boxtimes o^\textnormal{dg}_0)=o^{\textnormal{dg}}_0$. 
\end{theorem}
We will no longer distinguish between $o^{\dg}_{\alpha}$ and $o^\omega_{\alpha}$ and simply write $o_{\alpha}$ for both. Let $p$ denote the K-theory class of a sky-scraper sheaf at some point $x\in X$. We describe how to fix orientations $o_\alpha$ for $\alpha=N\llbracket \mathcal{O}_X\rrbracket + np$. We do so by using the procedure detailed by Joyce--Tanaka--Upmeier \cite[Thm. 2.27]{JTU} which starts with orientations determined for a set of generators $G\subset K^0(X)$ of the K-theory group (see also \cite[Thm. 5.4]{bojko} for compactly supported K-theory specialized to Calabi--Yau fourfolds). Writing a chosen $\alpha\in K^0(X)$ as a sum of the generators, one applies \eqref{eq:orientphis} with $\beta\in G$ multiple times to induce orientations for the sum under the isomorphisms of $\ZZ_2$-bundles. In the process, it is important to fix the order of the generators in which they are added up.

For our purposes, we can work just with the subgroup of $K^0(X)$ generated by $\llbracket \mathcal{O}_X\rrbracket$ and $p$, and we can choose $G=\{\llbracket \mathcal{O}_X\rrbracket,p\}$. We also choose the order by imposing $\llbracket \mO_X\rrbracket < p$.  Let $M_p$ denote the moduli scheme of sheaves of class $p$. There is an isomorphism $M_p = X$ and Cao--Leung \cite[Prop. 7.17]{CL}) showed that 
\begin{equation}
\label{Mpclass}
 [M_p]^{\textnormal{vir}} = \pm \textnormal{Pd}(c_3(X))\in H_2(X)\,,   
\end{equation}
where $\textnormal{Pd}(-)$ denotes the Poincare dual. There is a natural map $m_p: M_p\to \mathcal{M}_X$ leading to $\Gamma\circ m_p^{\textnormal{top}}: M_p^{\textnormal{top}}\to \mathcal{C}_p$. Similarly, the point moduli space $\{\mathcal{O}_X\} =*$ comes with a natural map $i_{\mathcal{O}_X}: \{\mathcal{O}_X\}\to \mathcal{M}_X$ and carries a natural virtual fundamental class $1\in H_0(*)\cong \ZZ$.
\begin{definition}
\label{definition canonicalor}
 We fix orientations $o_p$ and $o_{\llbracket\mO_X \rrbracket}$, such that $(\Gamma\circ m_p^{\textnormal{top}})^*(o_p)$ induces the virtual fundamental class $[M_p]^{\textnormal{vir}} =\textnormal{Pd}(c_3(X))$ and $o_{\llbracket\mathcal{O}_X\rrbracket}$ induces the virtual fundamental class $[\{\mathcal{O}_X\}]^{\textnormal{vir}}=1\in H_0(\textnormal{pt})$. We will denote these choices of orientations by $o^{\textnormal{can}}_p$ and $o^{\textnormal{can}}_{\llbracket\mathcal{O}_X\rrbracket}$, respectively.  By the procedure of \cite[Thm. 2.27]{JTU} or \cite[Thm. 5.4]{bojko} recalled above and using the chosen order $\llbracket \mO_X\rrbracket <p$,  these determine canonical orientations for all $\alpha=N\llbracket\mathcal{O}_X\rrbracket + np$. We will call the resulting orientations \textit{point-canonical}.
\end{definition}

\subsection{Vertex algebras over pairs}
\label{sec:VApa}
 Let us now construct the vertex algebra on the auxiliary category of pairs based on \cite{GJT} and \cite{JoyceSong}. Joyce--Song type wall-crossing will then be formulated using this vertex algebra, and wall-crossing from zero-dimensional sheaves to Hilbert schemes of points will be a special case of it. For the generalization of this construction and some additional reasoning behind it in addition to Remark \ref{rem:ob} see the sequel to this work \cite[§3.2]{bojkoquot}.

\begin{definition}
\label{definition pairs}
Let $X$ be a Calabi--Yau fourfold and $\mathcal{A} =\textnormal{Coh}(X)$. For now, we do not need to assume that $H^2(\mO_X)=0$. Fix an ample divisor $H$ and let $\tau$ denote the Gieseker stability condition with respect to $H$. Then let $\mathcal{A}_{q}$ be the full abelian subcategory of $\mathcal{A}$ with objects the zero sheaf and $\tau$-semistable sheaves with reduced Hilbert polynomial $q$. Define a further smaller exact subcategory $\mA_{q,k}$ consisting of objects $E$ in $\mA_q$ that are \textit{$k$-regular} -- they satisfy $H^{i}\big(E(k-i)\big) = 0$ for $i>0$. From the assumption, it follows that $H^i(E(k))=0$ for $i>0$ and $E(k)$ is generated by global sections for any $(E,V,\phi)$ in $\mB_{q,k}$. 

Let $\mathcal{B}_{q,k}$ be the exact category of triples $(E,V,\phi)$, where $E\in \textnormal{Ob}(\mathcal{A}_{q,k})$, $V\in \textnormal{Vect}_{\CC}$ and $\phi:V\otimes \mathcal{O}_X(-k)\to E$. The morphisms are pairs $(f,g): (E,V,\phi)\to (E',V',\phi')$ where $f:E\to E'$ and $g:V\to V'$ satisfy $\phi'\circ \big(g\otimes \textnormal{id}_{\mathcal{O}_X(-k)}\big) = f\circ \phi$. The moduli stack $\mathcal{N}_{q,k}$ of $\mathcal{B}_{q,k}$ is Artin by \cite[Lem. 13.2]{JoyceSong} combined with the openness of the $k$-regularity condition.

Let
\begin{equation}
\lambda: K_0(\mathcal{A})\longrightarrow K^0(X)
\end{equation}
be induced by the usual comparison map mapping each coherent sheaf to its topological K-theory class. We define $C(\mathcal{A}_{q,k})\subset K_0(\mathcal{A})$ to be the cone of $\llbracket E\rrbracket$ for non-zero $E\in \textnormal{Ob}(\mathcal{A}_{q})$ such that $(E,0,0)\in\textnormal{Ob}\big(\mathcal{B}_{q,k}\big)$. Let $C_0(\mathcal{B}_{q,k})=(C(\mathcal{A}_{q,k})\sqcup \{0\})\times (\N\sqcup\{0\})$, then for all $(\alpha,d)\in C_0(\mathcal{B}_{q,k})$ define $\mathcal{N}^{\alpha,d}_{q,k}$ to be the moduli stack of objects $(E,V,\phi)$ in $\mB_{q,k}$ with $\big(\llbracket E\rrbracket, \dim(V)\big) =(\alpha,d)$. They are constructed as follows:
\begin{itemize}
    \item Let $\mathcal{M}^{\alpha}_{q,k}$ be the moduli stack of objects $E$ in $\mA_{q,k}$ satisfying $\llbracket E\rrbracket=\alpha$.
    If $(\alpha,d)\in C(\mathcal{A}_{q,k})\times \N$ then $\mathcal{N}^{\alpha,d}_{q,k}$ is the total space of a vector bundle $\pi_{\alpha,d}: \pi_{2\,*}\big(\mathcal{F}^{\alpha}_{q,k}\big)\boxtimes \mathcal{V}^*_d\to \mathcal{M}^{\alpha}_{q,k}\times [*/\textnormal{GL}(d,\CC)]$. Here 
    \begin{equation}
    \label{falphan}
    \mathcal{F}_{q,k}^{\alpha}=\pi_1^*\big(\mathcal{O}_X(k)\big)\otimes \mathcal{E}^{\alpha}_{q,k}\,,
    \end{equation}
    where $\mathcal{E}^{\alpha}_{q,k}$ is the universal sheaf over $X\times \mathcal{M}^{\alpha}_{q,k}$ and  we used $\mathcal{V}_d$ to denote the universal vector bundle of rank $d$.
    \item  $\mathcal{N}^{\alpha,0}_{q,k} = \mathcal{M}^{\alpha}_{q,k}$,   $\mathcal{N}^{0,d}_{q,k}=[*/\textnormal{GL}(d,\CC)]$ and $\mathcal{N}^{0,0}_{q,k} = *$.
\end{itemize}
The moduli stack $\mN_{q,k}$ is the disjoint union
$$
\mathcal{N}_{q,k} = \bigsqcup_{(\alpha,d)\in C_0(\mathcal{B}_{q,k})}\mathcal{N}^{\alpha,d}_{q,k}\,.
$$
\end{definition}
We now describe the remaining ingredients needed in Definition \ref{definition hall}.  
\begin{definition}
\label{definition pairVA}
We have a natural action $\Phi_{\mathcal{N}_{q,k}}:[*/\GG_m]\times \mathcal{N}_{q,k}^{\alpha,d}\to  \mathcal{N}^{\alpha,d}_{q,k}$ compatible with the diagonal $[*/\GG_m]$ action on the base $\mathcal{M}_{q,k}^{\alpha}\times [*/\textnormal{GL}(d,\mathbb{C})]$.  We define the map of monoids 
\begin{equation}
\label{komega}
K(\Omega):C_0(\mathcal{B}_{q,k})\stackrel{(\lambda,\textnormal{id})}{\longrightarrow}K^0(X)\oplus \ZZ
\end{equation}
that will later be the K-theory version of the topological map \eqref{omega}. Let $$\Theta_{\alpha,\beta}=\underline{\textnormal{Hom}}_{\mathcal{M}_{q,k}^{\alpha}\times \mathcal{M}^{\beta}_{q,k}}(\mathcal{E}^{\alpha}_{q,k},\mathcal{E}_{q,k}^{\beta})^\vee\,.$$ 
Let $\mathcal{F}_{q,k}^{\alpha}$ be as in \eqref{falphan}. We define  $\Theta^{\textnormal{pa}}_{(\alpha_1,d_1),(\alpha_2,d_2)}$ on $(\mathcal{N}^{\alpha_1,d_1}_{q,k}\times \mathcal{N}^{\alpha_2,d_2}_{q,k})$ for all $(\alpha_i,d_i)\in C_0(\mathcal{N}_{q,k})$ by \footnote{To see that this is the correct choice for $\Theta^{\pa}$, one should use the natural map $\mN_{q,k}\to \mM_X$ mapping pairs to two term-complexes. Pulling back the naturally defined K-theory class $\Theta$ on $\mM_X\times \mM_X$ gives $\Theta^{\pa}$. We replaced $H^*(\mO_X)\otimes \big(\mV_{d_1}\boxtimes \mV^*_{d_2}\big)$ by $\big(\mV_{d_1}\boxtimes \mV^*_{d_2}\big)^{\oplus \chi(\mO_X)}$ in the process.}
   \begin{align*} 
   \label{thetacom}
 \Theta^{\textnormal{pa}}_{(\alpha_1,d_1),(\alpha_2,d_2)} =& (\pi_{\alpha_1,d_1}\times \pi_{\alpha_2,d_2})^*\\&
   \Big\{(\Theta_{\alpha_1,\alpha_2})_{1,3} \oplus \Big((\mV_{d_1}\boxtimes \mV^*_{d_2})^{\oplus \chi(\mO_X)}\Big)_{2,4}\oplus  \Big(\mathcal{V}_{d_1}\boxtimes \pi_{2\,*}(\mathcal{F}_{q,k}^{\alpha_2})^\vee\Big)_{2,3}[1]
   \\&\oplus \Big(\pi_{2\, *}(\mathcal{F}_{q,k}^{\alpha_1}) \boxtimes \mathcal{V}^*_{d_2}\Big)_{1,4} [1]\Big\}\,.
     \numberthis
   \end{align*}
\sloppy The perfect complex $\Theta^{\textnormal{pa}}$ on $\mathcal{N}_{q,k}\times \mathcal{N}_{q,k}$ is defined to have the restriction to $\mathcal{N}^{\alpha_1,d_1}_{q,k}\times \mathcal{N}^{\alpha_2,d_2}_{q,k}$ given by \eqref{thetacom}.
The corresponding bi-additive symmetric form is given by  
\begin{align*}
\label{chip}
&\chi^{\textnormal{pa}}\big((\alpha_1,d_1),(\alpha_2,d_2)\big)=\textnormal{rk}\big(\Theta^{\textnormal{pa}}_{(\alpha_1,d_1),(\alpha_2,d_2)}\big)\\&=\chi(\alpha_1,\alpha_2) + \chi(\mO_X) d_1d_2- d_1\chi(\alpha_2(k)) - d_2\chi\big(\alpha_1(k)\big)\,.
\end{align*}
The signs $\epsilon^{\textnormal{pa}}_{(\alpha_1.d_1),(\alpha_2.d_2)}$ are defined by:
\begin{equation}
\label{signsp}
   \epsilon^{\textnormal{pa}}_{(\alpha_1,d_1),(\alpha_2,d_2)}= \epsilon_{\lambda(\alpha_1-d_1\llbracket\mathcal{O}_X(-k)\rrbracket),\lambda(\alpha_2 -d_2\llbracket\mathcal{O}_X(-k)\rrbracket)}\,.
\end{equation}
where $\epsilon$ is from Theorem \ref{theorem CGJ}.
\end{definition}
The above data does not satisfy the conditions of Definition \ref{definition hall} in the strict sense. This does not pose any problems because these conditions are still satisfied on the level of
 K-theory. This observation is sufficient to conclude in Lemma \ref{pairtopdata} that the same recipe for constructing $$(\hat{H}(\mathcal{N}_{q,k}), \ket{0}, e^{zT},Y)$$ as in Definition \ref{definition hall} leads to a vertex algebra. We leave the proof to §\ref{sec:VAtoppa} after formulating a fully topological perspective of Joyce's vertex algebras.
\subsection{Wall-crossing conjecture for Calabi--Yau fourfolds}
\label{secwallcr}
In this subsection, we state the conjectural wall-crossing formula for Joyce--Song stable pairs based on Gross--Joyce--Tanaka \cite[Def. 4.3]{GJT} and following Joyce--Song \cite[§5.4]{JoyceSong}, Joyce \cite{JoyceWC}. For the abelian category of coherent sheaves, the conjecture has been stated in \cite[Conj. 4.11]{GJT}. Just like in Definition \ref{definition pairs}, we do not require $H^2(\mO_X)=0$ here.\\

Before stating the conjecture, we recall some background from \cite{Joycehall}. For an Artin stack $\mathcal{M}$ as in Definition \ref{definition hall}, we define $\mathcal{M}_{\backslash 0}=\mathcal{M}\backslash\{0\}$ where $0$ is the (point corresponding to the) zero object of $\mA$. Using rigidification as in Abramovich--Olsson--Vistoli \cite{AOV} and Romagny \cite{Romagny}, one defines $\mathcal{M}^{\textnormal{pl}}_{\backslash 0}=\mathcal{M}_{\backslash 0}\fatslash [*/\GG_m]$. Taking the shifted grading on $\check{H}_*(\mathcal{M}^{\textnormal{pl}}_{\backslash 0}) = H_{*+2-\chi(\alpha,\alpha)}(\mathcal{M}^{\textnormal{pl}}_{\backslash 0})$ such that the projection $\Pi^{\textnormal{pl}}:\mathcal{M}_{\backslash 0}\to \mathcal{M}^{\textnormal{pl}}_{\backslash 0}$ induces a map of graded $\Q$-vector spaces, Joyce \cite{Joycehall} proves that it factors as $\hat{H}_{*+2}(\mathcal{M}_{\backslash 0})\xrightarrow{\Pi}\check{H}_{*}(\mathcal{M}_{\backslash 0})\xrightarrow{\check{\Pi}_{*}}\check{H}_{*+2}(\mathcal{M}^{\textnormal{pl}}_{\backslash 0})$, such that $\check{\Pi}_0:\check{H}_0(\mathcal{M}_{\backslash 0})\to \check{H}_0(\mathcal{M}^{\textnormal{pl}}_{\backslash 0})$ is an isomorphism. 
For a stability condition $\tau$ on $\mathcal{A}$ as in \cite[Def. 4.1]{joyceconfigurations} and $0\neq\alpha\in K(\mathcal{A})$, let $M^{\textnormal{st}}_{\alpha}(\tau)$ denote the moduli space of $\tau$-stable objects in class $\alpha$ and $\mathcal{M}^{\textnormal{st}}_\alpha(\tau)\subset\mathcal{M}^{\textnormal{ss}}_\alpha(\tau)$ the finite type stacks of $\tau$-stable and $\tau$-semistable objects. There is a natural open embedding  $i^{\textnormal{st}}_{\alpha}:M^{\textnormal{st}}_{\alpha}(\tau)\hookrightarrow \mathcal{M}^{\textnormal{pl}}_{\backslash 0}$. In particular, if $[M^{\textnormal{st}}_{\alpha}(\tau)]^{\textnormal{vir}}\in H_*\big(M^{\textnormal{st}}_{\alpha}(\tau)\big)$ is defined, then we write $[M^{\textnormal{st}}_{\alpha}(\tau)]^{\vir}_{\mM}=i^{\textnormal{st}}_{\alpha\,*}\big([M^{\textnormal{st}}_{\alpha}(\tau)]^{\textnormal{vir}}\big)\in H_*\big(\mathcal{M}^{\textnormal{pl}}_{\backslash 0}\big)$.

Let now $X$ be a Calabi--Yau fourfold, $\mathcal{A} =\textnormal{Coh}(X)$ and $\tau$ a Gieseker stability, then in the case that $\mathcal{M}^{\textnormal{st}}_\alpha(\tau)=\mathcal{M}^{\textnormal{ss}}_\alpha(\tau)$, the moduli space $M^{\textnormal{st}}_{\alpha}(\tau)$ is a projective scheme. Therefore, Oh--Thomas \cite{OT} and Borisov--Joyce \cite{BJ} construct the virtual fundamental classes $[M^{\textnormal{st}}_\alpha(\tau)]^{\textnormal{vir}}\in H_{\chi(\mO_X)-\chi(\alpha,\alpha)}(M^{\textnormal{st}}_{\alpha}(\tau))$.
Thus we have $[M^{\textnormal{st}}_{\alpha}(\tau)]^{\vir}_{\mM}\in \hat{H}_0(\mathcal{M}^{\textnormal{pl}}_{\backslash 0})$. We 
make a choice of a class in $\big(\check{\Pi}_0\big)^{-1}\big([M^{\textnormal{st}}_\alpha(\tau)]^{\vir}_{\mM}\big)$ and denote it also by $[M^{\textnormal{st}}_\alpha(\tau)]^{\vir}_{\mM}$.

For $\mathcal{A} = \textnormal{Coh}(X)$ we now fix the data from Definition \ref{definition hall}.

\begin{definition}
\label{definition alling}
Define $(\mathcal{A},K(\mathcal{A}), \Phi,\mu,\Theta,\epsilon)$ as follows:
\begin{itemize}
    \item $\lambda: K_0(\mathcal{A})\xrightarrow{\lambda}K^0(X)=:K(\mathcal{A})$, $\Theta$ from Definition \ref{definition pairVA}.
    \item   For $\alpha,\beta\in K(\mathcal{A})$ define $\epsilon_{\alpha,\beta} = \epsilon_{\lambda(\alpha),\lambda(\beta)}$ using the orientations from Theorem \ref{theorem CGJ}.
\end{itemize}
\end{definition}
Moreover, use the fixed orientations above to construct the virtual fundamental classes $[M^{\textnormal{st}}_{\alpha}(\tau)]^\textnormal{vir}\in H_*(M^{\textnormal{st}}_{\alpha}(\tau))$ for all $\alpha,\tau$, such that  $\mathcal{M}^{\textnormal{st}}_\alpha(\tau)=\mathcal{M}^{\textnormal{ss}}_\alpha(\tau)$.
 Let $\tau^{\textnormal{pa}}$ denote the Joyce--Song stability condition on pairs (see Joyce--Song \cite[§5.4]{JoyceSong}), and let $M^{\textnormal{st}}_{(\alpha,1)}(\tau^{\textnormal{pa}})$ be the moduli space of $\tau^{\pa}$-stable pairs in $\mB_{q,k}$ with class $(\alpha,1)$, then \cite{BJ, OT} still give us $\big[M^{\textnormal{st}}_{(\alpha,1)}(\tau^{\textnormal{pa}})\big]^{\textnormal{vir}}\in H_{2-\chi^{\textnormal{pa}}((\alpha,1),(\alpha,1))}(M^{\textnormal{st}}_{(\alpha,1)}(\tau^{\textnormal{pa}}))$ using the usual traceless obstruction theory as in Remark \ref{rem:ob}. The chosen orientations are again used to determine orientations on $M^{\textnormal{st}}_{(\alpha,1)}(\tau^{\textnormal{pa}})$ under the morphism $M^{\textnormal{st}}_{(\alpha,1)}(\tau^{\pa})\to \mathcal{M}_X$ mapping each pair 
 $
 V\otimes \mO_X(-k)\to F
 $ to a two-term complex $[
 V\otimes \mO_X(-k)\to F
 ]$ in the derived category. There is also a natural map 
 $$
M^{\textnormal{st}}_{(\alpha,1)}(\tau^{\pa})\to \mN^{\pl}_{q,k} 
 $$
 for an appropriate choice of $q,k$, and we will write $[M^{\textnormal{st}}_{(\alpha,1)}(\tau^{\textnormal{pa}})]^{\vir}_{\mN}$ for the pushforward of the virtual fundamental class along it. 
\begin{conjecture}[Gross--Joyce--Tanaka {\cite[Def. 4.3]{GJT}}]
\label{conjecture WC}
  Let $\tau$ be a Gieseker stability, then there are unique classes $[\mathcal{M}^{\textnormal{ss}}_{\alpha}(\tau)]^{\inv}\in \check{H}_0(\mathcal{M}_\alpha)$ for all $\alpha\in K(\mathcal{A})$ satisfying:
  \begin{enumerate}[label=\roman*)]
      \item If $\mathcal{M}^{\textnormal{ss}}_\alpha(\tau)=\mathcal{M}^{\textnormal{st}}_{\alpha}(\tau)$ then $[\mathcal{M}^ {\textnormal{ss}}_{\alpha}(\tau)]^{\inv} = [M^{\textnormal{st}}_{\alpha}(\tau)]^{\vir}_{\mM}$\,. 
      \item If $\tilde{\tau}$ is another Gieseker stability condition then these classes satisfy the wall-crossing formula \cite[eq. (4.1)]{GJT} in $\check{H}_*(\mathcal{M})$.
 If $\mathcal{M}^{\textnormal{ss}}_{\beta}(\tau) = \mathcal{M}^ {\textnormal{ss}}_{\beta}(\tilde{\tau})$ then $[\mathcal{M}^ {\textnormal{ss}}_{\beta}(\tau)]^{\inv}=[\mathcal{M}^ {\textnormal{ss}}_{\beta}(\tilde{\tau})]^{\inv}$\,.
  \item We have the formula in $\check{H}_*(\mathcal{N}_{q,k})$:
  \begin{align*}
\label{wcfpairs}
 &[M^{\textnormal{st}}_{(\alpha,1)}(\tau^{\textnormal{pa}})]^{\vir}_{\mN}= \\
 &\sum_{\begin{subarray}a
l\geq 1,\alpha_1,\ldots,\alpha_l\in C(\mathcal{A})\\
\alpha_1+\cdots+\alpha_l = \alpha,\tau(\alpha)=\tau(\alpha_i)
\end{subarray}}\frac{(-1)^l}{l!}\big[\big[\ldots \big[[\mathcal{M}_{(0,1)}]^{\inv},[\mathcal{M}_{\alpha_1}^{\textnormal{ss}}(\tau)]^{\inv}],\ldots ],[\mathcal{M}_{\alpha_l}^{\textnormal{ss}}(\tau)]^{\inv}]\,,
\end{align*}
where $[\mathcal{M}_{(0,1)}]^{\inv}\in \check{H}_0(\mathcal{N}^{0,1}_{q,k})$ is the point class up to a sign determined by the choice of orientation $o_{\llbracket \mathcal{O}_X\rrbracket}$.
  \end{enumerate}
\end{conjecture}

One can define invariants by evaluating on $[M^{\textnormal{ss}}_\alpha(\tau)]^{\inv}$ the cohomology classes from Definition \ref{def weight0} below called \textit{weight 0 insertions}. We will reformulate the above vertex algebras and their associated Lie algebras purely in terms of topology to compute many such invariants for $\alpha = np$, in which case 
$$
M^{\textnormal{st}}_{(\alpha,1)}(\tau^{\textnormal{pa}}) = \Hilb\,.
$$

\label{secwcins}
\subsection{Topological construction of Joyce's vertex algebras}
Because the vertex algebra construction in Definition \ref{definition hall} only uses the underlying topological information of the stack $\mM_{\mA}$ after acting on it with $(-)^{\ttop}$, we extract from the definition the necessary conditions on $(\mM_{\mA})^{\ttop}$. We then extend the definition to any topological space and forget that we originally started from some category $\mA$. This perspective is not entirely new as it can be extracted from a more general formula for generalized complex cohomology theories in Gross \cite[Prop. 5.3.8]{grossdphil}.
\begin{definition}
		\label{DeftopVA}
		Let $(\mathcal{C},\mu, 0)$ be an H-space (see Hatcher \cite{Hatcher})with a $\mathbb{C}\mathbb{P}^{\infty}$ action $\Phi:\mathbb{C}\mathbb{P}^{\infty}\times \mathcal{C}\to \mathcal{C}$ which is an $H$-map with respect to the multiplication $\mu$ an identity $0$. Let $\theta\in K^0(\mathcal{C}\times \mC)  =[\mathcal{C}\times \mC,BU\times\ZZ]$ be a K-theory class satisfying $\sigma^*(\theta) = \theta^\vee$, where $\sigma:\mC\times\mC\to \mC\times \mC$ permutes the two factors, together with
		\begin{align*}
		(\mu\times \textnormal{id}_{\mathcal{C}})^*(\theta)= (\theta)_{1,3}+ (\theta)_{2,3}\,,&\qquad 
		(\textnormal{id}_{\mathcal{C}}\times \mu)^*(\theta)=(\theta)_{1,2}+ (\theta)_{1,3}\,,\\
		(\Phi\times \textnormal{id}_{\mathcal{C}})^*(\theta) = \mV_1\boxtimes \theta\,,&\qquad
		(\textnormal{id}_{\mathcal{C}}\times \Phi)^*(\theta) = \mV_1 ^*\boxtimes \theta\,,
		\end{align*}
		where $\mV_1\to \mathbb{C}\mathbb{P}^{\infty}$ is the universal line bundle.
		
		Let $\pi_0(\mathcal{C})\to K$ be a morphism of commutative monoids. Denote $\mathcal{C}_\alpha$ to be the open and closed subset of $\mathcal{C}$ which is the union of connected components of $\mathcal{C}$ mapped to $\alpha\in K$. We write  $\theta_{\alpha,\beta} = \theta|_{\mathcal{C}_\alpha\times \mathcal{C}_\beta}$, and $\chi^+(\alpha,\beta) = \textnormal{rk}(\theta_{\alpha,\beta})$ must be a symmetric bi-additive form on $K$. Let $\epsilon: K\times K\to  \{-1,1\}$ satisfying
		\begin{align*}
		\epsilon_{\alpha,\beta} &= (-1)^{\chi^+(\alpha,\beta)+\chi^+(\alpha,\alpha)\chi^+(\beta,\beta)}\epsilon_{\beta,\alpha}   \,,\quad  \forall \,\,\alpha,\beta\in A^+\\
		\epsilon_{\alpha,0} &= \epsilon_{0,\alpha}=1\,,\qquad
\epsilon_{\alpha,\beta}\epsilon_{\alpha+\beta,\gamma}=\epsilon_{\beta,\gamma}\epsilon_{\alpha,\beta+\gamma}\,,
		\numberthis
		\end{align*}
		be a group 2-cocycle and $\hat{H}_a(\mathcal{C}_\alpha) = H_{a-\chi(\alpha,\alpha)}(\mathcal{C}_\alpha)$. Then we denote by $$(\hat{H}_*(\mathcal{C}), \ket{0},T, Y)$$ the vertex algebra on the graded $\Q$-vector space $\hat H_*(\mathcal{C}) =\bigoplus_{\alpha\in K}\hat H_*(\mathcal{C}_{\alpha})$ defined from the data $$(\mathcal{C}, K(\mathcal{C}),\Phi,\mu,0,\theta, \epsilon)$$ by 
		\begin{itemize}
			\item $\ket{0} =0_*(*)$, where $*$ denotes the generator of $H_0(\textnormal{pt})$ and $T(u) = \Phi_*(t\boxtimes u)$, where $t$ is the dual of $c_1(\mV_1)$,
			\item the state to field correspondence $Y$ is given by
			\begin{align*}
			Y(u,z)v = \epsilon_{\alpha,\beta} (-1)^{a\chi^+(\beta,\beta)}z^{\chi^+(\alpha,\beta)}\mu_*(e^{zT}\otimes \textnormal{id})\big((u\boxtimes v)\cap c_{z^{-1}}(\theta_{\alpha,\beta})\big)\,,
			\end{align*}
			for all $u\in \hat{H}_a(\mathcal{C}_\alpha)$, $v\in \hat{H}_b(\mathcal{C}_\beta)$. 
		\end{itemize}
		\label{definition haltop}
	\end{definition}
\begin{remark}
We can assign to $(\mathcal{A},K(\mathcal{A}),\mathcal{M}, \Phi, \mu, \Theta,\epsilon)$ from Definition \ref{definition hall} the data
$$(\mathcal{M}^{\textnormal{top}},C_0(\mathcal{A}),\Phi^{\textnormal{top}},\mu^{\textnormal{top}}, 0^{\textnormal{top}}, \theta, \epsilon)$$
from Definition \ref{definition haltop}, where $C_0(\mathcal{A})\subset K(\mA)$ is the cone of all $\llbracket E\rrbracket\in K(\mA)$, $\Phi^{\textnormal{top}}$, $\mu^{\textnormal{top}}$, $0^{\textnormal{top}}$  are maps in $\textbf{Ho}(\textbf{Top})$ and $\theta:=\llbracket \Theta \rrbracket$. The two vertex algebras obtained on $\hat{H}_*(\mathcal{M})$ are clearly the same. 
\end{remark}
\subsection{Insertions}
\label{sectinsertions}
To compute invariants using the homology classes of Conjecture \ref{conjecture WC}, we need to consider elements in the dual of $\check{H}_0(\mathcal{M}_{\mA})$ or $\check{H}_0(\mathcal{N}_{q,k})$ (see Definition \ref{definition pairs}). We do so in the algebraic topological language, as it is more general and closer to the computations that follow. This leads to the notion of weigh 0 insertions which are formally dual to the elements of the Lie algebras. At the end of the section, we also state a topological version of \cite[Def. 2.11]{GJT} which allows us to twist vertex algebra from Definition \ref{DeftopVA} by the top Chern class of a vector bundle. 

\begin{definition}
\label{def weight0}
Let $(\mathcal{C}, K(\mathcal{C}), \Phi,\mu, 0,\theta, \epsilon)$ be the data as in Definition \ref{definition haltop}, then a \textit{weight 0 insertion}  is a cohomology class $ I\in H^*(\mathcal{C})$ satisfying
$$\Phi^*( I) = 1\boxtimes  I\,. $$
\end{definition}

 \begin{lemma}
 \label{lemma weight 0}
Let $ I\in H^*(\mathcal{C})$ be a weight 0 insertion, then $I_m\in H^m(\mathcal{C})$ induces a well defined map $\check{I}_{-2+\chi(\alpha,\alpha)+m}: \check{H}_{-2+\chi(\alpha,\alpha)+m}(\mathcal{C}_{\alpha})\to \Q$ for all $m\geq 0$. 
 \end{lemma}
  \begin{proof}
Suppose that we have $V,V'\in H_m(\mathcal{C})$ such that $V-V'= T(W)$ for $W\in H_{m-2}(\mathcal{C})$. As $T(W) = \Phi_*(t\boxtimes W)$, using the push-pull formula in (co)homology we see
 \begin{align*}
    &T(W)\cap I_m = \Phi_*(t\boxtimes W)\cap I_m = \Phi_*\big(t\boxtimes W\cap\Phi^*(I_m)\big) \\
    &=\Phi_*\big((t\boxtimes W)\cap(1\boxtimes I_m)\big)= \Phi_*\big(t\boxtimes (W\cap I_m)\big) =0\,.
 \end{align*}
Acting with the unit cohomology class $1\in H^0(\mathcal{C})$ on both sides shows that $ I_m(V-V')=0$.
  \end{proof}
   Let $[M]\in \check{H}_m(\mathcal{C})$ and $I$ a weight zero insertion. Then we will use the notation
\begin{equation}
\label{integral}
  \int_{[M]}I= \check{I}_{m}([M])\,.   
\end{equation}

\begin{example}
\label{exampleins}
Suppose that $\mathcal{J}\in K^0(\mathcal{C}\times \mathcal{C})$  satisfies 
$$
(\Phi\times \textnormal{id}_{\mathcal{C}})^*(\mathcal{J})= \mathcal{V}^*_1\boxtimes \mathcal{J}\,,\qquad (\textnormal{id}_{\mathcal{C}}\times \Phi )^*(\mathcal{J})= \mathcal{V}_1\boxtimes \mathcal{J}\,,
$$
then $\mathcal{I} =\Delta^*(\mathcal{J})$ satisfies $\Phi^*(\mathcal{I}) =1\boxtimes \mathcal{I}$. In particular if $p(x_1t,x_2t^2,\ldots)$ is a power series in infinitely many variables then $I=p(\textnormal{ch}_1(\mathcal{I}),\textnormal{ch}_2(\mathcal{I}),\ldots)$ is a weight zero insertion. 
\end{example}
Often times insertions behave well under direct sums. In the algebraic setting the following definition has been stated more generally in \cite[Def. 2.11]{GJT}. 
\begin{definition}
\label{E-twisted}
Let $(\hat{H}_*(\mathcal{C}), \ket{0},e^{zT}, Y)$ be a vertex algebra for the data in Definition \ref{definition haltop}. Let $F\to \mathcal{C}\times \mathcal{C}$ be a vector bundle satisfying
\begin{align*}
\label{vectconditions}
    (\mu\times \textnormal{id}_{\mathcal{C}})^*(F)\cong  \pi_{1,3}^*(F)\oplus \pi_{23}^*(F)\,,&\qquad
    (\textnormal{id}_{\mathcal{C}}\times \mu)^*(F)\cong \pi_{1,2}^*(F)\oplus \pi_{1,3}^*(F)\\
    (\Phi\times \textnormal{id}_{\mathcal{C}})^*(F) \cong V_1^*\boxtimes F\,,&\qquad
    (\textnormal{id}_{\mathcal{C}}\times \Phi)^*(F) \cong  V_1 \boxtimes F\,,
    \numberthis
    \end{align*}
such that $\xi(\alpha,\beta) :=\textnormal{rk}(F|_{\mathcal{C}_\alpha\times \mathcal{C}_\beta})$ is constant for all $\alpha,\beta\in K(\mathcal{C})$. Then define the $F$\textit{-twisted vertex algebra} associated to $(\hat{H}_*(\mathcal{C}), \ket{0},e^{zT}, Y)$ to be the vertex algebra given by Definition \ref{definition haltop} for the data $(\mathcal{C},K(\mathcal{C}), \Phi,\mu, 0,\theta^{F},\epsilon^{\xi})$, where
\begin{align*}
    \theta^{F}&=\theta +\llbracket F^*\rrbracket + \llbracket \sigma^*(F)\rrbracket\,,\\
    \epsilon^{\xi}_{\alpha,\beta}&=(-1)^{\xi(\alpha,\beta)}\epsilon_{\alpha,\beta}\qquad \forall\,\,\alpha,\beta\in K(\mathcal{C})\,.
\end{align*}
We denote this vertex algebra by $\big(\hat{H}^F_*(\mathcal{C}), \ket{0},e^{zT}, Y^{F}\big)$ and the associated Lie algebra by $\big(\check{H}^F_*(\mC), [-,-]^F\big)$ where $\check{H}^F_*(\mC) = \hat{H}^F_{*+2}(\mathcal{C})/T\big(\hat{H}^F_*(\mathcal{C})\big)$.
\end{definition}

One can then conclude by the same arguments as in the proof of \cite[Thm. 2.12]{GJT} (the proof appears in an update to the online article \cite{Joycehall} that has not yet been made publically available) that:
\begin{proposition}
\label{propmorphismtopchern}
In the situation of Definition \ref{E-twisted} let $E =\Delta^*(F)$ and consider the morphism of graded $\Q$-vector spaces
$
(-)\cap c_{\xi}(E): \hat{H}_*(\mathcal{C})\to \hat{H}^F_*(\mathcal{C})\,,
$
such that on $\hat{H}_*(\mathcal{C}_\alpha)$ it acts by $u\mapsto u\cap c_{\xi(\alpha,\alpha)}(E)$. Then it induces a morphism of vertex algebras
\begin{equation}
\label{topcherncap}
    (-)\cap c_{\xi}(E): (\hat{H}_*(\mathcal{C}), \ket{0},e^{zT}, Y)\longrightarrow (\hat{H}^F_*(\mathcal{C}), \ket{0},e^{zT}, Y^{F})\,.
\end{equation}
Moreover, the above induces a well-defined map of Lie algebras 
$$(-)\cap c_{\xi}(E):(\check{H}_*(\mathcal{C}),[-,-])\longrightarrow (\check{H}^F_*(\mathcal{C}),[-,-]^{F})\,.$$
\end{proposition}

\subsection{Vertex algebra of topological pairs}
\label{sec:VAtoppa}
In this subsection, we continue where we left off in §\ref{sec:VApa}. There it was promised, that we will prove that $(\hat{H}_*(\mathcal{N}_{q,k}), \ket{0}, e^{zT},Y)$ is a vertex algebra. We do so by showing that the data
\begin{equation}
\label{pairtopdata}
\big((\mathcal{N}_{q,k})^{\textnormal{top}},C_0(\mathcal{B}_{q,k}),\Phi_{\mathcal{N}_{q,k}}^{\textnormal{top}},\mu_{\mathcal{N}_{q,k}}^{\textnormal{top}}, 0^{\textnormal{top}}, \llbracket \Theta^{\textnormal{pa}}\rrbracket, \epsilon^{\textnormal{pa}}\big)
\end{equation}
from Definition \ref{definition pairs}
satisfy the assumptions of Definition \ref{DeftopVA}.

Later, we also construct the appropriate vertex algebra on $H_*(\mP_X)$ where $\mP_X =\mC\times BU\times \ZZ$, such that the natural map $\Omega:\mN_{q,k}\to \mP_X$ (see \eqref{omega}) induces an isomorphism of vertex algebras.  
\begin{lemma}
\label{lemmapaircon}
The data \eqref{pairtopdata} satisfies the conditions in \ref{definition haltop}. Denote by  the corresponding vertex algebra. 
\end{lemma}
\begin{proof}
To show that $\llbracket\Theta^{\textnormal{pa}}\rrbracket$ satisfies $\sigma^*(\llbracket\Theta^{\textnormal{pa}}\rrbracket)= \llbracket\Theta^{\textnormal{pa}}\rrbracket^\vee$ we note that
\begin{align*}
\sigma^*_{\alpha_1,\alpha_2}(\Theta_{\alpha_2,\alpha_1}) &\cong \Theta_{\alpha_1,\alpha_2}^\vee[-4]\,,\\
\sigma^*_{(\alpha_1,d_1),(\alpha_2,d_2)}\big(\mV_{d_2}\boxtimes \mV^*_{d_1})_{2,4} &= (\mV^*_{d_1}\boxtimes \mV_{d_2})_{2,4} =(\mV_{d_1}\boxtimes \mV^*_{d_2})^*\,,\\
\sigma^*_{(\alpha_1,d_1),(\alpha_2,d_2)}\big((\mV_{d_2}\boxtimes \pi_{2\,*}(\mathcal{F}_{q,k}^{\alpha_1})^\vee)_{2,3}\big) &= (\pi_{2\,*}(\mathcal{F}_{q,k}^{\alpha_1})^\vee\boxtimes \mV_{d_2})_{1,4} =(\pi_{2\,*}(\mathcal{F}_{q,k}^{\alpha_1})\boxtimes \mV_{d_2}^*)^\vee_{1,4} \,.
\end{align*}
The rest of the properties for $\llbracket \Theta^{\textnormal{pa}} \rrbracket$ follow immediately, because $\mV_d$ and $\pi_{2\,*}(\mathcal{F}^{\alpha,d}_{q,k})$ are weight $1$ (see Joyce \cite{Joycehall}) with respect to the $[*/\GG_m]$ action and they are additive under sums.
The signs $\epsilon_{(\alpha_1,d_1),(\alpha_2,d_2)}$ satisfy \eqref{signs} for $\chi^{\textnormal{pa}}$ because the map $\tau: K(\mathcal{A}_{q,k})\times \ZZ\to K(X)$ given by $\tau(\alpha,d)= \lambda(\alpha) -d\llbracket\mathcal{O}_X(-k)\rrbracket$ is a group homomorphism satisfying $\chi\circ(\tau\times\tau) = \chi^{\textnormal{pa}}$. 
\end{proof}
To be able to do computations, we will map the vertex algebra $\hat{H}_*(\mN_{q,k})$ to one that can be explicitly described.  We use the map $\Sigma:\mathcal{N}_{q,k}\to \mathcal{M}_X\times \textnormal{Perf}_{\CC}$ where $\mathcal{M}_X$ is the higher moduli stack of perfect complexes on $X$. For each $(\alpha,d)\in C_0(\mathcal{N}_{q,k})$, the restriction $\Sigma_{(\alpha,d)}=\Sigma|_{\mathcal{N}^{\alpha,d}_{q,k}}$ can be expressed as 
$
\Sigma_{(\alpha,d)} = (\iota^{\alpha}_{q,k}\times \iota_d)\circ\pi_{\alpha,d}\,,
$
where $\iota^{\alpha}_{q,k}:\mathcal{M}^{\alpha}_{q,k}\to \mathcal{M}_X$  and $\iota_d: [*/\textnormal{GL}(d,\CC)]\to \textnormal{Perf}_{\CC}$ are the natural inclusions. As $\pi_{\alpha,d}: \mathcal{N}^{\alpha,d}_{q,k}\to \mathcal{M}^{\alpha}_{q,k}\times [*/\textnormal{GL}(d,\CC)]$ is an $\A^1$-homotopy equivalence, it induces isomorphisms on $K^0$ and $H_*$, so we do not lose any information. 

While there is an explicit description of $H_*(\mathcal{M}_X)$ (see Gross \cite[Thm. 4.7]{gross}) in terms of the semi-topological K-theory groups $K^*_{\textnormal{sst}}(X)$ of Friedlander--Walker \cite{FriWal}, we will not use it because these can be complicated for general Calabi--Yau fourfolds. Instead, we transfer the problem into a completely topological one by using the natural map $\Gamma: \mM_X^{\ttop}\to \mC_X$ (defined above Theorem \ref{theorem CGJ}) and by introducing 
\begin{equation}
\label{omega}
    \Omega:=( \Gamma\times\textnormal{id})\circ \Sigma^{\textnormal{top}}:\mathcal{N}_{q,k}^{\textnormal{top}}\longrightarrow \mathcal{M}_X^{\textnormal{top}}\times BU\times\ZZ\longrightarrow \mathcal{C}_X \times BU\times\ZZ
\end{equation}
which relies on \eqref{eq:PerfBUZ}.
This will induce a morphism of vertex algebras when using the correct data on 
 $$\mathcal{P}_X:=\mathcal{C}_X\times BU\times \ZZ\,.$$ 
 Denote by $\mathfrak{U}$ and $\mathfrak{E}$ the universal K-theory classes on $BU\times \ZZ$, respectively $X\times \mathcal{C}_X$. We will also use the notation $\mathfrak{F}_k = \pi^*_{1}(\llbracket\mathcal{O}_X(k)\rrbracket)\cdot\mathfrak{E}$.
\begin{definition}
\label{definitionpairtop}
 Define $\theta_{\mathcal{P}}\in K^0(\mathcal{P}_X\times \mathcal{P}_X)$ by 
$$\theta_{\mathcal{P}}=(\theta_{\mathcal{C}})_{1,3}+\chi(\mathfrak{U}\boxtimes \mathfrak{U}^\vee)_{2,4} - \big(\mathfrak{U}\boxtimes \pi_{2\,*}(\mathfrak{F}_k)^\vee\big)_{2,3} - \big(\pi_{2\,*}(\mathfrak{F}_k)\boxtimes \mathfrak{U}^{\vee}\big)_{1,4}\,,$$
where 
$\theta_{\mathcal{C}} = \pi_{2,3\,*}\big(\pi^*_{1,2}(\mathfrak{E})\cdot\pi^*_{1,3}(\mathfrak{E})^\vee\big)$.

\setcounter{footnote}{0}
 Let $\Phi_{\mathcal{P}}$ be given by the diagonal action\footnote{Using the left-multiplication on $U(m)$ by $U(1)$ we get the action of $\mathbb{C}\mathbb{P}^{\infty}$ on $BU(m)$. Taking a union over all $n$ we get a $\mathbb{C}\mathbb{P}^{\infty}$ action on $\bigsqcup_m BU(m)$. As $\bigsqcup BU(m)\to BU\times\ZZ$ is a homotopy theoretic group completion, using \protect\cite[Prop. 1.2]{CCMT} we can extend to an action on $BU\times \ZZ$} on $\mathcal{C}_X\times (BU\times\ZZ)$. We use the natural H-space structure $(\mathcal{P}_X,\mu,0)$. Choosing $K(\mathcal{P}_X) = K^0(X)\oplus \ZZ$ we set for all $(\alpha_i,d_i)\in K^0(X)\oplus \ZZ$:
\begin{align*}
\label{ptop}
\tilde{\chi}\big((\alpha_1,d_1),(\alpha_2,d_2)\big) &:= \chi(\alpha_1,\alpha_2) + \chi d_1d_2 -d_1\chi\big(\alpha_1(k)\big)-d_2\chi\big(\alpha_2(k)\big) \,,\\
\tilde{\epsilon}_{(\alpha_1,d_1),(\alpha_2,d_2)} &:= \epsilon_{\alpha_1-d_1\llbracket\mathcal{O}_X(-k)\rrbracket,\alpha_2-d_2\llbracket\mathcal{O}_X(-k)\rrbracket}\,,
\numberthis{}
\end{align*}
where $\epsilon$ is from Theorem \ref{theorem CGJ}.
We construct the vertex algebra $(\hat{H}_*(\mathcal{P}_X),\ket{0},e^{zT}, Y)$.
\end{definition}

\begin{proposition}
\label{propmorphism}
\begin{enumerate}[label=\roman*)]
    \item The data of Definition \ref{definitionpairtop} satisfies the assumptions of Definition \ref{DeftopVA} implying that $(\hat{H}_*(\mathcal{P}_X),\ket{0},e^{zT}, Y)$ is a vertex algebra. 
    \item The map $\Omega_*: H_*(\mathcal{N}_{q,k})\to H_*(\mathcal{P}_X)$ induces a morphism of graded vertex algebras $$(\hat{H}_*(\mathcal{N}_{q,k}),\ket{0},e^{zT}, Y)\to (\hat{H}_*(\mathcal{P}_X),\ket{0},e^{zT}, Y)$$ leading to the morphism of Lie algebras
$$
\bar{\Omega}_*:(\check{H}_*(\mathcal{N}_{q,k}),[-,-])\longrightarrow (\check{H}_*(\mathcal{P}_X),[-,-])\,.
$$
\end{enumerate}
\end{proposition}
\begin{proof}
1) Using arguments from the proof of Lemma \ref{lemmapaircon} and Gross \cite[Prop. 5.3.12]{grossdphil}, we reduce it to showing that $\sigma^*(\theta_{\mathcal{C}}) = \theta_{\mathcal{C}}^\vee$. Using the universal space
$$\mV_X = \textnormal{Map}_{C^0}\big(X, \bigsqcup_{m\geq 0}BU(m)\big) $$
of unitary vector bundles on $X$, recall from \cite[Ex. 3.1.8]{grossdphil} that there is a natural homotopy theoretic group completion $\gamma:\mV_X\to\mathcal{C}_X$ . Using universality of homotopy theoretic group-completions proved by Caruso--Cohen--May--Taylor {\cite[Prop. 1.2]{CCMT}}, it is sufficient to show that
\begin{equation}
\label{serredualityk-theory}
  \sigma^*((\gamma\times \gamma)^*(\theta_{\mathcal{C}}))=(\gamma\times \gamma)^*(\theta_{\mathcal{C}})^\vee  .
\end{equation}
Two compact families with maps $K,L\to \mathcal{V}_X$ are equivalent to two families of vector bundles $V_K\to X\times K$, $V_L\to X\times L$ which we can assume to be smooth along $X$ so we choose partial connections $\nabla_{V_K},\nabla_{V_L}$ in the $X$ direction for both of them. The pullback of the class $\gamma^*(\theta_{\mathcal{C}})$ to $K\times L$ is the index of the family of operators $$(\bar{\partial}+\bar{\partial}^*)^{\nabla_{V_L^*\otimes V_K}}: \Gamma^{\infty}(\mathcal{A}^{0,\textnormal{even}}\otimes V_L^*\otimes V_K)\to \Gamma^{\infty}(\mathcal{A}^{0,\textnormal{odd}}\otimes V_L^*\otimes V_K)$$
obtained by rolling up $\bar{\partial}: \mA^{0,i}\to \mA^{0,i+1}$ and twisting by $\nabla_{V^*_L\otimes V_K}$ just as in \cite[Rem. 1.3]{CGJ}. Recall from \cite{AS4} that the index is defined by a more precise version of 
$$
\textnormal{ind}\Big((\bar{\partial}+\bar{\partial}^*)^{\nabla_{V_L^*\otimes V_K}}\Big) = \ker\Big((\bar{\partial}+\bar{\partial}^*)^{\nabla_{V_L^*\otimes V_K}}\Big)-\coker\Big((\bar{\partial}+\bar{\partial}^*)^{\nabla_{V_L^*\otimes V_K}}\Big)
$$
that holds in this exact form only at any given point of the family. The usual Hodge-star combined with a trivializing section of $\mA^{0,4}\cong \ov{\mA^{4,0}}$ induces the complex analytic version of the Serre duality (see \cite[p. 102]{GriHar}) which translates into
$$
\textnormal{ind}\Big((\bar{\partial}+\bar{\partial}^*)^{\nabla_{V_L^*\otimes V_K}}\Big)=\textnormal{ind}\Big((\bar{\partial}+\bar{\partial}^*)^{\nabla_{V_L\otimes V_K^*}}\Big)^\vee\in K^0(K\times L)\,.$$  This is precisely \eqref{serredualityk-theory} when restricted to $K\times L$.

To show that $\Omega_*$ induces a morphism of vertex algebras we note that in $\textbf{Ho}(\textbf{Top})$, $\Omega: (\mathcal{N}_{q,k})^{\textnormal{top}}\to \mathcal{P}_X$ is a morphism of monoids with $\mathbb{C}\mathbb{P}^{\infty}$ action. The pullback $\Omega^*(\theta_{\mathcal{P}})$ is equal to $\llbracket \Theta^{\textnormal{pa}}\rrbracket$ by construction and arguments in the proofs of \cite[Prop. 5.12, Lem. 6.2]{gross}. By considering the action of $\Omega$ on connected components, we get precisely $K(\Omega):C_0(\mathcal{B}_{q,k})\to K^0(X)\oplus \ZZ$ from \eqref{komega} which satisfies
\begin{align*}
    \tilde{\chi}\circ\big(K(\Omega)\times K(\Omega)\big) &= \chi^{\textnormal{pa}} : C_0(\mathcal{B}_{q,k})\times C_0(\mathcal{B}_{q,k})\longrightarrow \ZZ\,,\\
    \tilde{\epsilon}_{K(\Omega)(\alpha_1,d_1), K(\Omega)(\alpha_2,d_2)}&= \epsilon^{\textnormal{pa}}_{(\alpha_1,d_1), (\alpha_2,d_2)}
\end{align*}
for the same choices of $\epsilon_{\alpha,\beta}$ in \eqref{signsp} and \eqref{ptop}. Therefore $\Omega_*:\hat{H}_*(\mathcal{N}_{q,k})\to \hat{H}_*(\mathcal{P}_X)$ is a degree 0 graded morphism compatible with the vertex algebra structure.
\end{proof}

\begin{remark}
We will only use the case when $k$=0, as we will be working with 0-dimensional sheaves only in the following sections.
\end{remark}

\subsection{Explicit description of the vertex algebra of topological pairs}
\label{sec:ecplicitVApa}
We give here an explicit description of the vertex algebra $(\hat{H}_*(\mathcal{P}_X),\ket{0},e^{zT}, Y)$ which will apply some of the work of Joyce \cite{Joycehall} and Gross \cite{gross}. We also set some notations, and conventions and write down some useful identities. All of this will later be used to compute explicitly the classes
$$\bar{\Omega}_*\big([\Hilb]^{\vir}_{\mN}\big)\in \check{H}_*(\mP_X)\,.$$

In the following, $X$ is any connected smooth projective variety of dimension $n$ unless specified.

\begin{definition}
Let us write $(0,1)$ for the generator of $\ZZ$ in $K^0(X)\oplus \ZZ$. Let $\textnormal{ch}: K^*(X,\Q) :=K^*(X)\otimes \Q \to H^*(X)$ be the Chern character isomorphism. For each $0\leq i\leq 2n$ choose a subset $B_i\subset K^0(X,\Q)\oplus K^1(X,\Q)$ such that $\textnormal{ch}(B_i)$ is a basis of $H^{i}(X)$. We take $B_0 = \{\llbracket\mathcal{O}_X\rrbracket\}$ and $B_{2n} = \{p\}$.  Then we write $B = \bigsqcup_i B_i $ and $\mathbb{B} = B\sqcup\{(0,1)\}$. Let $K_*(X, \Q)$ be defined as the dual of $K^*(X,\Q)$, and define $\textnormal{ch}^\vee:K_*(X,\Q)\to H_*(X)$ by the commutativity of the following diagram
\begin{equation*}
    \begin{tikzcd}
      \arrow[d] K_*(X,\Q) \otimes_{\Q} K^*(X, \Q)\arrow[r,"\textnormal{ch}^\vee\otimes \textnormal{ch}"]&H_*(X)\otimes H^*(X)\arrow[d]\\
      \Q\arrow[r,"\textnormal{id}"]&\Q
    \end{tikzcd}\,,
\end{equation*}
 then choose 
$B^\vee\subset K_*(X, \Q)$ such that $B^\vee$ is a dual basis of $B$. We also write $\mathbb{B}^\vee = B^\vee\sqcup \{(0,1)\}$, where $(0,1)$ is the natural generator of $\Q$ in $K_*(X)\oplus \Q$. The dual of $\sigma\in \mathbb{B}$ will be denoted by $\sigma^\vee\in \mathbb{B}^\vee$. Let $\mathfrak{E}_{\mP}$ be the universal object  on $(X\sqcup *)\times \mathcal{P}_X$ and $\mathfrak{E}_{(\alpha,d)}$ its restriction to $(X\sqcup *)\times \mathcal{P}_{(\alpha,d)}$ where $\mathcal{P}_{(\alpha,d)}$ is the connected component of $\mP_X$ labelled by $(\alpha,d)\in K^0(X)\oplus \ZZ$. For each $\sigma\in \mathbb{B}$and $(\alpha,d)\in K^0(X)\oplus \ZZ$, we define the slant product\footnote{There is a direct way to define K-homology and slant products on K-theory without using Chern characters, but there is no benefit of introducing this additional difficulty in the present work.} $\sigma^\vee\backslash \mathfrak{E}_{(\alpha,d)}\in K^*(\mP_X,\Q)$ by
$$
\ch(\sigma^\vee\backslash\mathfrak{E}_{(\alpha,d)})= \ch^\vee(\sigma^\vee)\backslash\ch(\mathfrak{E}_{(\alpha,d)})\,.
$$
We used the slant product on cohomology $-\backslash - :H_j(Y)\times H^i(Y\times Z)\to H^{i-j}(Z)$ (see e.g. \cite[p. 280]{Hatcher})\footnote{We act with the slant product from the left because of our convention to write $X\times \mC_X$ in this order. This is different from \cite{Hatcher} and other standard textbooks precisely by this reordering in the definition of the slant product. The two definitions would differ by a sign corresponding to swapping between the two orders but this does not play a role in the computations in the later sections.}. 
Following \cite[(4.6)]{gross}, we then set the notation 
\begin{equation}
\label{mudef}
   e^{(\alpha,d)}\otimes \mu_{\sigma,i} =\ch_i\big(\sigma^\vee\backslash\mathfrak{E}_{(\alpha,d)}\big)
\end{equation}
where $e^{(\alpha,d)}$ keeps track of the connected component of $\mP_X$.

We have the natural inclusion $\iota_{\mathcal{C},\mathcal{P}}: \mathcal{C}_X\to \mathcal{P}_X$: $x\mapsto (x,1,0)\in \mathcal{C}_X\times  BU\times \ZZ$, so we identify $H_*(\mathcal{C}_X)$ with the image of $(\iota_{\mathcal{C},\mathcal{P}})_*$, which in turn corresponds to $H_*(\mathcal{C}_X)\boxtimes 1\subset H_*(\mathcal{C}_X)\boxtimes H_*(BU\times \ZZ) =H_*(\mathcal{P}_X)$. Similarly, there is a surjective map $\iota_{\mC,\mP}^*: H^*(\mP_X)\to H^*(\mC_X)$ that we describe below Lemma \ref{linear exp}.
\end{definition}

The next proposition follows the arguments in the proof of Gross \cite[Thm. 4.7, Thm. 4.15]{gross} and the remark below it. 

\begin{proposition}
The cohomology ring $H^*(\mathcal{P}_X)$ is generated by $$\{e^{(\alpha,d)}\otimes \mu_{\sigma,i}\}_{(\alpha,d)\in K^0(X)\oplus \ZZ, \sigma\in \mathbb{B},i\geq 1}\,.$$
Moreover, there is a natural isomorphism of rings
\begin{equation}
\label{explcitcoh}
   H^*(\mathcal{P}_X)\cong \Q[K^0(X)\oplus\ZZ]\otimes_{\Q}\textnormal{SSym}_{\Q}\llbracket\mu_{\sigma,i}, \sigma\in \mathbb{B},i>0\rrbracket\,.
\end{equation}
\end{proposition}

From now on, when we compute explicitly with $H^*(\mathcal{P}_X)$, we replace it using this isomorphism.  The dual of \eqref{explcitcoh} gives us an isomorphism 
\begin{equation}
\label{isohomol}
   H_*(\mathcal{P}_X)\cong \Q[K^0(X)\oplus\ZZ]\otimes_{\Q}\textnormal{SSym}_{\Q}[ u_{\sigma,i}, \sigma\in \mathbb{B},i>0] \,, 
\end{equation}
where we use the normalization
\begin{equation}
\label{normalization}
 e^{(\alpha,d)}\otimes \prod_{\begin{subarray}a \sigma\in \mathbb{B}\\
 i\geq 1\end{subarray}}\mu_{\sigma,i}^{m_{\sigma,i}}\Big(e^{(\beta,e)}\otimes \prod_{\begin{subarray}a \tau\in \mathbb{B}\\
 j\geq 1\end{subarray}}u_{\tau,j}^{n_{\tau,j}}\Big)= \begin{cases}
 \prod_{\begin{subarray}a \sigma\in \BB\\
 i\geq 1\end{subarray}}\frac{m_{\sigma,i}!}{(i-1)!^{m_{\sigma,i}}} &\textnormal{ if }\begin{subarray}a
 (\alpha,d)=(\beta,e), m_{\sigma,i}=n_{\sigma,i}\\\forall\, \sigma\in\mathbb{B},i\geq 1
 \end{subarray} \\
 0&\textnormal{otherwise.}
 \end{cases}
\end{equation}
We will be using the following simple result in computations later.
\begin{lemma}
\label{linear exp}
Let $f(x_1,x_2,\ldots)$ be a power-series, then for any set of coefficients $a_{\sigma,j}$ we have
\begin{align*}
&e^{(\alpha,d)}\otimes\textnormal{exp}\Big(\sum_{\begin{subarray}a j>0 \\
\tau\in \mathbb{B}\end{subarray}}a_{\tau,j}\mu_{\tau,j}q^j\Big)\Big(e^{(\beta,e)}\otimes f(u_{\sigma,1},u_{\sigma,2},\ldots)\Big)
\\&= \delta_{\alpha,\beta}\delta_{d,e}f(a_{\sigma,1}q,a_{\sigma,2}q^2,\ldots,\frac{a_{\sigma,k}}{(k-1)!}q^k,\ldots)\,.
\end{align*}
\end{lemma}
\begin{proof}
Notice that acting with $e^{(\alpha,d)}\otimes \prod_{\begin{subarray}
a i\geq 1\\
\sigma\in \mathbb{B}
\end{subarray}}\mu_{\sigma,i}^{m_{\sigma,i}}$ corresponds to acting with $$\delta_{\alpha,\beta}\delta_{d,e}\prod_{\begin{subarray}
a i\geq 1\\
\sigma\in \mathbb{B}
\end{subarray}}\Big(\frac{1}{(i-1)!}\frac{d}{du_{\sigma,i}}\Big)^{m_{\sigma,i}}$$ and then evaluating at $u_{\sigma,i}=0$. As a result, we obtain
\begin{align*}
& e^{(\alpha,d)}\otimes \textnormal{exp}\Big(\sum_{\begin{subarray}a i\geq 1\\ \sigma \in \BB\\\end{subarray}}a_{\sigma,i}\frac{d}{du_{\sigma,i}}q^i\Big)\Big(e^{(\beta,e)}\otimes f(u_{\sigma,1},u_{\sigma,2},\ldots)\Big)|_{u_{\sigma,i}=0}
\\&= \delta_{\alpha,\beta}\delta_{d,e}f(a_{\sigma,1}q,a_{\sigma,2}q^2,\ldots,\frac{a_{\sigma,k}}{(k-1)!}q^k,\ldots)\,,
\end{align*}
by a standard computation. 
\end{proof}
We now describe the maps $\iota_{\mC,\mP}^*: H^*(\mP_X)\to H^*(\mC_X)$ and $(\iota_{\mC,\mP})_*: H_*(\mC_X)\to H_*(\mP_X)$. It follows from the arguments of \cite[Thm. 4.7, Thm. 4.15]{gross} that
\begin{align*}
\label{isohomolC}
  H^*(\mathcal{C}_X)\cong \Q[K^0(X)]\otimes_{\Q}\textnormal{SSym}_{\Q}\llbracket \mu_{v,i}, v\in B,i>0\rrbracket\,,\\
   H_*(\mathcal{C}_X)\cong \Q[K^0(X)]\otimes_{\Q}\textnormal{SSym}_{\Q}[ u_{v,i}, v\in B,i>0]\,,
   \numberthis
\end{align*}
with the same normalization convention on the variables $u$ with respect to $\mu$ as in \eqref{normalization}. We also know that the universal K-theory class $\mathfrak{E}_{\mathcal{P}}$ restricts to $\mathfrak{E}\boxtimes 1$ on $(X\times \mathcal{C}_X)\times BU\times \ZZ$ and   $1\boxtimes \mathfrak{U}$ on $*\times \mathcal{C}_X\times (BU\times \ZZ)$.  
Writing
 $$\mu_{\sigma,i} =\mu_{v,i}\,,\quad u_{\sigma,i} = u_{v,i}\quad \textnormal{or}\quad\mu_{\sigma,i} = \beta_{i} \,,\quad u_{\sigma,i} = b_{i}\,.$$
 when $\sigma =(v,0)$ or $\sigma = (0,1)$ even for the classes of $H^*(\mP_X), H_*(\mP_X)$, we see from the above that $\iota^*_{\mC,\mP}$ acts by setting $\beta_i$ and $e^{(0,1)}$ to zero. This means that for each connected component, $(\iota_{\mC,\mP})_*$ is just the usual inclusion into a polynomial ring with extra variables.  
 
Writing  $\mathfrak{E}_{\alpha} = \sum_{v\in B}v\boxtimes C_{\alpha, v}$ for some $C_{\alpha,v}\in K^0(\mC_X,\Q)$, we have
$$
w^\vee\backslash\mathfrak{E}_{\alpha}=\sum_{v\in B} w^\vee\backslash(v\boxtimes C_{\alpha, v}) =C_{\alpha,w}
$$
where the last step follows from the definition of the slant product as in \cite[p. 280]{Hatcher}. This implies that 
\begin{equation}
\label{chernofE}
    \textnormal{ch}(\mathfrak{E}_{\alpha}) =\ch\Big(\sum_{v\in B}v\boxtimes (v^\vee\backslash\mathfrak{E}_{\alpha})\Big)= \sum_{v\in B}\textnormal{ch}(v)\boxtimes (\sum_{i\geq 0}e^{\alpha}\otimes\mu_{v,i})\,.
\end{equation}
Let $X$ now be a Calabi--Yau fourfold. The following theorem is the topological version of \cite[Thm. 1.1]{gross}, \cite[Thm. 5.19]{Joycehall}\footnote{Joyce's work is currently unfinished so we expect that the label of the theorem in the final version will change.} extending it also to pairs. 
\begin{proposition}
\label{theorem VA}
Let $\Q[K^0(X)\oplus \ZZ]\otimes_{\Q}\textnormal{SSym}_{\Q}[ u_{\sigma,i}, \sigma\in \mathbb{B},i>0]$ be the generalized super-lattice vertex algebra associated to $\big((K^0(X)\oplus\ZZ)\oplus K^1(X),\tilde{\chi}^\bullet\big)$, where $\tilde{\chi}^\bullet = \tilde{\chi}\oplus \chi^-$  for $\tilde{\chi}$ from \eqref{ptop} and 
\begin{equation}
\label{chi-}
\chi^-: K^1(X)\oplus K^1(X)\longrightarrow \ZZ\,,\quad \chi^-(\alpha,\beta) = \int_X\textnormal{ch}(\alpha)^\vee\textnormal{ch}(\beta)\textnormal{Td}(X)\,.    
\end{equation}
Notice that $\tilde{\chi}$ is symmetric while $\chi^-$ is anti-symmetric. The isomorphism \eqref{isohomol} induces a graded isomorphism of vertex algebras
$$
\hat{H}_*(\mathcal{P}_X)\cong \Q[K^0(X)\oplus \ZZ]\otimes_{\Q}\textnormal{SSym}_{\Q}[ u_{\sigma,i}, \sigma\in \mathbb{B},i>0]\,,
$$
if the same signs $\tilde{\epsilon}_{(\alpha,d),(\beta,e)}$ from \eqref{ptop} are used for constructing the vertex algebras on both sides. 
On the right-hand side, the degrees are given by 
\begin{align*}
&\textnormal{deg}\Big(e^{(\alpha,d)}\otimes \prod_{\begin{subarray}a
\sigma\in B_{\textnormal{even}}\sqcup\{(0,1)\}\,,
i>0
\end{subarray}}u_{\sigma,i}^{m_{\sigma,i}}\otimes \prod_{\begin{subarray}a
v \in B_{\textnormal{odd}}\\
j>0
\end{subarray}}u_{v,j}^{m_{v,j}}\Big)\\
&= \sum_{\begin{subarray}a \sigma\in B_{\textnormal{even}}\sqcup\{(0,1)\}\\
i>0\end{subarray}}m_{\sigma,i}2i+\sum_{\begin{subarray}a
v \in B_{\textnormal{odd}}\\
j>0
\end{subarray}}m_{v,j}(2j-1)-\tilde{\chi}\Big((\alpha,d),(\alpha,d)\Big)\,.
\end{align*}
\end{proposition}
\begin{proof}
The proof is nearly identical to \cite[Thm. 1.1]{gross}, \cite[Thm. 5.19]{Joycehall}. We need an explicit expression for $\textnormal{ch}_k(\theta_{\mathcal{P}})$ replacing Proposition \cite[Prop. 5.2]{gross} and a similar result in \cite[Prop. 5.15]{Joycehall} for quivers. This is given in Lemma \ref{lemma chern}. 
\end{proof}

Before we move on to the applications, let us write down some identities we will need later on. From now on, we always fix a point-canonical orientation of Definition \ref{definition canonicalor}, the associated signs of $\epsilon_{\alpha,\beta}$ of Theorem \ref{theorem CGJ} and the corresponding $\tilde{\epsilon}_{(\alpha,d),(\beta,e)}$ from \eqref{ptop}.

\begin{lemma}
\label{lemma identities}
\label{lemmasigns}
Consider the vertex algebra $(\hat{H}_*(\mathcal{P}_X),\ket{0},e^{zT},Y)$, then
\begin{enumerate}[label=\roman*)]
    \item $\textnormal{rk}(\sigma^\vee\backslash\mathfrak{E}_{(\alpha,d)}) = (\alpha,d)(\sigma^\vee)$
    \item Let $v_1,\ldots v_k\in B_{\textnormal{even}}$ and $i_1,\ldots, i_k\geq 1$, then
    \begin{align*}T\big(e^{(\alpha,d)}\otimes u_{v_1,i_1}\cdots u_{v_k,i_k}\big)& = e^{(\alpha,d)}\sum_{\sigma\in B_{\textnormal{even}}\sqcup\{(0,1)\}}(\alpha,d)(\sigma^{\vee})u_{\sigma,1}u_{v_1,i_1}\cdots u_{v_k,i_k}\\
    &+\sum_{l=1}^ki_lu_{v_1,i_1}\cdots u_{v_l,i_{l-1}}u_{v_l,i_l+1}u_{v_l,i_{l+1}}\cdots u_{v_k,i_k}\,,
    \end{align*}
 \item  For all $k,l,M,N\geq 0$ we have 
$\tilde{\epsilon}_{(kp,N),(lp,M)} = (-1)^{Mk}$.
\end{enumerate}
\end{lemma}
\begin{proof}
 i. To see this, we use the functoriality of the slant product:
\begin{align*}
&\textnormal{rk}(\sigma^\vee\backslash\mathfrak{E}_{(\alpha,d)}) = \textnormal{rk}\big(i_{c,b}^*(\sigma^\vee\backslash\mathfrak{E}_{(\alpha,d)})\big) \\&=
\textnormal{rk}\big((\textnormal{id}\times i_{c,b})^*(\sigma^\vee\backslash\mathfrak{E}_{(\alpha,d)})\big) = \sigma^\vee\backslash \big((\alpha,d)\boxtimes 1\big) =(\alpha,d)(\sigma^\vee)\,,
\end{align*}
where $i_{c,b}$ is an inclusion of a point into $\mathcal{P}_{(\alpha,d)}$.
The second statement  is a generalization of \cite[Lem. 5.5]{gross} using i. A similar formula has been shown in \cite{Joycehall} for quivers. 
The last statement follows from \cite[Thm. 5.5]{bojko} together with Definition \ref{definition canonicalor} and \eqref{ptop}.
\end{proof}
We will often avoid specifying the connected component where the (co)homology class sits by simply omitting $e^{(\alpha,d)}$, $e^\alpha$ where it is obvious.

\section{Cao--Kool conjecture}
\label{cao--kool section}
In this section, we work with the the pair moduli stack $\mN_1 := \mN_{1,0}$ for $q=1,k=0$. We first twist the vertex algebra $\hat{H}(\mN_1)$ by $c_n(L^{n})$ for a line bundle $L$ using Definition \ref{E-twisted}. This way, we may use Conjecture \ref{conjecture WC} to wall-cross from the invariants \eqref{DT4defform} to $[\mM^{\ses}_{mp}]^{\inv}$ where $m\in \ZZ$. To be able to do explicit computations, we also modify the vertex algebra $\check{H}(\mP_X)$ so that $\Omega_*$ from Proposition \ref{propmorphism} remains a vertex algebra morphisms after twisting. 

In this paragraph, we assume that $X$ satisfies $H^1(\mO_X)=0$. Because Park \cite[Cor. 0.3]{Huy}, Cao-Qu \cite[Thm. 1.2]{CaoQu} compute \eqref{DT4defform} only when $L=\mO_X(D)$ for a smooth effective divisor $D$, we need to first use Bertini's theorem to obtain a basis of $H^2(X)$ given by very ample divisors that determine line bundles $L$ satisfying Conjecture \ref{conjecture CK}. Due to the linearity of the wall-crossing formulae, we conclude that the orientations used by \cite{Huy, CaoQu} coincide with our point-canonical ones from Definition \ref{definition canonicalor}. Finally, by wall-crossing back we can then extend the result to any line bundle $L$ thus proving Conjecture \ref{conjecture CK}.  

\subsection{L-twisted vertex algebras}
\label{secttauinshilbschem}
From now on, we assume that our Calabi--Yau fourfold $X$ satisfies $H^2(\mO_X)=0$ just like we did in §\ref{sec:intro}. Let $\Hilb$ be the Hilbert scheme of $n$ points on $X$ and $\big[\Hilb\big]^{\textnormal{vir}}\in H_{2n}\big(\Hilb\big)$ the virtual fundamental class defined by Oh-Thomas \cite[Thm. 4.6]{OT} using the orientations in Definition \ref{definition canonicalor}.  We consider the vector bundle $L^{[n]}\to \Hilb$ given by \eqref{Ln}. The real rank of $L^{[n]}$ is $2n$, so Cao--Kool \cite{CK1} define
\begin{equation}
\label{DT4defform}
   I_{n}(L)= \int_{[\Hilb]^{\textnormal{vir}}}c_{n}\big(L^{[n]}\big)\,.  
\end{equation}

The proof of Conjecture \ref{conjecture CK} will be given at the end of subsection \ref{computing zero-dim} in the following form. 
\begin{theorem}
\label{theorem bojko}
Let $X$ be a smooth projective Calabi--Yau fourfold for which Conjecture \ref{conjecture WC} holds, and $L$ a line bundle on $X$. Then 
$$
I(L;q)=1+ \sum_{n=1}^\infty I_{n}(L)q^n =M(-q)^{c_1(L)\cdot c_3(X)}
$$
for the point-canonical orientations of Definition \ref{definition canonicalor}. 
\end{theorem}
For the invariants $I_n(L)$ this is equivalent to
\begin{align*}
\label{coefficient}
&I_n(L)=  \sum_{j\geq 1}d_j(n)I(L)^j\,,\textnormal{ where } \\&
d_j(n)=\sum_{\begin{subarray}a
n_1,\ldots,n_j\\
\sum_i n_i = n
\end{subarray}}\frac{1}{j!}\prod_{i=1}^j\sum_{l|n_i}(-1)^{n_i}\frac{n_i}{l^2},\quad I(L)=c_1(L)\cdot c_3(X) \,.
\numberthis
\end{align*}
Let us interpret this in the language of §\ref{secwallcr} by twisting the vertex algebra of pairs by the top Chern class of a vector bundle extending $L^{[n]}$. Take $\mathcal{A}_q=\mA_1$ to be the abelian category of sheaves with 0-dimensional support. Let $\mathcal{B}_q =\mB_1$ be the corresponding category of pairs from Definition \ref{definition pairs} and $\mN_1$ its moduli stack with $k=0$. We have the identification
$\Hilb = N^{\textnormal{st}}_{(np,1)}(\tau^{\textnormal{pa}})$ after
noting that $q(F)=1$ for any zero-dimensional sheaf $F$. This leads to 
$$
\big[\Hilb\big]^{\vir} = \big[N^{\textnormal{st}}_{(np,1)}(\tau^{\textnormal{pa}})\big]^{\vir}\,.
$$

As $X\times\Hilb$ carries a universal family $$\mathscr{I}_n = (\mO\to \mF_n)$$
where $\mF_n$ is the universal zero-dimensional sheaf, there exists a natural lift of the open embedding $\iota^{\textnormal{pl}}_n: \textnormal{Hilb}^n(X)\to \mathcal{N}^{\textnormal{pl}}_0$
\begin{equation}
\label{inlift}
    \begin{tikzcd}
       &\mN_1\arrow[d,"\Pi^{\textnormal{pl}}"]\\
       \Hilb\arrow[ur,"\iota_{n}"]\arrow[r,"\iota^{\textnormal{pl}}_n"]&\mathcal{N}^{\textnormal{pl}}_0
    \end{tikzcd}\,.
\end{equation}
We use $\iota_n$ to express \eqref{DT4defform} in terms of insertions on $\mN_1$. 
\begin{definition}
Here, we will continue using the notation from Definition \ref{definition pairs} 
and Definition \ref{definition pairVA}. For all $(np,d)\in C_0(\mB_1)$, we will also write $\mathcal{N}_{n,d} = \mathcal{N}_{np,d}$. Then define 
$
\mathcal{L}^{[n_1,n_2]}_{d_1,d_2}\to \mathcal{N}_{n_1,d_1}\times \mathcal{N}_{n_2,d_2}
$
by
$$
\mathcal{L}^{[n_1,n_2]}_{d_1,d_2}= (\pi_{n_1,d_1}\times \pi_{n_2,d_2})^*\Big(\mathcal{V}^*_{d_1}\boxtimes {\pi_2}_*\big(\pi_X^*(L)\otimes \mE_1\big)\Big)_{2,3}\,,
$$
where $\mE_1$ is the universal sheaf on $\mM_1$ the moduli stack of $\mA_1$. It is a vector bundle of rank $d_1n_2$. We define
$$
\mathcal{L}^{[-,-]}|_{\mathcal{N}_{n_1,d_1}\times \mathcal{N}_{n_2,d_2}} = \mathcal{L}^{[n_1,n_2]}_{d_1,d_2}\,.
$$
\end{definition}
Set $\mathcal{L} = \Delta^*\big(\mathcal{L}^{[-,-]}\big)$. From Example \ref{exampleins}, we know that $c_i(\mathcal{L})$ is a weight 0 insertion and 
from definition it follows that $\iota^*_n(\mathcal{L}) = L^{[n]}$. Using this together with  $\big[\Hilb\big]^\textnormal{vir}_{\mN} = \Big(\Pi^{\textnormal{pl}}\circ\iota_n\Big)_*\Big(\big[\Hilb\big]^{\textnormal{vir}}\Big)$, we see that
\begin{equation}
I_n(L) = \int_{[\Hilb]^{\vir}_{\mN}}c_{n}(\mathcal{L})\,.
\end{equation}

The following is clear from the construction. 
\begin{lemma}
\label{lemma Lconditions}
The vector bundle $\mathcal{L}^{[-,-]}\to \mN_1\times \mN_1$ satisfies the conditions of Definition \ref{E-twisted}. Let $(\hat{H}^L_*(\mN_1),\ket{0},e^{zT}, Y^{L})$ be the $\mathcal{L}^{[-,-]}$-twisted vertex algebra, $ \big(\check{H}^L_*(\mN_1), [-,-]^L\big)$ the associated Lie algebra. By Proposition \ref{propmorphismtopchern} we have the morphism $(-)\cap c_{\textnormal{top}}\big(\mathcal{L}\big):\big(\check{H}_*(\mN_1), [-,-]\big)\to \big(\check{H}^L_*(\mN_1), [-,-]^L\big)$\,.
\end{lemma}
 We construct its topological counterpart.

\begin{definition}
\label{def twistedbyl}
Define the data $\big(\mathcal{P}_X, K(\mathcal{P}_X), \Phi_{\mathcal{P}}, \mu_{\mathcal{P}}, 0, \theta^L_{\mathcal{P}}, \tilde{\epsilon}^L\big)$ as follows:

\begin{itemize}
    \item $K(\mathcal{P}_X) = K^0(X)\oplus \ZZ$.
    \item Set $\mathfrak{L} = \pi_{2\,*}\big(\pi_X^*(L)\otimes  \mathfrak{E}\big)\in K^0(\mathcal{C}_X)$.  Then on $\mathcal{P}_X\times \mathcal{P}_X$ we define
$$
 \theta^L_{\mathcal{P}}= (\theta)_{1,3} + \chi(\mathfrak{U}\boxtimes \mathfrak{U}^\vee)_{2,4} - \Big(\mathfrak{U}\boxtimes \big(\pi_{2\,*}(\mathfrak{E})-\mathfrak{L}\big)^\vee\Big)_{2,3} - \Big(\big(\pi_{2\,*}(\mathfrak{E}) - \mathfrak{L}\big)\boxtimes \mathfrak{U}^\vee\Big)_{1,4}\,.
$$
\item The symmetric form $\tilde{\chi}^L: (K^0(X)\oplus \ZZ)\times (K^0(X)\oplus \ZZ)\to \ZZ $ is given by 
\begin{align*}
\label{chiLtil}
&\tilde{\chi}^L\big((\alpha,d),(\beta,e)\big) = \textnormal{rk}\big(\theta^L_{(\alpha,d),(\beta,e)}\big)\\
=&\chi(\alpha,\beta) + \chi de -d\big(\chi(\beta) - \chi(\beta\cdot L)\big)-e\big(\chi(\alpha) - \chi(\alpha\cdot L)\big)\,.
\numberthis
\end{align*}
\item The signs are defined by
\begin{equation}
\label{signstwisted}
\tilde{\epsilon}^L_{(\alpha,d),(\beta,e)} =(-1)^{d\chi(L\cdot\beta)} \tilde{\epsilon}_{(\alpha,d),(\beta,e)}\,,
\end{equation}
in terms of $\tilde{\epsilon}_{\alpha,\beta}$ from \eqref{ptop}. 
\end{itemize}
\sloppy We denote by $(\hat{H}^L_*(\mathcal{P}_X),\ket{0}, e^{zT}, Y^L)$ the vertex algebra associated to this data and $(\check{H}^L_*(\mathcal{P}_X),[-,-]^L)$ the corresponding Lie algebra. 
\end{definition}
We are unable to use Proposition \ref{propmorphismtopchern} directly because $\mathfrak{U}^\vee\boxtimes\mathfrak{L}$ is not a vector bundle. However, one can easily show the following result similar to Proposition \ref{theorem VA} and Proposition \ref{propmorphism}.
\begin{proposition}
\label{propmorphismL}
Let $\Q[K^0(X)\oplus \ZZ]\otimes_{\Q}\textnormal{SSym}_{\Q}[u_{\sigma,i}, \sigma\in \mathbb{B},i>0]$ be the generalized super-lattice vertex algebra associated to $\big((K^0(X)\oplus\ZZ)\oplus K^1(X),(\tilde{\chi}^L)^\bullet\big)$, where $(\tilde{\chi}^L)^\bullet = \tilde{\chi}^L\oplus \chi^-$  for $\tilde{\chi}^L$ from \eqref{chiLtil} and $\chi^-$ from \eqref{chi-}. The isomorphism \eqref{isohomol} induces an isomorphism of graded vertex algebras
$$
\hat{H}^L_*(\mathcal{P}_X)\cong \Q[K^0(X)\oplus \ZZ]\otimes_{\Q}\textnormal{SSym}_{\Q}[u_{\sigma,i}, \sigma\in \mathbb{B},i>0]\,,
$$
if the same signs $\tilde{\epsilon}^L_{(\alpha,d),(\beta,e)}$ from \eqref{signstwisted} are used for constructing the vertex algebras on both sides. On the right-hand side the degrees are given by 
\begin{align*}
&\textnormal{deg}\Big(e^{(\alpha,d)}\otimes \prod_{\begin{subarray}a
\sigma\in B_{\textnormal{even}}\sqcup\{(0,1)\}\,,
i>0
\end{subarray}}u_{\sigma,i}^{m_{\sigma,i}}\otimes \prod_{\begin{subarray}a
v \in B_{\textnormal{odd}}\\
j>0
\end{subarray}}u_{v,j}^{m_{v,j}}\Big)\\
&= \sum_{\begin{subarray}a \sigma\in B_{\textnormal{even}}\sqcup\{(0,1)\}\\
i>0\end{subarray}}m_{\sigma,i}2i+\sum_{\begin{subarray}a
v \in B_{\textnormal{odd}}\\
j>0
\end{subarray}}m_{v,j}(2j-1)-\tilde{\chi}^L\Big((\alpha,d),(\alpha,d)\Big)\,.
\end{align*}
\sloppy The map $\Omega_*: H_*(\mN_1)\to H_*(\mathcal{P}_X)$ induces a morphism of graded vertex algebras $(\hat{H}^L_*(\mN_1),\ket{0},e^{zT}, Y^L)\to (\hat{H}^L_*(\mathcal{P}_X),\ket{0},e^{zT}, Y^L)$ and of graded Lie algebras
$$
\bar{\Omega}_* :\big(\check{H}^L_*(\mathcal{N}_1),[-,-]^L\big)\longrightarrow \big(\check{H}^L_*(\mathcal{P}_X),[-,-]^L\big)\,.
$$
\end{proposition}
\begin{proof}
Using Lemma \ref{lemma chern} for $\alpha=\llbracket L\rrbracket$, we see that
$$
 \textnormal{ch}_k(\mathfrak{U}^\vee\boxtimes \mathfrak{L})=\sum_{\begin{subarray}a v\in B_{\textnormal{even}}\\
  j=l+k
  \end{subarray}}(-1)^l\chi(L^*,v)\beta_l\boxtimes \mu_{v,k}\,.
$$
Using $\chi(L^*,v) = \chi(v\cdot L)$, one can prove the first part of the theorem by following the proof of  \cite[Thm. 1.1]{gross} or \cite[Thm. 5.19]{Joycehall}. To show the second part, note that $(\Omega\times \Omega)^*(\mathfrak{U}^\vee\boxtimes \mathfrak{L})=\mathcal{L}^{[-,-]}$ and 
$$\xi^L\big((n_1p,d_1),(n_2p,d_2)\big):=\textnormal{rk}\Big(\mathcal{L}^{[n_1,n_2]}_{d_1,d_2}\Big) = d_1n_2=d_1\chi(n_2p\cdot L)\,.$$
The statement then follows from Definition \ref{E-twisted} and Definition \ref{def twistedbyl} by the same arguments as in the proof of Proposition \ref{propmorphism}.
\end{proof}
This completes the following diagram of morphisms of Lie algebras:
$$
\begin{tikzcd}
\arrow[d,"\cap \,c_{\textnormal{top}}(\mathcal{L})"']\big(\check{H}_*(\mathcal{N}_1),[-,-]\big)\arrow[r,"\bar{\Omega}_*"]&\check{H}_*(\mathcal{P}_X,[-,-])\\
\big(\check{H}^L_*(\mathcal{N}_1),[-,-]^L\big)\arrow[r,"\bar{\Omega}_*"]&\check{H}^L_*(\mathcal{P}_X,[-,-]^L)
\end{tikzcd}\,.
$$

%where we use the Pontryagin--Thom construction of Cohen--Klein \cite{CohenKlein} axiomatizing integration along the fiber $X$ to obtain ``wrong way" maps in K-theory.

\subsection{Computing virtual fundamental cycles of 0-dimensional sheaves}
\label{computing zero-dim}
Let us now summarize the different wall-crossing formulae at our disposal. These include the usual wall-crossing without insertions and the twisted version from Lemma \ref{lemma Lconditions}.

Applying part iii. of Conjecture \ref{conjecture WC}  to $\big[\Hilb\big]^{\vir}_{\mN}$ we obtain in $\check{H}_*(\mN_1)$
\begin{equation}
\label{wcfhilb1}
 \big[\Hilb\big]^{\vir}_{\mN}=  \sum_{\begin{subarray}a
k\geq 1,n_1,\ldots,n_k>0\\
n_1+\cdots+n_k = n
\end{subarray}}\frac{(-1)^k}{k!}\big[\big[\ldots \big[[\mathcal{M}_{(0,1)}]^{\inv},[\mathcal{M}_{n_1p}]^{\inv}],\ldots ],[\mathcal{M}_{n_kp}]^{\inv}]\,,
\end{equation}
\sloppy where we used part ii. of Conjecture \ref{conjecture WC} to conclude that $[\mathcal{M}_{n_ip}]^{\inv} = [\mathcal{M}^{\textnormal{ss}}_{n_ip}(\tau)]^{\inv}$ are independent of stability conditions\footnote{We also start using $k$ as a dummy index in the summation instead of the sufficiently large $k$ as in Definition \ref{definition pairs} which has now been set to be 0.}.
To make the notation simpler, we write 
$$
\mathscr{H}_n=\bar{\Omega}_*\Big(\big[\Hilb\big]^{\vir}_{\mN}\Big)\,,\quad\textnormal{and}\quad  \mathscr{M}_{np}= \bar{\Omega}_*\big([\mathcal{M}_{np}]^{\inv}\big)\,.
$$
Using Lemma \ref{lemma Lconditions}, we can apply Proposition \ref{propmorphismtopchern} together with \eqref{wcfhilb1} to get 
\begin{equation}
\label{wcfhilbtwistL}
 I_n(L)1_{(n,1)}=  \sum_{\begin{subarray}a
k\geq 1,n_1,\ldots,n_k>0\\
n_1+\cdots+n_k = n
\end{subarray}}\frac{(-1)^k}{k!}\big[\big[\ldots \big[[\mathcal{M}_{(0,1)}]^{\inv},[\mathcal{M}_{n_1p}]^{\inv}\big]^L,\ldots \big]^L,[\mathcal{M}_{n_kp}\big]^{\inv}\big]^L\,,
\end{equation}
where $1_{(n,1)}\in \check{H}_0(\mathcal{N}_{(n,1)})$ denotes the natural generator. 

Applying Proposition \ref{propmorphism} to \eqref{wcfhilb1}, we get a wall-crossing formula in $\check{H}_*(\mathcal{P}_X)$
\begin{equation}
\label{wcfhilb2}
 \mathscr{H}_n=  \sum_{\begin{subarray}a
k\geq 1,n_1,\ldots,n_k>0\\
n_1+\cdots+n_k = n
\end{subarray}}\frac{(-1)^k}{k!}\big[\big[\ldots \big[e^{(0,1)}\otimes 1,\mathscr{M}_{n_1p}\big],\ldots \big],\mathscr{M}_{n_kp}\big]\,.
\end{equation}
Applying Proposition \ref{propmorphismL} to \eqref{wcfhilbtwistL} we obtain a wall-crossing formula in  $\check{H}^L_*(\mathcal{P}_X)$:
\begin{equation}
\label{wcfhilb2twistL}
 I_n(L)e^{(np,1)}\otimes 1=  \sum_{\begin{subarray}a
k\geq 1,n_1,\ldots,n_k>0\\
n_1+\cdots+n_k = n
\end{subarray}}\frac{(-1)^k}{k!}\big[\big[\ldots \big[e^{(0,1)}\otimes 1,\mathscr{M}_{n_1p}\big]^L,\ldots \big]^L,\mathscr{M}_{n_kp}\big]^L\,.
\end{equation}
We will first use \eqref{wcfhilb2twistL} to compute the classes $\mathscr{M}_{np}$ from the data provided to us by Theorem \ref{thm:park}. Using the same formula to wall-cross back, we will be able to prove Theorem \ref{theorem bojko}.

Let $\Omega_{np}=\Omega|_{(\mathcal{N}_{(np,0)})^{\textnormal{top}}}: (\mathcal{N}_{(np,0)})^{\textnormal{top}}\to \mathcal{C}_X\subset \mathcal{P}_X$, then we will describe  the image of $H_2(\mathcal{N}_{(np,0)})$ under $(\Omega_{np})_*$. For this, note that it follows from \eqref{isohomolC} that there is a natural isomorphism for all $\alpha\in K^0(X)$
\begin{equation}
\label{H-2}
   H_2(\mathcal{C}_\alpha)\cong H^{\textnormal{even}}(X)\oplus \Lambda^2 H^{\textnormal{odd}}(X)\,.
\end{equation}

We will for now work with the following assumption: 
\begin{equation}
\label{eqAsH}
    H^1(X,\mathcal{O}_X) =0= H^1(X,\mathbb{Q}) \,.
\end{equation}
This will be dropped in Remark \ref{remark H}, before we compute any integrals on $\big[\Hilb\big]^{\textnormal{vir}}$.
\begin{lemma}
\label{lemma0dimhom}
If \eqref{eqAsH} holds, the image of $(\Omega_{np})_*:H_2(\mathcal{N}_{(np,0)})\to H_2(\mathcal{C}_{np})$ is contained in $H^6(X)\oplus H^8(X)$ under the isomorphism \eqref{H-2}.  
\end{lemma}
\begin{proof}
We show that $\Omega_{(np,0)}^*(e^{np}\otimes \mu_{v,1})=0$  whenever $v\notin B_6\cup B_8$. Then for any class $U\in H_*(\mathcal{N}_{(np,0)})$ we get $$e^{np}\otimes \mu_{v,1}\big((\Omega_{np})_*(U)\big)= \Omega_{np}^*(e^{np}\otimes\mu_{v,1})(U)=0$$
for $v\in B_{\textnormal{even}}\backslash (B_6\cup B_8)$ and $$e^{np}\otimes\mu_{v,1}\mu_{w,1}\big((\Omega_{np})_*(U)\big)=\Omega_{np}^*(e^{np}\otimes\mu_{v,1}\mu_{w,1})(U)=0$$ 
for $v,w\in B_{\textnormal{odd}}$.
The conclusion then follows from \eqref{normalization}.

The K-theory  class $\llbracket\mathcal{E}_{np}\rrbracket$ of the universal sheaf of points on $\mathcal{N}_{(np,0)}$ is given by $(\textnormal{id}_X\times\Omega_{np})^*(\mathfrak{E}_{np})$. Then from \eqref{chernofE} we see
\begin{equation}
\label{pullchernofE}
   \textnormal{ch}(\mathcal{E}_{np})=\sum_{\begin{subarray}a v\in B\\
   i\geq 0\end{subarray}}\ch(v)\boxtimes \Omega_{np}^*(e^{np}\otimes\mu_{v,i})\,. 
\end{equation}
We also know that $\textnormal{ch}_i(\mathcal{E}_{np})=0$ for $i<4$ by dimension arguments. By assumption \eqref{eqAsH} and Poincaré duality, we have $H^7(X)$ =0. We thus only need to consider $v\in B_j$ for $j<6$. Then from looking at \eqref{pullchernofE} we see $v\boxtimes \Omega_{np}^*(e^{np}\otimes \mu_{v,1})=0$ because it is in degree $2+j<8$ or $1+j<8$ and $B$ is a basis. Therefore $\Omega_{np}^*(e^{np}\otimes\mu_{v,1})=0$.
\end{proof}

For degree reasons, we may write 
\begin{equation}
\label{Nnpdef}
    \mathscr{M}_{np} = e^{(np,0)}\otimes 1\cdot \mathscr{N}_{np} + \Q T(e^{(np,0)}\otimes 1)\in \check{H}_*(\mP_X).
\end{equation}
  where $\mathscr{N}_{np}$ is a polynomial in the variables $u_{v,k}$ describing a choice of a lift of $\mathscr{M}_{np}$. 
  \begin{proposition}
  \label{proposition 68}
Assuming \eqref{eqAsH}, there is a unique $\mathscr{N}_{np}$ from \eqref{Nnpdef}, such that for some $a_v(n)\in \Q$, $v\in B_6$ we have
$$
  \mathscr{N}_{np} = \sum_{v\in B_{6}}a_v(n)u_{v,1}\,.
  $$
  \end{proposition}
  \begin{proof}
 As 
  $\mathscr{M}_{np} =(\Omega_{np})_*\big([\mathcal{M}_{np}]^{\inv}\big)$, 
  by Lemma \ref{lemma0dimhom}, we have 
  $$\mathscr{N}_{np} = \sum_{v\in B_{6}\cup B_8}a_v(n)u_{v,1}\,.$$
  From Lemma \ref{lemma identities}, we see that $T(e^{np}\otimes 1) = e^{np}\otimes nu_{p,1}$. Therefore, we get  $H^8(X) = T\big(H_0(\mathcal{C}_{np})\big)$ which concludes the proof.
  \end{proof}
 
 \begin{lemma}
\label{lemma caoqu}
For each very ample line bundle $L$ such that $c_1(L)\cdot c_3(X)\neq 0$, there exist unique orientations $o_n(L)$ on $\Hilb$ such that Conjecture \ref{conjecture CK} holds for $L$.
\end{lemma}
\begin{proof}
Since $\textnormal{dim}(X)>1$, Bertini's theorem \cite[Thm. 8.18]{hartshorne} tells us that there exists a smooth connected divisor $D$ such that $L=\mathcal{O}_X(D)$. The lemma then follows from 
\begin{theorem}[Park {\cite[Cor. 0.3]{Huy}, Cao--Qu \cite[Thm. 1.2]{CaoQu}}]
\label{thm:park}
Conjecture \ref{conjecture CK} holds for any $X$, $L\cong \mathcal{O}_X(D)$ for a smooth connected divisor $D$, and some choices of orientations $o_n(L)$, $n\geq 1$.
\end{theorem}
 The uniqueness of $o_n(L)$ in the case $c_1(L)\cdot c_3(X)\neq 0$ follows because changing orientations changes the sign by the connectedness of $\Hilb$ for all $n$ (see Hartshorne \cite{HarConHilb}).
\end{proof}
   Let $\textnormal{Amp}(X)\subset H^2(X)$ be the image of the ample cone under the natural map $A^1(X,\Q)\to H^2(X,\Q)$. Let us choose $B_2$ such that its elements are $c_1(L)$ for very ample line bundles $L$. This is possible: We assumed $H^2(\mathcal{O}_X)=0$ for $X$ a Calabi--Yau fourfold, so $H^2(X) =H^{1,1}(X)$. Thus every element in $H^2(X)$ is obtained as $\frac{m}{n}[D]$ for an algebraic divisor $D\subset X$ and $m,n\in \ZZ$. On the other hand $[D]+n[H]$ is very ample if $n\gg0$ and $H$ very ample so $[D]=([D]+n[H])-n[H]$, where both terms are very ample. Using the Poincaré pairing on $H^2(X)\times H^6(X)\to \Q$ we choose a basis $B_6$ of $H^6(X)$ which is dual to $B_2$. 

Let us denote $o(n)$ the orientations on $\Hilb$ induced by the point-canonical orientations. We will see that the orientations $o_n(L)=o(n)$ for all $L$ with $c_1(L)\cdot c_3(X)\neq 0$. In stating the results below, we will also use that 
$$
\Q[K^0(X)]\otimes_{\Q}\textnormal{SSym}_{\Q}[u_{v,i}, v\in B,i>0]
$$
inherits a natural algebra structure as the tensor product of a group algebra and a polynomial algebra.
  \begin{theorem}
  \label{theorem workhorse}
If Conjecture \ref{conjecture WC} holds for $X$ and $H^1(\mO_X)=0$, then the following is true:
\begin{enumerate}[label=\roman*)]
    \item For all $L$ from Lemma \ref{lemma caoqu} with $c_1(L)\cdot c_3(X)\neq 0$ the orientation $o_n(L)$ coincides with the ones obtained from the point-canonical orientations in Definition \ref{definition canonicalor}.
    \item Let  $$\mathscr{N}(q)=\sum_{n>0}e^{np}\otimes \mathscr{N}_{np}\,q^n\in \Q[K^0(X)]\otimes_{\Q}\textnormal{SSym}_{\Q}[ u_{v,i}, v\in B,i>0]\llbracket q\rrbracket $$
    be the generating series of $\mathscr{N}_{np}$.  then we can express its exponential as
\begin{equation}
\label{expNq}
     \exp\big( \mathscr{N}(q) \big)= 
     M(e^pq)^{\Big(\sum_{v\in B_6}c_3(X)_vu_{v,1}\Big)}\,,    
\end{equation}
    where $c_3(X)_v= c_3(X)\big(\textnormal{ch}^{\vee}(v^\vee)\big)$\footnote{The right hand side should be evaluated as $\exp\Big[\sum_{v\in B_6}c_3(X)_vu_{v,1}\log \Big(M(e^pq) \Big)\Big]$.}.  Equivalently, we can write this as 
    \begin{equation}
\label{Nnpform}
  \mathscr{N}_{np} = \sum_{l|n}\frac{n}{l^2}\sum_{\begin{subarray}a v\in B_6 \end{subarray}}c_3(X)_v u_{v,1}\,.  
\end{equation}
  \end{enumerate}
  \end{theorem}
  \begin{proof}
  We prove the theorem by induction on $n$. We begin by giving an explicit formula for the brackets in \eqref{wcfhilb2twistL}. Using \eqref{fields} together with Lemma \ref{lemmasigns} iii., we have:
  \begin{align*}
&Y^L(e^{(mp,1)}\otimes 1,z)(e^{(np,0)}\otimes \mathscr{N}_{np}) \\
&={\scriptstyle (-1)^ne^{\big((m+n)p,1\big)} \textnormal{exp}\bigg[\sum_{j>0}\frac{b_j+my_j}{j}z^j\bigg]\bigg[1-z^{-1}\sum_{v\in B_{\textnormal{even}}}\tilde{\chi}^L((mp,1),(v,0))\frac{d}{du_{v, 1}}\bigg]\mathscr{N}_{np}}\,,
\end{align*}
where we used that $\mathscr{N}_{np}$ is linear in $u_{v,1}$.
Using \eqref{chiLtil} together with Proposition \ref{proposition 68} we get the following after taking $[z^{-1}](-)$ of the last formula:
\begin{equation}
\label{liebrackettwL}
    [e^{(mp,1)}\otimes 1, e^{(np,0)}\otimes \mathscr{N}_{np}]^L = -(-1)^ne^{\big((m+n)p,1\big)}\otimes \sum_{v\in B_6}\int_Xc_1(L)\textnormal{ch}(v)a_{v}(n)\,.
\end{equation}
 Let $L_1$ be such that $c_1(L_1)\in B_2$ and $v_1\in B_6$ its dual. For now let us not fix the orientation $o_p = o^{\textnormal{can}}_p$, but fix  $o_{\llbracket\mathcal{O}_X \rrbracket} = o^{\textnormal{can}}_{\llbracket \mathcal{O}_X\rrbracket}$ and use the rest of Definition \ref{definition canonicalor}. 

For $n=1$
, we can choose $o_p$ so that $o(1) = o_1(L_1)$ , then  $I_{1}(L_1) =  -I(L_1)$. Using \eqref{liebrackettwL} together with \eqref{wcfhilbtwistL} and \eqref{coefficient}, we get
$
\int_Xc_1(L_1)\textnormal{ch}(v_1)a_{v_1}(1) = I(L_1)\,.
$
Therefore $a_{v_1}(1)= c_3(X)_{v_1}$. Suppose that $L_2$ is a line bundle with $c_1(L_2)\in B_2$ different from $c_1(L_1)$ and $I(L_i) \neq 0$ for $i=1,2$. If $o_1(L_2) = -o_1(L_1)$, then $a_{v_2}(1)= -c_3(X)_{v_2}$ for $v_2\in B_6$ the dual of $c_1(L_2)$ and this contradicts Lemma \ref{lemma caoqu}: For any $A,B\in \ZZ_{>0}$ we know that $L=L_1^{\otimes A}\otimes L_2^{\otimes B}$ is very ample, then  from \eqref{liebrackettwL} we get
$$
-\Big[AI(L_1)+BI(L_2) \Big] = I_{1}(L) = - AI(L_1)+ BI(L_2)
$$
which can not be true for all $A,B$. For any $v\in B_6$, we then have $a_{v_1}(1)= c_3(X)_{v_1}$.

This shows i. and ii. for $n=1$ except $o_p = o^{\textnormal{can}}_p$. Let us now assume that i. and ii. hold for all $1\leq k\leq n$ except $o_p = o^{\textnormal{can}}_p$. If $o_{n+1}(L_1) = -o(n+1)$ then $I_{n+1}(L_1) = -[q^{n+1}]\{M(-q)^{c_1(L_1)\cdot c_3(X)}\}$\footnote{Note that, we assume that $I_{n+1}(L_1)\neq 0$. We can do this because otherwise, we would obtain a contradiction using $I(L_1)\neq 0$.}. Using the assumption together with \eqref{liebrackettwL} we get using the notation of \eqref{coefficient}
\begin{align*}
&\sum_{\begin{subarray}a
k>1,n_1,\ldots,n_k>0\\
n_1+\cdots+n_k = n+1
\end{subarray}}\frac{(-1)^k}{k!}\big[\big[\ldots \big[e^{(0,1)}\otimes 1,\mathscr{M}_{n_1p}\big]^{L_1},\ldots \big]^{L_1},\mathscr{M}_{n_kp}\big]^{L_1}
=  \sum_{k>1}d_k(n+1)I(L_1)^k\,. 
\end{align*}
Subtracting this from $I_{n+1}(L_1)$ and using \eqref{coefficient} expresses $a_{v}(n+1)$ as
$$
a_{v}(n+1) =-d_1(n+1)I(L_1)-2d_2(n+1)I(L_1)^2 - \cdots -2d_{n+1}(n+1)I(L_1)^{n+1}\,,
\quad 
$$
Let $L=L_1^{\otimes N}$ for $N>0$, then wall-crossing and using \eqref{liebrackettwL} with \eqref{wcfhilb2twistL} gives
\begin{align*}
   I_{n+1}(L) &=\sum^{n+1}_{k=2}d_k(n+1)\big(NI(L_1)\big)^k-d_1(n+1)NI(L_1) - 2N\sum^{n+1}_{k=2}d_k(n+1) I(L_1)^k \,.
\end{align*}
By comparing the coefficients of different powers of $N$ with $\pm I_{n+1}(L)$, we obtain a contradiction.  This also shows $o_{n+1}(L_1)=o(n+1)=o_{n+1}(L_2)$ for any $L_2$ with $I(L_2)\neq 0$. Assuming ii. holds for coefficients $k<n$ and using \eqref{wcfhilbtwistL},\eqref{coefficient} gives us 
\begin{align*}
  (-1)^{n+1} a_{v_1}(n+1) &= \sum_{k\geq 1}d_k(n+1)I(L_1)^k-\sum_{k>1}d_k(n+1)I(L_1)^k \\
  &= (-1)^{n+1}\sum_{l|n+1}\frac{n+1}{l^2}c_3(X)_{v_1}\,.
\end{align*}
This holds for all $L_i\in B_2$ with their duals $v_i\in B_6$, so part ii. follows as we obtain \eqref{Nnpform}.\\

To finish the proof of part i., we only need to show that $o_p = o^{\textnormal{can}}_p$. For this, choose $L$ such that $I(L)\neq 0$. Using Lemma \ref{lemma chern}, we see $c_1(\mathfrak{L})= \sum_{v\in B_{\textnormal{even}}}\chi(L^\vee,v)\mu_{v,1}$. Using $X=M_p$ and \eqref{Mpclass}, we see that $\int_{[M_p]^{\textnormal{vir}}}c_1(L) =\pm I(L)$\,.  By \eqref{Nnpdef} this is equal to 
$$
\int_{\mathscr{M}_p}\sum_{v\in B_{\textnormal{even}}}\chi(L^\vee,v)\mu_{v,1}\,,\qquad \mathscr{M}_p = \sum_{v\in B_6}c_3(X)_vu_{v,1} + c_pu_{p,1}\,,
$$
which gives $I(L) + c_p$. As $c_p$ does not depend on $L$ it has to be 0. Therefore for the invariants to coincide, we need $o_p = o^{\textnormal{can}}_p$\,.
  \end{proof}

 \begin{remark}
Changing orientation $o_p\mapsto -o_p$ changes $o_{np}\mapsto (-1)^no_{np}$, so if the classes $[\mathcal{M}_{np}]^{\inv}$ were constructed using Borisov--Joyce \cite{BJ} or Oh--Thomas \cite{OT}, then we would get
$$
  \mathscr{N}_{np} = (-1)^n\sum_{l|n}\frac{n}{l^2}\sum_{\begin{subarray}a v\in \Lambda \end{subarray}}c_3(X)_vu_{v,1}
$$
However, as these are obtained indirectly through wall-crossing, we should check this is satisfied. Choosing $o_p$ such that $o(1)=-o_1(L_1)$ in the proof of Theorem \ref{theorem workhorse} does indeed give this formula. Similarly, switching to $-o^{\textnormal{can}}_{\llbracket\mathcal{O}_X\rrbracket}$ does not change the result as it should not. 
 \end{remark}
 
 The following is shown just for completeness, as we will prove a much more general statement for tautological insertions using any K-theory class in §\ref{sectalltaut}\\
 
 \textit{Proof of Theorem \ref{theorem bojko}}
 Using \eqref{liebrackettwL}, \eqref{Nnpform} and \eqref{wcfhilb2twistL} we obtain for any line bundle $L$ that \eqref{coefficient} holds. 
\qed

\section{Virtual classes of Hilbert schemes of points and invariants}
\label{VFC section}
In this section, we use the results of Theorem \ref{theorem workhorse} (relying on Conjecture \ref{conjecture WC}) to compute $\mathscr{H}_n$. Knowing $\mathscr{N}_{np}$ and the explicit vertex algebra structure on $\hat{H}_*(\mP_X)$ described in §\ref{sec:ecplicitVApa}, we may use the wall-crossing formula \eqref{wcfhilb2} to do so. The classes $\mathscr{H}_n\in \check{H}_{0}(\mathcal{P}_X)$ that we compute could be considered as the final product in describing the topological data of virtual fundamental classes of Hilbert schemes of points, but compared to integrals of appropriate insertions, they lack enlightening structure. This is why we explicitly determine the values of integrals of multiplicative genera of tautological classes paired with multiplicative genera of virtual tangent bundles. This result follows from a more general one that computes integrals of cohomology classes of the form $\textnormal{exp}\big[F(\mu_{v,k})\Big]$, where $F(\mu_{v,k})$ is linear in $\mu_{v,k}$, so we can use Lemma \ref{linear exp}.

Section \ref{sec:invariants} will focus on particular choices of the multiplicative genera, and it will rely on the results presented here.
\subsection{Virtual fundamental cycle of Hilbert schemes}
\label{sec:vfcHn}
The following is the main technical result of this section. It computes the generating series of classes $\mathscr{H}_n$ by using \eqref{wcfhilb2} and the previously obtained description of $\mathscr{M}_{np}$. 
  \begin{theorem}
  \label{theoremhilb}
  If Conjecture \ref{conjecture WC} holds for $X$ and $H^1(\mO_X)=0$, then the generating series
  $
  \mathscr{H}(q) = 1+\sum_{n>0}\frac{\mathscr{H}_n}{e^{(np,1)}}q^n
  $ for point-canonical orientations
   is given by
\begin{equation}
\label{Hq}
\mathscr{H}(q)= \exp\bigg[\sum_{n>0}\sum_{l|n,v\in B_6}(-1)^n\frac{n}{l^2}c_3(X)_v [z^n]\Big\{U_v(z)\textnormal{exp}\Big[\sum_{j>0}\frac{ny_j}{j}z^j\Big]\Big\}q^n\bigg]\,,
\end{equation}
where we fix the notation $y_j=u_{p,j}$ and  $U_v(z)=\sum_{k>0}u_{v,k}z^k$. 
  \end{theorem}
  \begin{remark}
    \label{remark K1}
  \label{remark nob}
  Notice that the only $u_{\sigma,k}$ that appear in \eqref{Hq} are for $\sigma = (v,0)$, $v\in B_6\cup B_8=:B_{6,8}$. We may therefore assume $K^1(X)=0$ from now on when $\eqref{eqAsH}$ holds (see also Remark \ref{remark H} for the general case).  As there is no contribution of $b_j$, this is the unique representation of $\mathscr{H}_n$ without terms with $b_j$ as can be seen from Lemma \ref{lemma identities}. Using \eqref{inlift}, we have a class $\tilde{\mathscr{H}}_n = \Omega_*\circ \iota_{n\,*}\Big(\big[\Hilb\big]^\textnormal{vir}\Big)$ which satisfies $\Pi_{0}(\tilde{\mathscr{H}}_n) = \mathscr{H}_n$. There will also be no terms containing $b_j$ in $\tilde{\mathscr{H}}_n$, thus $[q^n]\big(\mathscr{H}(q)\big)$ describe this canonical lift.
  \end{remark}
  \begin{proof}
  We begin by using the reconstruction lemma for vertex algebras
 to write the field $ Y(e^{(np,0)}\otimes u_{v,1})=: Y(u_{v,1},z) Y(e^{(np,0)}\otimes 1,z) :$, where $:-:$ denotes the normal ordered product for fields of vertex algebras (see \cite[§3.8]{LLVA}, \cite[Def. 2.2.2]{BZVA}, \cite[§3.1]{KacVA}) which acts on $e^{(mp,1)}\otimes U$ as 
 \begin{align*}
 \label{acting}
 &: Y(u_{v,1},z) Y(e^{(np,0)}\otimes 1,z) : (e^{(mp,1)}\otimes U)=\\
     &(-1)^nz^{-n}e^{((n+m)p,1)}
   \otimes\bigg\{\big(\sum_{k>0} u_{v,k}z^{k-1}\big)\textnormal{exp}\Big[\sum_{i>0}\frac{ny_i}{i}z^{i}\Big]\textnormal{exp}\Big[-n\sum_{i>0}\frac{d}{du_{\llbracket\mathcal{O}_X\rrbracket,i}}z^{-i}\Big]\\
   &\textnormal{exp}\Big[n\sum_{i>0}\frac{d}{db_{i}}z^{-i}\Big]
   +\textnormal{exp}\Big[\sum_{i>0}\frac{ny_i}{i}z^{i}\Big]\textnormal{exp}\Big[-n\sum_{i>0}\frac{d}{du_{\llbracket\mathcal{O}_X\rrbracket,i}}z^{-i}\Big]\\
   &\textnormal{exp}\Big[n\sum_{i>0}\frac{d}{db_{i}}z^{i}\Big]\Big[\tilde{\chi}((v,0),(mp,1))z^{-1}
   +\sum_{k>0,\sigma\in \mathbb{B}}k\tilde{\chi}((v,0),\sigma)\frac{d}{du_{\sigma,k}}z^{-k-1}\Big]\bigg\}U
   \numberthis{}
 \end{align*}
Where we used \eqref{ptop} to get $\tilde{\chi}((np,0),\sigma) = n$ if $\sigma = (\llbracket \mathcal{O}_X\rrbracket,0)$, $\tilde{\chi}((np,0),\sigma)=-n$ if $\sigma =(0,1)$ and 0 otherwise together with part iii. of Lemma \ref{lemmasigns}.

We claim that for any $r>0$, $n_1,\ldots, n_r>0$ and $\sum_{i=1}^rn_i=n$, we have the following :
\begin{align*}
 &\big[e^{(n_1p,0)}\otimes \mathscr{N}(n_1p), \big[e^{(n_2p,0)}\otimes \mathscr{N}(n_2p),\big[\ldots ,\big[e^{(n_rp,0)}\otimes \mathscr{N}(n_rp),e^{(0,1)}\otimes 1\big]\ldots\big]\big]\\& = e^{(np,1)}\otimes\prod^r_{i=1} \mathscr{H}_{n_i}\,   
\end{align*}
in $\check{H}(\mathcal{P}_X)$, where 
$$
\mathscr{H}_n= \sum_{l|n,v\in B_6}(-1)^n\frac{n}{l^2}c_3(X)_v [z^n]\Big\{U_v(z)\textnormal{exp}\Big[\sum_{j>0}\frac{ny_j}{j}z^j\Big]\Big\}\,.
$$
We show this by induction on $r$, where for $r=0$ it is obvious. Assuming that the statement holds for $r-1$, we need to compute
\begin{align*}
  &[e^{(n_1p,0)}\otimes \mathscr{N}(n_1p), e^{((n-n_1)p,1)}\otimes \prod_{i=2}^r\mathscr{H}_{n_i}]\\
  &=[z^{-1}]\Big\{:Y(u_{v,1},z)Y(e^{(n_1p,0)}\otimes 1,z):e^{((n-n_1)p,1)}\otimes \prod_{i=2}^r\mathscr{H}_{n_i}\Big\}
\end{align*}
 From Remark \ref{remark nob}, we see that we can replace all $\textnormal{exp}\Big[n\sum_{i>0}\frac{d}{db_i}\Big]$ and $\textnormal{exp}\Big[-n\sum\frac{d}{du_{\llbracket\mathcal{O}_X\rrbracket,i}}\Big]$ by 1 in \eqref{acting}. The second term under the curly bracket in \eqref{acting} vanishes, because it contains $\tilde{\chi}((v,0),(mp,1))z^{-1} = \chi(v,mp)z^{-1} -\chi(v)z^{-1}$
\sloppy where the result is zero for degree reasons and because $\textnormal{td}_1(X)=0$. In the term with 
$\sum_{k>0,\sigma\in \mathbb{B}}k\tilde{\chi}((v,0),\sigma)\frac{d}{du_{\sigma,k}}z^{-k-1}$ the sum can be taken over all $\sigma=(w,0)$, $w\in B_{6,8}$ by Remark \ref{remark nob} so it vanishes
because $\chi(v,w) = 0$ whenever $v,w\in B_{6,8}$. 
We are therefore left with 
\begin{align*}
  &[z^{-1}]\Big\{(-1)^{n_1}z^{-n_1}e^{(np,1)}
   \otimes\big(\sum_{k>0} u_{v,k}z^{k-1}\big)\textnormal{exp}\Big[\sum_{i>0}\frac{n_1y_i}{i}z^{i}\Big]\Big\}\prod_{i=2}^r\mathscr{H}_{n_i}\\
   &= (-1)^ne^{(np,1)}[z^{n_1}]\Big\{U_v(z)\textnormal{exp}\Big[\sum_{j>0}\frac{n_1y_j}{j}z^j\Big]\Big\} \prod_{i=2}^r\mathscr{H}_{n_i}\,.  
\end{align*}
Multiplying with the coefficients $\sum_{l|n}\frac{n}{l^2}\sum_{v\in B_6}c_3(X)_v$ of $\mu_{v,1}$ in \eqref{Nnpform} and summing over all $v\in B_6$, we obtain the result as we are able to rewrite \eqref{wcfhilb2} by reordering the terms keeping track of the signs as
$$
\mathscr{H}_n =  \sum_{\begin{subarray}a
k\geq 1,n_1,\ldots,n_k\\
\sum n_i = n
\end{subarray}}\frac{1}{k!}\big[e^{(n_1p,0)}\otimes \mathscr{N}(n_1p) ,\big[\ldots,\big[e^{(n_kp,0)}\otimes \mathscr{N}(n_kp),e^{(0,1)}\otimes 1\big]\ldots\big]\big]\,.
$$
 
  \end{proof}
  
  We now describe a general formula for integrating topological insertions over $\mathscr{H}_n$ which will be applied in the following two examples. Note that we no longer need \eqref{eqAsH} to state it.
 \begin{proposition}
 \label{theorem big}
Let $X$ be a Calabi--Yau fourfold satisfying Conjecture \ref{conjecture WC} (and $H^2(\mO_X) = 0$), and
 let $\mathscr{A}\subset K^0(X)\backslash \{0\}$ be a finite subset. For each $\alpha\in \mathscr{A}$, take a generating series
 $$
 A_\alpha(z,\mathfrak{p}) = \sum_{n\geq 0}a_\alpha(n,\mathfrak{p})\frac{z^n}{n!}\,,
 $$
 where $\mathfrak{p}=(p_1,p_2,\ldots)$ are additional variables and $a_\alpha(n,\mathfrak{p})\in \Q\llbracket\mathfrak{p} \rrbracket$, s.t. $a_\alpha(0,0)=0$. \sloppy If $\mathcal{I}_n\in H^*\big(\Hilb\big)$ is obtained as $\mathcal{I}_n = (\Omega\circ \iota_n)^*(\mathcal{T})$ for a weight 0 insertion $\mathcal{T}\in H^*(\mathcal{P}_X)$ such that 
 $$
 \int_{\mathscr{H}_n}\mathcal{T} = \int_{\mathscr{H}_n}e^{(np,1)}\otimes \textnormal{exp}\Big[\sum_{\alpha\in \mathscr{A}}\sum_{\begin{subarray}a k\geq 0\\
 v\in B_{6,8}\end{subarray}}a_\alpha(k,\mathfrak{p})\chi(\alpha^\vee,v)\mu_{v,k}\Big]\,,
 $$
 then the generating series $\textnormal{Inv}(q) =1+\sum_{n>0}\int_{[\Hilb]^{\textnormal{vir}}}\mathcal{I}_nq^n$ is given by
\begin{align*}
\label{explinearhilbform}
  &\prod_{\alpha\in \mathscr{A}}\textnormal{exp}\Big\{\sum_{n>0}(-1)^n\sum_{l|n}\frac{n}{l^2}[z^n]\Big[z\frac{d}{dz}\big(A_{\alpha}(z,\mathfrak{p})-A_\alpha(0,\mathfrak{p})\big)\\
 &\textnormal{exp}\Big(\sum_{\beta\in \mathscr{A}}\textnormal{rk}(\beta)A_\beta(z,\mathfrak{p})\Big)^n\Big]q^n\Big\}^{c_1(\alpha)\cdot c_3(X)}\,. 
 \numberthis
\end{align*}
 \end{proposition}
  \begin{proof}
   Using Lemma \ref{linear exp} to act on the homology classes $\mathscr{H}_n$ from Theorem \ref{theoremhilb}, we obtain
   \begin{align*}
  &\textnormal{Inv}(q)\\
  &= \textnormal{exp}\bigg[\sum_{n>0}\sum_{\begin{subarray}a l|n\\v\in B_6\end{subarray}}(-1)^n\frac{n}{l^2}c_3(X)_v[z^n]\Big\{\sum_{k>0}\sum_{\alpha\in \mathscr{A}}\chi(\alpha^\vee,v)\frac{a_{\alpha}(k,\mathfrak{p})}{(k-1)!}z^k\\ &\cdot \textnormal{exp}\Big[\sum_{j\geq 0}\sum_{\alpha\in \mathscr{A}}n\frac{\textnormal{rk}(\alpha)a_{\alpha}(j,\mathfrak{p})}{j!}z^j\Big]\Big\}q^n\bigg]
\end{align*}
which can be seen to be equal to \eqref{explinearhilbform}.
  \end{proof}
  We get the following simple consequence of the above results. To state the second part of it, we use the notation 
  $$
  \mu^{\mF_n}_{v,k} = (\Omega\circ \iota_n)^*\mu_{v,k}
  $$
  which implies that
  $$
  \mu^{\mF_n}_{v,k} =\big(\ch^\vee(v^\vee)\backslash\ch(\mF_n)\big)_k \in H^*(\Hilb)
  $$
  are the usual descendent classes on $\Hilb$ constructed using the universal 0-dimensional sheaf $\mF_n$. 
  \begin{corollary}
  \label{cordependence}
   With the notation and assumptions from Proposition \ref{theorem big}, it follows that $\textnormal{Inv}(q)$ depends only on $c_1(\alpha)\cdot c_3(X)$ and $\textnormal{rk}(\alpha)$ for all $\alpha\in \mathscr{A}$. 

If additionally $H^1(\mO_X)=0$, then for any polynomial $p(\mu_{v,k})$ in $\mu_{v,k}$, the descendent integral 
$$
\int_{[\Hilb]^{\vir}}p(\mu^{\mF_n}_{v,k})
$$
depends only on $\int_Xc_3(X)\cdot(-): H^2(X) \to \ZZ$.
  \end{corollary}

  \begin{remark}
  \label{noinv}
    For the classes $[\mathcal{M}_{np}]^{\inv}\in \check{H}_2(\mathcal{N}_X)$, we did not find any interesting non-zero invariants of the form $\int_{[\mathcal{M}_{np}]^{\inv}}\nu$ for some weight 0 insertions $\nu$ on $\mathcal{N}_X$. We already know that $\mathcal{L}^{[-,-]}|_{ \mathcal{N}_{n,0}}=0$, and this will be also true for insertions correcting $T_0(\alpha)=\pi_{2\,*}(\pi_X^*(\alpha)\otimes \mE_1)$ to be weight 0. Moreover, consider the weight 0 complex $\mathbb{E}_0=\pi_*\big(\underline{\textnormal{Hom}}_{\mM_1}(\mE_1,\mE_1)\big)^\vee$ on $\mM_1$ and set $\nu = p(\textnormal{ch}_1(\mathbb{E}_0),\textnormal{ch}_2(\mathbb{E}_0),\ldots)$ for some polynomial $p(x_1,x_2,\cdots)$. Then we get
    $$
    \int_{[\mathcal{M}_{np}]^{\inv}}\nu=\int_{e^{(0,np)}\otimes \mathscr{N}_{np}}p(\textnormal{ch}_1(\Delta^*(\theta)),\textnormal{ch}_2(\Delta^*(\theta),\ldots)\,,
    $$
    which can be shown to be always zero.
    \end{remark}
 \subsection{Multiplicative genera as insertions}
 \label{multigen paragraph}
The main examples of insertions we want to address are \textit{multiplicative genera} of \textit{tautological classes} below. The results in this and the following sections are stated for a general Calabi--Yau fourfold $X$ without requiring the condition $H^1(\mO_X)=0$ (recall that we still need the vanishing of $H^2(\mO_X)=0$ due to the arguments leading up to Theorem \ref{theorem workhorse}). Furthermore, all the results specific to $[\Hilb]^{\vir}$ are assuming Conjecture \ref{conjecture WC} because they are relying on the computations in §\ref{sec:vfcHn}. \\
 
 For a scheme $S$, let $G^0(S)$ and $G_0(S)$ denote its Grothendieck groups of vector bundles and coherent sheaves respectively. We have the map $\lambda:G^0(S)\to K^0(S)$ which we often neglect to write, i.e. $\lambda(\alpha) = \alpha$. We have the Chern-character $\textnormal{ch}:G^0(S)\to A^*(S,\Q)$ which under the natural maps to $K^0(S)$ and $H^*(S,\Q)$ corresponds to the topological Chern-character $\textnormal{ch}:K^0(S)\to H^{\textnormal{even}}(X,\Q)$.

  Set 
  $$R=\Q\llbracket \mathfrak{p}\rrbracket$$
  to be the ring of formal power-series in the variables $\mathfrak{p}=(p_1,\ldots,p_k)$, and let $f(\mathfrak{p},z) =\sum_{n\geq 0}f_n(\mathfrak{p})z^n\in R\llbracket z\rrbracket$ be a formal power-series satisfying $f(0,0)=1$. For another non-negatively graded ring $Q^*$ write $\big(Q^*\llbracket \mathfrak{p}\rrbracket\big)_1$ to denote the multiplicative group of elements of the ring $Q^*\llbracket\mathfrak{p}\rrbracket$ with the part of their constant term in the variables $\mathfrak{p}$ that is contained in $Q^0$ being equal to 1. The \textit{multiplicative Hirzebruch-genus} $\mathscr{G}_{f(\mathfrak{p},\cdot)}$ of \cite[§4]{Hirzebruch} associated to $f$ is a group homomorphism
\begin{equation}
\label{genus}
  \mathscr{G}_{f}: G^0(X)\longrightarrow \big(A^*(X,\Q)\llbracket \mathfrak{p}\rrbracket\big)_1\,.  
\end{equation} For each vector bundle $E\to X$ of $\textnormal{rk}(E) =a$, let $x_1,\ldots,x_a$ be its Chern roots. The map $\mathscr{G}_{f}$ is given by
$$
\mathscr{G}_{f}(E) =\prod_{i=1}^af(\mathfrak{p},x_i)\,.
$$
 Define $\Lambda^\bullet_t: G^0(S)\to \Big(G^0(S)\llbracket t\rrbracket\Big)_1$ to be the group homomorphism acting on each vector bundle $E$ by
 $$
[E]\mapsto \sum_{i=0}^\infty [\Lambda^iE](-t)^i\,,
 $$
 where $G^0(S)$ is a ring with the multiplication induced by the tensor product. We also have $\textnormal{Sym}^\bullet_t: G^0(S)\to \Big(G^0(S)\llbracket t\rrbracket\Big)_1$, where $\textnormal{Sym}^\bullet_{-t}(\alpha) = \Big(\Lambda^\bullet_t(\alpha)\Big)^{-1}$ for all $\alpha\in G^0(S)$.
 On $\Hilb$, we will consider the classes
\begin{equation}
\label{alphan}
  \alpha^{[n]}= \pi_{2\, *}\big(\pi^*_{X}(\alpha)\otimes \mathcal{F}_n\big)\,,\quad \alpha\in G^0(X)  \,,\qquad
    T^{\textnormal{vir}}_n =\underline{\textnormal{Hom}}_{\Hilb}\big(\mathscr{I}_n,\mathscr{I}_n\big)_0[1]\,,  
\end{equation}
where $\mathscr{I}_n =\big(\mathcal{O}_X\to \mathcal{F}_n\big)$ is the universal complex on $\Hilb$ and $(-)_0$ denotes the trace-less part.  The corresponding topological analogs are 
$$
\mathfrak{T}(\alpha) =  \mathfrak{U}^\vee\boxtimes \pi_{2\,*}\big(\pi_X^*(\alpha)\otimes  \mathfrak{E}\big)\qquad 
\textnormal{and}\qquad -\theta^\vee_{\mathcal{P}}\in  K^0(\mathcal{P}_X\times \mathcal{P}_X)\,.$$

\begin{lemma}
\label{lemma chern}
In $H^*(\mathcal{P}_X\times\mathcal{P}_X)$ the following holds for all $\alpha \in K^0(X)$:
\begin{align*}
  \textnormal{ch}_k\big(\mathfrak{T}(\alpha) \big) &= \sum_{\begin{subarray}a v\in B_{\textnormal{even}}\\
  i+j=k
  \end{subarray}}(-1)^j\chi(\alpha^\vee,v)\beta_i\boxtimes \mu_{v,j}\,,\\
 \textnormal{ch}_k(\theta_{\mathcal{P}}) &=\sum_{\begin{subarray}a i+j=k
  \\
  \sigma,\tau\in\mathbb{B}\backslash B_{\textnormal{odd}} \end{subarray}}(-1)^j\tilde{\chi}(\sigma,\tau)\mu_{\sigma,i}\boxtimes \mu_{\tau,j}   +\sum_{\begin{subarray}a i+j=k+1
  \\
v,w\in  B_{\textnormal{odd}} \end{subarray}}(-1)^i\chi^-(v,w)\mu_{v,i}\boxtimes \mu_{w,j}\,.
  \end{align*}
  We also have the identity 
  $$\textnormal{ch}(T^{\textnormal{vir}}_n)=-(\Omega\circ \iota_n)^*\big(\textnormal{ch}(\Delta^*\theta^\vee_{\mathcal{P}})\big) +\chi(\mO_X)$$
  in $H^*\big(\Hilb\big)$. 
\end{lemma}
\begin{proof}
 Using Atiyah--Hirzebruch--Riemann--Roch \cite{dold}, we get 
\begin{align*}
      \textnormal{ch}_i\big(\pi_{2\,*}(\pi^*_X(\alpha)\otimes \mathfrak{E})\big) = \sum_{v\in B_{\textnormal{even}}}\int_X \textnormal{ch}(\alpha)\textnormal{ch}(v)\textnormal{Td}(X)\boxtimes \mu_{v,i} = \sum_{v\in B_{\textnormal{even}}}\chi(\alpha^\vee,v)\mu_{v,i}\,.
\end{align*}
Taking a product with $\textnormal{ch}(\mathfrak{U}^\vee)$ and using $\beta_j=\textnormal{ch}_j(\mathfrak{U})$, we get
\begin{equation}
  \textnormal{ch}_j\big(\mathscr{T}(\alpha)\big) = \sum_{\begin{subarray}a v\in B_{\textnormal{even}}\\
  j=l+k
  \end{subarray}}(-1)^l\chi(\alpha^\vee,v)\beta_l\boxtimes\mu_{v,k}\,.  
\end{equation}
A similar explicit computation leads to the second formula. Let us, therefore, address the final statement.

\setcounter{footnote}{0}
Let $\mathscr{P}:\Hilb\to \mathcal{M}_X$ map $[\mathcal{O}_X\to F]$ to $[\mathcal{O}_X[1]\oplus F]$ and $\mathcal{E}\textnormal{xt}_n=\underline{\textnormal{Hom}}_{\Hilb}\big(\mathscr{I}_n,\mathscr{I}_n\Big)$. We have the following $\A^1$-homotopy commutative diagram\footnote{See \protect\cite[§3.2]{Blanc} for the notion of $\A^1$-homotopy for simplicial presheaves, but here we only mean that there is an $\A^1$-family of maps containing both paths of the diagram.}:
\begin{equation}
\label{A^1commutative}
    \begin{tikzcd}
    \Hilb\arrow[rd,"\mathcal{E}\textnormal{xt}_n"] \arrow[r,"\mathscr{P}"]&\mathcal{M}_X\arrow[d
    ,"\mathcal{E}\textnormal{xt}"]\\
    &\textnormal{Perf}_{\CC}
    \end{tikzcd}\,,
\end{equation}
where $\mathcal{E}\textnormal{xt},  \mathcal{E}{\textnormal{xt}}_{n}$ are the maps associated to the perfect complexes of the same name and $\mathcal{E}\textnormal{xt} = \underline{\textnormal{Hom}}_{\mathcal{M}_X}(\mathcal{U},\mathcal{U})$ for the universal perfect complex $\mathcal{U}\to X\times \mathcal{M}_X$. From Definition \ref{thetacom}, we easily deduce $\iota_n^*\circ\Delta^*\big(\Theta^{\textnormal{pa}}\big) =\mathscr{P}^*\mathcal{E}\textnormal{xt}^\vee$. Taking topological realization of \eqref{A^1commutative}, we obtain that 
$$\llbracket\mathcal{E}\textnormal{xt}_n\rrbracket =(\mathscr{P}^{\textnormal{top}})^*\llbracket\mathcal{E}\textnormal{xt}\rrbracket =\iota_n^*\llbracket\Delta^*\big(\Theta^{\textnormal{pa}}\big)^\vee\rrbracket = (\Omega\circ \iota_n)^*\big(\Delta^*(\theta_{\mathcal{P}})\big)^\vee \,.$$
Finally, we use $\textnormal{rk}\big((\mathcal{E}\textnormal{xt}_n)_0\big) = \textnormal{rk}(\mathcal{E}\textnormal{xt}_n)-2$.
\end{proof}

\begin{remark}
\label{remark H}
We explain now how to drop the condition \eqref{eqAsH}.
Going through the arguments in Theorem \ref{theorem workhorse} and Theorem \ref{theoremhilb} without the assumption \eqref{eqAsH}, one can check that under the projection $\Pi_{\textnormal{even}}:\check{H}_*(\mathcal{P}_X)\to \check{H}_{\textnormal{even}}(\mathcal{P}_X)$ we still obtain the same results. This is sufficient for us, because we never integrate odd cohomology classes, except when integrating polynomials in $\textnormal{ch}_k(T^{\textnormal{vir}}_n)$, but as the only terms $\mu_{v,k}$ contained in $\textnormal{ch}_k(T^{\textnormal{vir}})$  for $v\in B_{\textnormal{odd}}$ are given for $v\in B_7$, each such integral will contain a factor of $\chi^-(v,w)=0$ for $v,w\in B_7$. Using \ref{remark b1}, we may therefore now assume $K^1(X)=0$ in computations.
\end{remark}

To simplify notation, we will not write $\mathfrak{p}$ unless necessary, and we will use $f(\alpha^{[n]})$ and $f(T^{\textnormal{vir}}_n)$ to denote $\mathscr{G}_f(\alpha^{[n]})$ and $\mathscr{G}_f(T^{\vir}_n)$.

\begin{lemma}
\label{propmain}
 Let $f\in \Q\llbracket \mathfrak{p}, z\rrbracket$ be a power-series satisfying $f(0,0)=1$, then
 \begin{align*}
    \int_{[\textnormal{Hilb}^n(X)]^{\textnormal{vir}}} f(\alpha^{[n]}) &= \int_{\mathscr{H}_n} \textnormal{exp}\Big[\sum_{\begin{subarray}a k\geq 0\\
 v\in B_{6,8}\end{subarray}}a_\alpha(k)\chi(\alpha^\vee,v)\mu_{v,k}\Big]\,,\quad \textnormal{where}\\
 A_\alpha(z) &=\sum_{k\geq 0}a_\alpha(k)\frac{z^k}{k!} = \textnormal{log}\big(f(z)\big)\,, \\
    \int_{[\textnormal{Hilb}^n(X)]^{\textnormal{vir}}} f(T^{\textnormal{vir}}_n) &= \int_{\mathscr{H}_n} \textnormal{exp}\Big[\sum_{\begin{subarray}a k\geq 0\\
 v\in B_{6,8}\end{subarray}}a_{\llbracket \mathcal{O}_X\rrbracket}(k)\chi(v)\mu_{v,k}\Big]\,,\quad \textnormal{where}\\
 A_{\llbracket \mathcal{O}_X\rrbracket}(z) &=\sum_{k\geq 0}a_{\llbracket \mathcal{O}_X\rrbracket}(k)\frac{z^k}{k!} = \textnormal{log}\big(f(z)f(-z)\big)\,.
 \end{align*}
\end{lemma}

 \begin{proof}
We show that in the action of $\textnormal{ch}_k(\theta_{\mathcal{P}}^\vee)$ from Lemma \ref{lemma chern} on $\mathscr{H}_n$ only terms linear in $\mu_{v,k}$, $k>0$ will have non-trivial contributions.
\sloppy In Remark \ref{remark H}, we set $K^1(X)=0$, thus we only need to look at $\sum_{\begin{subarray}a i+j=k
  \\
  \sigma,\tau\in\mathbb{B}\backslash B_{\textnormal{odd}} \end{subarray}}(-1)^i\tilde{\chi}(\sigma,\tau)\mu_{\sigma,i}\boxtimes \mu_{\tau,j}$ and we claim it reduces to the action by
  \begin{equation}
\label{reducedchtheta}
  -\sum_{v\in B_{6,8}}\big(1+(-1)^k\big)\chi(v)\mu_{v,k} =  -\big(1+(-1)^k\big)\mu_{p,k}
\end{equation}
For $i,j>0$ if $\sigma = (0,1)$ or $\tau = (0,1)$, then due to Remark \ref{remark nob} this term vanishes. If $\sigma=(v,0)$, $\tau = (w,0)$ then $v,w\in B_{6,8}$ and $\tilde{\chi}(\sigma,\tau) = \chi(v,w) = 0$. So consider the case $i=0$, then $j=k>0$. If $\sigma=(v,0)$, $\tau = (w,0)$ or $\tau =(0,1)$ then the term is again 0, because $\mu_{v,0} = np(v^\vee) = 0$ unless $v=p$ in which case $\chi(v,w)=0$ for each $w\in B_{6,8}$. However, if $\sigma=(0,1)$, $\tau = (v,0)$, then $\mu_{\sigma,0}= 1$ and $\tilde{\chi}((0,1),(v,0))= -\chi(v)$. If $j=0$, then the same applies, thus the statement follows.

Let $E$ be a vector bundle  with $c(E) = \prod_{i=1}^a(1+x_i)$, then we write 
$$
f(E) = \prod_{i=1}^af(x_i) =  \exp\Big[\sum_{n>0}g_n\sum_{i=1}^a\frac{x_i^n}{n!}\Big]\,,
$$
where $\sum_{n>0}\frac{g_n}{n!}x^n =\log \big(f(x)\big) $ and $\sum_{i=1}^a\frac{x_i^n}{n!} = \textnormal{ch}_i(E)$. This extends to any class $\alpha\in G^0\big(\Hilb,\Q\big)$, so we get after using Remark \ref{remark nob}, \eqref{reducedchtheta} and Lemma \eqref{lemma chern} that
\begin{align*}
  \int_{\big[\textnormal{Hilb}^n(X)\big]^{\textnormal{vir}}}f(\alpha^{[n]}) &=\int_{\mathscr{H}_n} \textnormal{exp}\Big[\sum_{\begin{subarray}ak\geq 0\\ v\in B_{6,8}\end{subarray}}g_k\chi(\alpha^\vee,v)\mu_{v,k}\Big]\,, \\
    \int_{\big[\textnormal{Hilb}^n(X)\big]^{\textnormal{vir}}}f(T^{\textnormal{vir}}_n) &=\int_{\mathscr{H}_n} \textnormal{exp}\Big[\sum_{\begin{subarray}a k\geq 0\\ v\in B_{6,8}\end{subarray}}g_k\big(1+(-1)^k\big)\chi(v)\mu_{v,k}\Big]\,,
\end{align*}
where, we use $\mu_{p,0}= n$ and $\textnormal{rk}(T^{\textnormal{vir}}_n) = 2n$.
\sloppy From this we immediately see $A_{\alpha}(z)= \textnormal{log}\big(f(z)\big)$ and $A_{\llbracket\mathcal{O}_X\rrbracket}(z)=\textnormal{log}\big(f(z)f(-z)\big)$.
\end{proof}
As an immediate Corollary of \eqref{reducedchtheta}, we obtain the following:
\begin{corollary}
\label{cor vanishing}
For each $n>0$ let  $p_n(x_1t,x_2t^2,\ldots)$ be a formal power-series in infinitely many variables, then
$$
\int_{[\Hilb]^\textnormal{vir}}p_n(\textnormal{ch}_1(T^{\textnormal{vir}}_n),\textnormal{ch}_2(T^{\textnormal{vir}}_n),\ldots) = 0\,.
$$
\end{corollary}
\begin{proof}
We use
$$
\int_{\big[\Hilb\big]^\textnormal{vir}}p_n(\textnormal{ch}_1(T^{\textnormal{vir}}_n),\textnormal{ch}_2(T^{\textnormal{vir}}_n),\ldots) = \int_{\mathscr{H}_n}\tilde{p}_n(\mu_{p,1},\mu_{p,2},\ldots)\,, 
$$
where we use some new formal power-series $\tilde{p}(x_1t,x_2t^2\,,\ldots)$ given by \eqref{reducedchtheta}. Because each term in \eqref{Hq} contains at least one factor of the form $\mu_{v,k}$ for $v\in B_6$, $k>0$, the above integral is zero by \eqref{normalization}.
\end{proof}
We introduce the following definition as it will be also used in the sequels to the current paper \cite{bojkoquot}, \cite{ bojkosur}.
\begin{definition}
Let us define the \textit{universal transformation} $U$ of formal power-series  $U: \big(R\llbracket t\rrbracket\big)_1\to \big(R\llbracket t\rrbracket\big)_1$ by 
\begin{equation}
\label{Qmap}
 f(t)\mapsto \prod_{n>0}\prod_{k=1}^nf(-e^{\frac{2\pi i k}{n}}t)^{n}\,.
\end{equation}
Moreover, we will use the notation 
$$\{f\}(t) = f(t)f(-t)\,.$$
\end{definition}
\begin{remark}
\label{rem:Uandbracket}
\begin{enumerate}
    \item In fact, $U$ is a well-defined bijection. To see that it is well-defined on $\big(R\llbracket t\rrbracket\big)_1$, note that
$\textnormal{log}\Big(\prod_{k=1}^nf(-e^{\frac{2\pi i k}{n}}t)^{n}\Big) = \sum_{m>0}n^2f_{nm}t^{nm}$ by Knuth \cite[eq. (13), p. 89]{Knuth}  where $\textnormal{log}\big(f(t)\big) = \sum_{i>0}f_it^i$. Therefore $\prod_{k=1}^nf(-e^{\frac{2\pi i k}{n}}t)^{n} =1+O(t^n)$ which is precisely the condition necessary for the infinite product to be well-defined. Because of the above description, we also see that the different roots of unity cancel out so that the resulting power series has coefficients again in $R$. 

To construct the inverse, we observe that the logarithm of \eqref{Qmap} becomes
$\sum_{n>0}\sum_{m>0}n^2f_{nm}q^{nm}=\sum_{n>0}\sum_{l|n}\frac{n^2}{l^2}f_nq^n$. This corresponds to acting with a diagonal invertible matrix on the coefficients $f_n$, so we have an inverse. 
\item The definition of $\{f\}$ is deeply rooted in the origin of DT4 invariants and their definition. When $M$ is a moduli space of stable sheaves on a Calabi--Yau fourfold, its virtual tangent bundle $T^{\vir}_M$ (which for the present purposes should be thought of as just a complex K-theory class) is self-dual. Describing the background only approximately: Borisov--Joyce \cite{BJ} use this self-duality to construct a splitting 
$$
T^{\vir}_M = T^{\vir,+}_M + T^{\vir,-}_M
$$
where $T^{\vir,\pm}_M$ are real K-theory classes of equal rank. A choice of orientations from Theorem \ref{theorem CGJ} replaces the real K-theory $KO^0(M)$ by the real oriented K-theory $KSO^0(M)$.  Oh--Thomas \cite{OT} translated this into taking a purely complex splitting (after lifting $T^\vir_M$ to some isotropic flag bundle over $M$) 
$$
T^{\vir}_M = \Lambda +\Lambda^\vee
$$
where this time $\Lambda$ is in the complex K-theory $K^0(M)$. In fact, when viewed as classes in $KSO^0(M)$, $\Lambda = T^{\vir,+}_M$ if the orientations are chosen in a compatible way. This alternative splitting is what we observe on the level of cohomology when we are computing the second statement of Lemma \ref{propmain} because the contribution of $\ch(T^\vir)$ looks exactly like $\ch(\Lambda)+\ch(\Lambda^\vee)$ when we integrate it against $\mathscr{H}_n$ in \eqref{reducedchtheta}. From the perspective of the computation, it replaces $f$ by $\{f\}$ for the multiplicative genus acting on $T^\vir$. 

Coming back to the real oriented perspective, we will explain later in Remark \ref{rem:realgenus} that there is a notion of real oriented multiplicative genera and that 
$$
f\mapsto \{f\}
$$
is the most natural transformation mapping a complex multiplicative genus to a real oriented one. This will motivate working with real oriented genera over complex ones when defining generalizations of the K-theoretic invariants in §\ref{k-theoretic paragraph}.
\end{enumerate}
\end{remark}
\begin{example}
Acting with $U^{-1}$ on the MacMahon function $M(-q)$, we obtain $\frac{1}{(1+q)}$.  
\end{example}
We will need later the following generalization of the Lagrange inversion theorem:
\begin{lemma}
\label{Gessel-Lagrange}
Let $Q(t)\in R\llbracket t\rrbracket$  (with a non-vanishing constant term) and $g_i(x)$ for $i=1,\ldots,N$ be the different solutions to \begin{equation}
\label{solutionto}\big(g_i(x)\big)^N=xQ\big(g_i(x)\big)\end{equation}
then for any formal power-series $\phi(t)$, $\phi(0)=0$ we have
$$
\sum_{k=1}^N\phi\big(g_i(x)\big) =\sum_{n>0}\frac{1}{n} [t^{nN-1}]\Big(\phi'(t)Q(t)^n\Big)x^n\,.
$$
\end{lemma}
\begin{proof}
The usual Lagrange formula (see e.g. Gessel \cite[Thm. 2.1.1]{Gessel}) tells us that for $h(x) = xQ\big(h(x)\big)$, we have $
[x^n]\phi\big(h(x)\big) =\frac{1}{n} [t^{n-1}]\phi'(t)Q(t)^n\quad  \textnormal{for}\quad n>0\,.
$
Taking the unique Newton--Puiseux series satisfying $g(x^{\frac{1}{N}})=x^{\frac{1}{N}}Q^{\frac{1}{N}}\big(g(x^{\frac{1}{N}})\big)$ for a fixed $N$'th root of $Q$, we can write by Weierstrass preparation theorem together with the Newton--Puiseux theorem (see e.g. \cite[Chap. 3.2, Chap. 5.1,]{JohnPfi} every solution of \eqref{solutionto} by $g_k(x)=g(e^{\frac{2\pi ki}{N}}x^{\frac{1}{N}})$.  We obtain \begin{align*} \sum_{k=1}^N\phi\big(g_k(x)\big)&=\sum_{k=1}^N\sum_{n>0}\frac{1}{n}[t^{n-1}]\big(\phi'(t)Q^{\frac{n}{N}}(t)\big)\Big(e^{\frac{2\pi ik}{N}}x^{\frac{1}{N}}\Big)^n \\&=\sum_{n>0}\frac{1}{n}[t^{nN-1}]\Big(\phi'(t)Q^n(t)\Big)x^n\,.\end{align*}
\end{proof}

We prove now the main result that we will use throughout the next section.

\begin{proposition}
\label{mega theorem}
Let $f_0(\mathfrak{p},\cdot), f_1(\mathfrak{p},\cdot),\ldots, f_M(\mathfrak{p},\cdot)$ be power-series with $f(0,0)=1$, then define 
\begin{align*}
   \textnormal{Inv}(\vec{f}, \vec{\alpha},q)=1+\sum_{n>0}\int_{\big[\Hilb\big]^{\textnormal{vir}}}f_0(T^{\textnormal{vir}}_n)f_1(\alpha^{[n]}_1)\cdots f_M(\alpha^{[n]}_M) q^n\,,
\end{align*}
where $\vec{(-)}$ is meant to represent a vector, and we omit the additional variables. Then setting $\textnormal{rk}(\alpha_i)=a_i$, we have
\begin{align*}
\label{generalformula}
      \textnormal{Inv}(\vec{f},\vec{a},q)& =U\bigg\{ \bigg[\prod_{i=1}^M\frac{f_i( H(q))}{f_i(0)}\bigg]^{c_1(\alpha_i)\cdot c_3(X)}\bigg\}\,,
      \numberthis
\end{align*}
where $H(q)$ is the unique solution for

$$
      q=\frac{H(q)}{\prod_{j=1}^M f_j^{a_j}\big(H(q)\big)\,\{f_0\}\big(H(q)\big)}\,.
$$
\end{proposition}
\begin{proof}
Combining Lemma \ref{propmain} with Proposition \ref{theorem big}, we obtain 
\begin{align*}
 \textnormal{Inv}(q)=\prod_{i=1}^M\textnormal{exp}\Big\{\sum_{n>0}\sum_{l|n}\frac{n}{l^2}(-1)^n[z^{n-1}]\Big[\frac{d}{dz}\Big(\textnormal{log}\big(f_i(z)\big)-\textnormal{log}\big(f_i(0)\big)\Big)\\
 \prod_{j=1}^Mf_j(z)^{a_jn}\{f_0\}^n(z)\Big]q^n\Big\}^{c_1(\alpha_i)\cdot c_3(X)}    
\end{align*}
setting $\phi=\textnormal{log}(f_i) - \textnormal{log}\big(f_i(0)\big)$, $Q=\prod_{i=1}^M  f_i^{a_i}\{f_0\}$ and using Lemma \ref{Gessel-Lagrange}, this gives
\begin{align*}
  & \prod_{i=1}^M\textnormal{exp}\Big\{\sum_{n>0}\sum_{l|n}\frac{n^2}{l^2}[t^n]\Big[\textnormal{log}\,f_i\big(H(t)\big)-\textnormal{log}\,f_i(0)\Big](-q)^n\Big\}^{c_1(\alpha_i)\cdot c_3(X)}\\
 & =\prod_{i=1}^M\textnormal{exp}\Big\{\sum_{n,m>0}n\sum_{k=1}^n[t^{m}]\Big[\textnormal{log}\,f_i\big(H(t)\big)-\textnormal{log}\,f(0)\Big](-e^{\frac{2\pi i k}{n}}q)^m\Big\}^{c_1(\alpha_i)\cdot c_3(X)}
  \\
     &=\prod_{i=1}^M\prod_{n>0}\prod_{k=1}^n\Big[\frac{f_i\big(H(-e^{\frac{2\pi ik}{n}}q)\big)}{f_i(0)}\Big]^{n c_1(\alpha_i)\cdot c_3(X)} = \prod_{i=1}^MU\Big[\frac{f_i\big(H(q)\big)}{f_i(0)}\Big]^{c_1(\alpha_i)\cdot c_3(X)}\,,
\end{align*}
where $H(q)$ is the solutions of \eqref{generalformula}.
\end{proof}

 \section{Computing the generating series of newly defined invariants}
 \label{sec:invariants}
We define and compute many new invariants using the formula derived in the previous section. These include tautological series,  Nekrasov genera and virtual Verlinde numbers. The tautological series and the Verlinde series were inspired by their analogs on surfaces that we recalled in the introduction. Unlike the surface case, the definition of DT4 Verlinde invariants requires an additional twist by a square-root line bundle that we explain in §\ref{sect verlinde}. We also argue that the appropriate generalization of the Nekrasov genus is obtained in terms of the real-oriented Witten genus due to the connection of DT4 invariants to real geometry. For the above invariants, we additionally study their symmetries and their relation to lower-dimensional geometries.

In particular, we obtain an explicit correspondence via the universal transformation $U$ between virtual Donaldson invariants on surfaces satisfying $c_1(S)^2=0$ and DT4 invariants on projective Calabi--Yau fourfolds. The virtual invariants for surfaces are the integrals against the virtual fundamental classes of Quot schemes $[\textnormal{Quot}_S(\mathbb{C}^N,\beta,n)]^{\textnormal{vir}}$. These classes have been used by Marian--Oprea--Pandharipande \cite{MOP1} to prove Lehn's conjecture \cite{Lehn} for the generating series of tautological invariants on Hilbert schemes of points. More recently their virtual fundamental classes  were studied by AJLOP \cite{AJLOP}, Lim \cite{Lim}, Johnson--Oprea--Pandharipande \cite{JOP} and Oprea--Pandharipande \cite{OP1}.  Segre--Verlinde correspondence for the virtual geometry of Quot-schemes was investigated in Oprea--Pandharipande \cite{OP1} and AJLOP \cite{AJLOP}. Combined with the 4D-2D correspondence, their result could be used to deduce Segre--Verlinde correspondence for Calabi--Yau fourfolds which is an entirely new phenomenon. As we found the fourfold version independently and were only made aware of the similarity to the Quot-scheme case later, we prefer to present our original computation first. 

Our goal in the final subsection is to recover the formulae of \cite{AJLOP} for integrals over Quot schemes on surfaces $S$ satisfying $c_1^2(S)$ when $\beta=0$. We do this to explain the 4D-2D correspondence as a consequence of both theories being governed by the same wall-crossing formula. Further study of invariants on degree 0 Quot schemes for surfaces and Calabi--Yau fourfolds is left to \cite{bojkoquot} where we no longer restrict to trivial vector bundles for the Quot-schemes.
\subsection{Segre series}
\label{sectalltaut}
Setting $f_0=1$ and $f_i=(1+t_ix)^{-1}$ in Proposition \ref{mega theorem}, we obtain the \textit{generalized $\textnormal{DT}_4$-Segre series}
$$
R(\vec{\alpha},\vec{t}; q) =1+\sum_{n>0}q^n\int_{[\textnormal{Hilb}^n(X)]^{\textnormal{vir}}}s_{t_1}(\alpha_1^{[n]})\cdots s_{t_M}(\alpha_M^{[n]})\,.
$$
Recall our notation for generating series of  Fuss-Catalan numbers from \eqref{fusscatalan series}.
\begin{theorem}
\label{theorem segre}
Let $\alpha_1,\ldots,\alpha_M \in G^0(X)$, $a=\textnormal{rk}(\alpha)$, then assuming Conjecture \ref{conjecture WC} for point-canonical orientations we have
$$
R(\vec{\alpha},\vec{t}; q) = U\Big[(1+t_1z)^{-c_1(\alpha_1)\cdot c_3(X)}\cdots (1+t_Mz)^{c_1(\alpha_M)\cdot c_3(X)}\Big]^{-1}\,,
$$
where $z=z(q)$ is the solution to $z(1+t_1z)^{a_1}\cdots(1+t_Mz)^{a_M} = q$. Moreover, we have the explicit expression:
\begin{equation}
\label{segreseries}
R(\alpha;q) = 
\left\{\begin{array}{ll}
    \displaystyle  U\big[\mathscr{B}_{a+1}(-q)^{c_1(\alpha)\cdot c_3(X)}\big]  & \textnormal{for }a\geq 0 \\
    &\\
  \displaystyle    U\big[ \mathscr{B}_{-a}(q)^{-c_1(\alpha)\cdot c_3(X)}\big] &\textnormal{for }a< 0 \\
\end{array} 
\right.
  \end{equation}
  where $\mB_a(q)$ are the generating series of Fuss--Catalan numbers \eqref{fusscatalan}. 
\end{theorem}
\begin{proof}
The first statement follows immediately from Proposition \ref{mega theorem}.
Specializing to the \textit{$DT_4$ Segre series}
$$R(\alpha;q)=\int_{[\textnormal{Hilb}^n(X)]^{\textnormal{vir}}}s_n(\alpha^{[n]})$$
we obtain for $a=\textnormal{rk}(\alpha)$
$$
R(\alpha;q)=U\big[(1+z)\big]^{-c_1(\alpha)\cdot c_3(X)}\,,\quad \textnormal{where}\quad q=z(1+z)^{a}\,.
$$
The theorem then follows from the following lemma.
\begin{lemma}
\label{lemma change}
Let $y$ be the solutions of $y(y+1)^{a}=q$ for $a>0$, then 
\begin{equation}
  \frac{1}{1+y}=\left\{\begin{array}{ll}
    \displaystyle  \mathscr{B}_{a+1}(-q)  & \textnormal{for }a\geq 0 \\
    &\\
  \displaystyle     \mathscr{B}_{-a}(q)^{-1} &\textnormal{for }a< 0 \\
\end{array} 
\right.
 \,.  
\end{equation}

\end{lemma}
\begin{proof}
We use Lemma \ref{Gessel-Lagrange}. For $a\geq 0$ we change variables $z=1/(1+y)$ this implies $g(z):= (1-z)/z^{a+1} =q$. Then the statement follows from
\begin{align*}
  [(z-1)^{n-1}]\Big(\frac{(z-1)}{g(z)}\Big)^n &= [(z-1)^{n-1}](-1)^n\big(1+(z-1)\big)^{(a+1)n} \\&= \frac{(-1)^n}{n}{(a+1)n \choose n-1}= \frac{(-1)^n}{(a+1)n+1}{(a+1)n+1\choose n}\\&=(-1)^nC_{n,a}\,,
\end{align*}
where we used the notation from \eqref{fusscatalan}.
When $a<0$, then change variables by $(u+1)=(z+1)^{-1}$ to get $z=-u(u+1)^{-1}$ and thus $-u(u+1)^{-a-1} = q$. Then $(1+z)^{-1} =(1+u)=\mathscr{B}_{-a}(q)^{-1}$ by the above.
\end{proof}
Using this, we obtain \eqref{segreseries}.
\end{proof}

\label{sectnewins}
\subsection{K-theoretic insertions}
\label{k-theoretic paragraph}
In this section, we use the Riemann--Roch formula for the twisted virtual structure sheaf of Oh--Thomas \cite{OT}.
\begin{theorem}[Oh--Thomas {\cite[Thm. 6.1]{OT}}]
Let $M$ be projective with a fixed choice of orientation and let $V\in G^0\big(M\big)$. Then 
\begin{equation}
\label{VRROT}
 \hat{\chi}^{\textnormal{vir}}\big(M,V) = \int_{[M]^{\textnormal{vir}}}\sqrt{\textnormal{Td}}(\mathbb{E})\textnormal{ch}(V)\,.
\end{equation}
\end{theorem}
\begin{proof}
This is just Theorem \cite[Thm. 6.1]{OT} stated in terms of $\hat{\chi}^{\textnormal{vir}}(-)$ and using the notation of \eqref{chivir}.
\end{proof}

 Recall that $\textnormal{Td} = \mathscr{G}_f$ for $
 f(x) = \frac{x}{1-e^{-x}}$ which satisfies
\begin{equation}
\label{squareroottoddgenus}
  \{\sqrt{f}\}(x) = \frac{x}{e^{\frac{x}{2}}-e^{-{\frac{x}{2}}}}\,.   
\end{equation}
The above is called the $\hat{A}$-genus. An immediate consequence of Corollary \ref{cor vanishing} is
\begin{corollary}
\label{vanishingprop}
  For all $n>0$, $\hat{\chi}^{\textnormal{vir}}\big(\Hilb\big)=0$\,.
\end{corollary}

\textit{Nekrasov genus} 
\begin{equation}
\label{eq:Nekrasov}
\mathscr{N}_y\big(\alpha^{[n]}\big)=\Lambda^\bullet_{y^{-1}}\big(\alpha^{[n]}\big)\otimes \textnormal{det}^{-\frac{1}{2}}\big(\alpha^{[n]}\cdot y^{-1}\big) \in G^0(M_{\alpha,L},\Q)(\!(y^{-\frac{1}{2}})\!)
\end{equation}
determined by the power series
 $$\mathscr{N}_y(x)=(y^{\frac{1}{2}}e^{-\frac{x}{2}}-y^{-\frac{1}{2}}e^{\frac{x}{2}})$$
gives us refinements of invariants considered in §\ref{sectalltaut} as
$$
K_n(\vec{\alpha},\vec{y})=\hat{\chi}^{\textnormal{vir}}\Big(\mathscr{N}_{y_1}\big(\alpha_1^{[n]}\big)\cdots\mathscr{N}_{y_M}\big(\alpha_M^{[n]}\big)\Big)\,,\qquad K(\vec{\alpha},\vec{y};q) = \sum_{n>0}K_n(\vec{\alpha},\vec{y})q^n\,.
$$
 Note that $\mathscr{N}_0(0)\neq 1$, but we can write it as $\mathscr{N}_y(x)=(1-y^{-1}e^{x})e^{-\frac{x}{2}}y^{\frac{1}{2}}$ and simply keep track of $y^{\frac{1}{2}}$ separately.
The DT$_4$ Segre series of §\ref{sectalltaut} can be obtained as a \textit{classical limit} of these invariants. Explicitly this means the following:
\begin{proposition}
\label{proplimit}
For any $\alpha_1,\ldots,\alpha_M\in G^0(X)$ with $a_j=\textnormal{rk}(\alpha_j)$ and $i_j$, such that $\sum_{j}i_j=n$, we have
\begin{align*}
     \lim_{y_1\to 1^+}\cdots  \lim_{y_M\to 1^+}(1-y_1^{-1})^{i_1-a_1n}\cdots (1-y_M^{-1})^{i_M-a_Mn}K_n(\vec{\alpha},\vec{y}) =\\
   (-1)^n\int_{\big[\textnormal{Hilb}^n(X)\big]^{\textnormal{vir}}}c_{i_1}\big(\alpha^{[n]}_1\big)\cdots c_{i_M}\big(\alpha^{[n]}_M\big) \,.  
\end{align*}
\end{proposition}
\begin{proof}
We conclude it from a more general result for any scheme $S$.

Define  $A^{>m}=\bigoplus_{i> m}A^i(S,\Q)$ and $A^{\leq m}=A^*(S,\Q)/A^{>m}$. Let $\gamma\in G^0(S)$,  $a=\textnormal{rk}(\gamma)$ and $k\geq 0$, then we claim that in $ A^{\leq k}$ the following holds:
$$
\mathscr{L}_k(\gamma):=\lim_{y\to 1^+}(1-y^{-1})^{k-a}\Big[\textnormal{ch}\Big(\Lambda^{\bullet}_{y^{-1}}\gamma\cdot\textnormal{det}^{-\frac{1}{2}}(\gamma\cdot y^{-1})\Big)\Big] = (-1)^kc_k(\gamma) \,:
$$
Let $\gamma = \llbracket E\rrbracket - \llbracket F\rrbracket$ and $c(E) = \prod_{i=1}^e(1-x_i)$, $c(F) = \prod_{i=1}^f(1-z_i)$, then in $A^{\leq l}$, we have
\begin{align*}
  \mathscr{L}_l(E)& = \lim_{\lambda\to 0^+}\Big[(1-e^{-\lambda})^{l-e}\prod_{i=1}^e(e^{\frac{\lambda}{2}-\frac{x_i}{2}}-e^{-\frac{\lambda}{2}+\frac{x_i}{2}})\Big] \\
  &= \lim_{\lambda\to 0^+}\Big[(\lambda - O(\lambda^2))^{l-e}\prod_{i=1}^e((\lambda-x_i)+O\big((\lambda -x_i)^3\big)\Big]=[\lambda^{-l}]\prod_{i=1}^e(1-\lambda^{-1}x_i)\\
  &=(-1)^lc_l(\llbracket E\rrbracket) \,.
\end{align*}
Similarly, we obtain in $A^{\leq m}$
\begin{align*}
 &\mathscr{L}_m(-\llbracket F\rrbracket) = \lim_{\lambda\to 0^+}\Big[(1-e^{-\lambda})^{m+f}\prod_{j=1}^f(e^{\frac{\lambda}{2}-\frac{z_j}{2}}-e^{-\frac{\lambda}{2}+\frac{z_j}{2}})^{-1}\Big] = [\lambda^{-m}]\prod^f_{j=1}(1-\lambda^{-1} z_j)^{-1}\\&=(-1)^mc_m(-\llbracket F\rrbracket)\,.  
\end{align*}
 We combine these two to obtain $\mathscr{L}_k(\gamma) = \sum_{l+m=k}\mathscr{L}_l(E)\mathscr{L}_m(-\llbracket F\rrbracket) = (-1)^kc_k(\gamma)$, where both equalities are true only in $A^{\leq k}$. 

To conclude the proof, we apply this statement to each $\mathscr{N}_{y_i}\big(\alpha_i^{[n]}\big)$ separately. Then using $\sum_{j}i_j=n$ we see from Theorem \ref{VRROT} that we are integrating $c_{i_1}\big(\alpha_1^{[n]}\big)\ldots c_{i_M}\big(\alpha_M^{[n]}\big)\sqrt{\textnormal{Td}}_0\big(T^{\textnormal{vir}}_n\big)$.
\end{proof}

Only the case where $\sum_{j}\textnormal{rk}(\alpha_j)=2b+1$ for $b\in \ZZ$ is going to be addressed below, because, under the additional assumption that $c_1(\alpha_i)$ are divisible by two, integrality holds. It is clear from computations that when either of the two conditions is violated, one will in general obtain coefficients contained in $\ZZ[1/2]$. Moreover, we mostly focus on $M=1$ and $\textnormal{rk}(\alpha_1)=1$ which is motivated by the work of Nekrasov \cite{NMF}, Nekrasov--Piazzalunga \cite{NP} and Cao--Kool--Monavari \cite{CKM}. 
 
\begin{theorem}
\label{theorem nekrasov}
If Conjecture \ref{conjecture WC} holds, then for all $\alpha_1,\ldots,\alpha_M$ with $a_i=\textnormal{rk}(\alpha_i)$,  $\sum_{i}a_i=2b+1$ and point-canonical orientations, we have
\begin{align*}
K(\vec{\alpha},\vec{y};q)&=\prod_{i=1}^MU\bigg[\frac{(y_i-1)^2u}{(y_i-u)^2}\bigg]^{\frac{1}{2}c_1(\alpha_i)\cdot c_3(X)}\,,\textnormal{\quad where}\quad
q=\frac{(u-1)u^b}{\prod_{j=1}^M (y^{\frac{1}{2}}-y^{-\frac{1}{2}}u)^{a_i}}\,.
\end{align*}
When $M=1$, $\alpha_1=\alpha$, $y_1=y$, $a_1=1$, then
$$
 K(\alpha,y; q) =\textnormal{Exp}\bigg[\chi\Big(X,q\frac{(TX-T^*X)(\alpha^{\frac{1}{2}}y^{\frac{1}{2}}-\alpha^{-\frac{1}{2}}y^{-\frac{1}{2}}}{(1-q\alpha^{\frac{1}{2}}y^{\frac{1}{2}})(1-q\alpha^{-\frac{1}{2}}y^{-\frac{1}{2}})}\Big)\bigg]\,,
$$
where $\textnormal{Exp}[f(y,q)] = \textnormal{exp}\Big[\sum_{n>0}\frac{f(y^n,q^n)}{n}\Big]$. In particular, the coefficients of $K(\vec{\alpha},\vec{y};q)$ lie in $\ZZ[y^{\pm\frac{1}{2}}_1,\ldots, y^{\pm\frac{1}{2}}_M]$ if $\frac{c_1(\alpha_i)}{2}\in H_2(X,\ZZ)$.
\end{theorem}
\begin{proof}
Using Proposition \ref{mega theorem} together with \eqref{squareroottoddgenus} and Theorem \ref{VRROT}, we obtain
$$
K(\vec{\alpha},\vec{y};q) =\prod_{i=1}^MU\bigg[\frac{y^{\frac{1}{2}}-y^{-\frac{1}{2}}}{y_i^{\frac{1}{2}}e^{-\frac{z_i}{2}}-y^{\frac{1}{2}}_ie^{\frac{z_i}{2}}}\bigg]^{c_1(\alpha_i)\cdot c_3(X)}\,,\quad\textnormal{where}\,\, q=\frac{e^{\frac{z_i}{2}}-e^{-\frac{z_i}{2}}}{\prod_{i=1}^M(y^{\frac{1}{2}}e^{-\frac{z_i}{2}}-y^{-\frac{1}{2}}e^{\frac{z_i}{2}})}\,,
$$
setting $u=e^{z}$ and using
$$
\sqrt{\frac{(1-uy_i^{-1})^2}{(1-y_i^{-1})^2u}} = \frac{\big(y^{\frac{1}{2}}u^{-\frac{1}{2}}-y^{-\frac{1}{2}}u^{\frac{1}{2}}\big)}{y^{\frac{1}{2}}-y^{-\frac{1}{2}}}
$$
we obtain 
$$
K(\vec{\alpha},\vec{y};q)=\prod_{i=1}^MU\bigg[\frac{(y_i-1)^2u}{(y_i-u)^2}\bigg]^{\frac{1}{2}c_1(\alpha_i)\cdot c_3(X)}\,,\textnormal{\quad where}\quad
q=\frac{(u-1)u^b}{\prod_{j=1}^M (y^{\frac{1}{2}}-y^{-\frac{1}{2}}u)^{a_j}}\,.
$$
When $M=1$ and $a_1=1$ then this gives 
$$
u = \frac{1+qy^{\frac{1}{2}}}{1+qy^{-\frac{1}{2}}} 
$$
which after plugging into the above formula gives rise to
\begin{align*}
K(\alpha,y;q) &= U\Big[\big(1+qy^{\frac{1}{2}}\big)\big(1+qy^{-\frac{1}{2}}\big)\Big]^{\frac{1}{2}c_1(\alpha)\cdot c_3(X)} \\
&= \sqrt{M(qy^{\frac{1}{2}})M(qy^{-\frac{1}{2}})}^{c_1(\alpha)\cdot c_3(X)}\\
   &=\textnormal{Exp}\Big[\frac{qy^{\frac{1}{2}}}{(1-qy^{\frac{1}{2}})^2}-\frac{qy^{-\frac{1}{2}}}{(1-qy^{-\frac{1}{2}})^2}\Big]^{-\frac{1}{2}c_1(\alpha)\cdot c_3(X)}\\
   &=  \textnormal{Exp}\bigg[\chi\Big(X,q\frac{(TX-T^*X)(\alpha^{\frac{1}{2}}y^{\frac{1}{2}}-\alpha^{-\frac{1}{2}}y^{-\frac{1}{2}}}{(1-q\alpha^{\frac{1}{2}}y^{\frac{1}{2}})(1-q\alpha^{-\frac{1}{2}}y^{-\frac{1}{2}})}\Big)\bigg]\,,
\end{align*}
where the second equality uses $M(q)=\textnormal{Exp}[\frac{q}{(1-q)^2}]$ and the last equality uses Grothendieck--Riemann--Roch.
\end{proof}
Now that we have computed the pairing between the Nekrasov genus and the twisted virtual structure sheaf, we would like to generalize these invariants. This direction was motivated by Cao--Kool--Monavari \cite[Rem. 1.19]{CKM} who started studying this question. A naïve guess would be to study the $\chi_y$-genus or the elliptic genus of Fantechi--Göttsche \cite{GF}, as was attempted by the authors of loc. cit. in the case of the $\chi_y$-genus. We explain first why this is not the right choice in the next remark that expands on Remark \ref{rem:Uandbracket} and discusses multiplicative genera for different structure groups. 
\begin{remark}
\label{rem:realgenus}
The work of Borisov--Joyce \cite{BJ} formulated DT4 invariants as integrals over \underline{real oriented} derived manifolds where the data of orientations is equivalent to the notion recalled above Theorem \ref{theorem CGJ}. While the work of Oh--Thomas provides a far better language for algebraic geometers to work with these invariants, it hides the original perspective that was already present in the work of Donaldson--Thomas \cite[p. 5]{DT96} who worked with the \underline{real} $\pm 1$ eigenvalue subspaces of the Hodge star $*:\Omega^{0,2}\to \Omega^{0,2}$. 

Because the $\chi_y$ genus of some manifold $M$ requires the presence of Kähler structure on $M$ to be defined as 
$$
\chi_y(M) =\sum_{p,q}(-1)^{p+q} h^{p,q}(M)y^p\,,
$$
it would not make sense to define it for a general real oriented derived manifold. In the language of \cite[§2]{Hopkins} (which is only useful if the reader is already familiar with it), this states that $\chi_y$-genus is the multiplicative genus for the complex cobordism ring $MU_*$ where $U$ stands for the unitary groups $U(n)$. This issue is not fixed by introducing reasonable insertions as supported by multiple computations that I have done with different potential options. All of them produced invariants that were not integer and showed no clear pattern. 

At first sight, this may seem to contradict the correspondence to virtual invariants of Quot schemes on surfaces in §\ref{4d2d1dsection} where computing $\chi_y$ genus with insertions makes perfect sense. However, when going from surface invariants to Calabi--Yau fourfolds, we always replace $g(T^{\vir}_n)$ by $\{g\}(T^\vir)$ in Theorem \ref{4d2d1d}. 

A motivation for the operation $\{-\}$ was given in Remark \ref{rem:Uandbracket} on the level of K-theory/cohomology and we already alluded to its connection to different types of structure groups. Here we explain that such symmetrizations appear because we are transitioning from multiplicative genera for $MU_*$ to multiplicative genera for the real oriented cobordism ring $MSO_*$ (see again \cite[§2]{Hopkins}). The latter are determined by power series that satisfy
$$
f(z) = f(-z)\,,
$$
so $f=\{g\}$ presents a class of examples. This modification is a consequence of complexifying the real tangent bundle of a real manifold $M$: $TM\otimes \CC$ can naturally be split into real vector bundles as
$$
TM\otimes \CC = TM^+\oplus TM^{-}
$$
where $\pm$ symbolize the 
$\pm 1$ eigenvalues of the natural complex conjugation. Up to lifting the problem along isotropic flag bundles just like in \cite[3.1]{OT}, there is also a splitting 
$$
TM\otimes \CC =\Lambda\oplus \Lambda^*
$$
for an isotropic subbundle $\Lambda\subset TM\otimes \CC$. We then see that 
$$
f(TM\otimes \CC)= f(\Lambda)f(\Lambda^*)
$$
explaining the $\{-\}$ symmetrization operation. This is parallel to what happened in Remark \ref{rem:Uandbracket} in K-theory of $\Hilb$.
\end{remark}
Based on the above remark, we propose a more appropriate set of invariants to study that are determined by multiplicative genera of the oriented real cobordism ring $MSO_*$. Examples of these invariants are the $\hat{A}$-genus, the $L$-genus, and the Witten genus. We already noticed that the twisted virtual structure sheaf of Oh--Thomas leads to the $\hat{A}$-genus in \eqref{squareroottoddgenus}. As its generalization, we focus on the Witten genus which we will relate to previously studied invariants of 3-folds. 

On a real oriented manifold $M$ of real dimension $m$, the Witten genus paired with $V\in K^0(M)$ is defined as the integral
$$
W(M,V,q) = \int_{[M]}\hat{A}(M)\ch(V)\prod_{k\geq 1}\textnormal{ch}\big(\textnormal{Sym}_{-q^k}^\bullet(TM\otimes \CC - m\cdot 1)\big)\,.
$$
It was originally defined by Witten in \cite{WittenDirac}. From Remark \ref{rem:Uandbracket} and Remark \ref{rem:realgenus}, it is evident that when $M$ is replaced by a moduli space of sheaves on a Calabi--Yau fourfold with the virtual tangent bundle $T^{\vir}_M$, one should also use $T^\vir_M$ instead of $TM\otimes \CC$. This warrants the following definition.
\begin{definition}
We define the \textit{DT$_4$ Witten-genus} of $M$ paired with $V\in K^0(M)$ by
$$
W(M,V,q) = \chi\Big(\hat{\mathcal{O}}^{\textnormal{vir}}\otimes V\cdot\bigotimes_{k\geq 1} \textnormal{Sym}_{-q^k}^{\bullet}(T^\vir_M-\textnormal{rk}_{\CC}(T^\vir_M)\cdot 1)\Big)
$$
whenever $\hat{\mO}^{\vir}$ is defined and the above Euler-characteristic\footnote{For a general $V$ in the complex topological K-theory, it is defined in terms of Chern characters and the virtual Riemann--Roch \protect\eqref{VRROT}.} exists.
\end{definition}

\begin{example}
Let $M$ be a moduli scheme with a perfect obstruction theory $\mathbb{F}^\bullet \xrightarrow{\textnormal{At}} \mathbb{L}_M$ as in Behrend--Fantechi \cite{BF}, then \cite{OT, CL, DYS} consider the 3-term obstruction theory $\mathbb{E}^\bullet =\mathbb{F}^\bullet \oplus (\mathbb{F}^\bullet)^\vee[2]\xrightarrow{(\textnormal{At},0)}\mathbb{L}_M$ of the $-2$\textit{-shifted cotangent bundle} .  In this situation, Oh--Thomas \cite[§8]{OT} show that
$$
\hat{\mathcal{O}}^{\textnormal{vir}} = \mathcal{O}^{\textnormal{vir}}\sqrt{K^{\textnormal{vir}}}\,,
$$
where $\mathcal{O}^{\textnormal{vir}}$ is the virtual structure sheaf of Fantechi--Göttsche \cite{GF}, $K^{\textnormal{vir}}=\textnormal{det}(\mathbb{F}^\bullet)$ and the square root is taken in $G^0(M,\mathbb{Z}[2^{-1}])$, where it always exists (see Oh--Thomas \cite[Lem. 2.1]{OT}). The term on the right hand side is in fact the \textit{twisted virtual structure sheaf} $\hat{\mathcal{O}}^{\textnormal{vir}}_{NO}$ of Nekrasov--Okounkov \cite{NO}. If $\textnormal{rk}(\mathbb{F}^\bullet) = 0$, i.e. the virtual dimension of $M$ is 0, then  
$$
W(M,V,q) = \chi\Big(\hat{\mathcal{O}}_{NO}^{\textnormal{vir}}\otimes \bigotimes_{k\geq 1}\textnormal{Sym}^{\bullet}_{-q^k}\big(\mathbb{F}^\bullet\oplus (\mathbb{F}^\bullet)^\vee\big)\otimes V\Big)\,,
$$
which is the \textit{virtual chiral elliptic genus} of Fasola--Monavari--Ricolfi \cite[Def. 8.1]{FMR} motivated by the work of physicists Benini--Bonelli--Poggi--Tanzini  \cite{Benini}. The discrepancy in signs in front of $q^k$ compared to \cite{FMR} comes from different conventions for $\Sym^\bullet_{(-)}$, but the end result is the same.

In the perfect scenario, one should work equivariantly to avoid the simplicity coming from rank 0. Such equivariant dimensional reduction was already considered in \cite{CKM2, KRdraft}, where it is only applied to the $\hat{A}$-genus.

For Hilbert schemes, it seems natural to study the \textit{Nekrasov--Witten genus}
$$W(\Hilb, \mathcal{N}_y(L^{[n]}),q)\,.$$
which generalizes Nekrasov's genus.  Using Proposition \ref{mega theorem}, the corresponding generating series can be expressed as 
$$
1+\sum_{n>0}z^nW\Big(\Hilb, \mathcal{N}_y(L^{[n]}),q\Big) =U\Big[ \frac{(y_i-1)^2u}{(y_i-u)^2}\Big]\,,
$$
where 
$$
z = \frac{u-1}{y^{\frac{1}{2}}-y^{-\frac{1}{2}}u}\prod_{k>0}\frac{(1-q^ku)(1-q^ku^{-1})}{(1-q^k)^2}\,.
$$
\end{example}
We leave the study of its properties for the future. 
\subsection{Untwisted K-theoretic invariants}
\label{sect verlinde}
 We propose a version of \textit{DT$_4$ Verlinde numbers} for Calabi--Yau fourfolds as higher dimensional analogs of Verlinde numbers for surfaces studied in \cite{EGL, MOPhigher, GotVer}. After computing generating series for these invariants, we obtain a simple Segre--Verlinde correspondence. 

 Due to the inherent twist by a square root line bundle in the construction of $\hat{\mO}^{\vir}$ in \cite[Def. 5.3]{OT}, the twisted virtual structure sheaf may lie in $G_0\big(\Hilb, \ZZ[1/2]\big)$ because such square-roots may not exist as integral classes. After pairing it with integral K-theory classes on $\Hilb$, we saw in our computations that this is indeed the case. In the definition of the Nekrasov genus as recalled in \eqref{eq:Nekrasov}, the square root line bundle $\det^{\frac{1}{2}}(L^{[n]})$ was used to fix this. This was also observed in \cite[Prop. 1.15]{CKM} for DT/PT-stable pairs.
 
 We follow this idea when defining K-theoretic invariants based on lower-dimensional classical ones (like the Verlinde series) in which case we choose $L=\mO_X$.
 \begin{definition}
\label{defuntwisted}
Let $E = \textnormal{det}\big(\mathcal{O}_X^{[n]}\big)$, then the \textit{untwisted virtual structure sheaf} is defined by
$$
\mathcal{O}^{\textnormal{uvir}} = \hat{\mathcal{O}}^{\textnormal{vir}}\otimes E^{\frac{1}{2}}.
$$
We define the \textit{untwisted virtual characteristic}
$$
\chi^{\textnormal{uvir}}\Big(\Hilb,A\Big) = \hat{\chi}^{\textnormal{vir}}\Big(\Hilb,E^{\frac{1}{2}}\otimes A\Big) = \int_{[\Hilb]^{\textnormal{vir}}}\sqrt{\textnormal{Td}}\big(T^{\textnormal{vir}}_n\big)\textnormal{ch}\big(E^{\frac{1}{2}}\big)\textnormal{ch}(A)\,.
$$
\end{definition}
 From the perspective of Lemma \ref{propmain} and Proposition \ref{mega theorem}, describing a large class of invariants, this choice is the simplest and the most natural one. The above modification changes $A_{\llbracket \mathcal{O}_X\rrbracket}(z)=z/(e^{\frac{z}{2}}-e^{-\frac{z}{2}})$ from Lemma \ref{propmain} to $$A_{\llbracket \mathcal{O}_X\rrbracket}(z)=z/(1-e^{-z})$$
 which is the generating series for the usual Todd genus. The reason for calling the resulting class $\mO^{\uvir}$  untwisted is given in \cite[5.2]{BH} -- the most compelling argument supporting Definition \ref{defuntwisted} thus far. There we work equivariantly with $X=\CC^4$ and we use Kool--Rennemo's \cite{KRdraft} that describes $\hat{\mO}^\vir$ on $\textnormal{Hilb}^n(\CC^4)$ explicitly. We see then that the only non-integral term that appears in the construction of $\hat{\mO}^\vir$ is the square root $E^{-\frac{1}{2}}$. Multiplying by $E^{\frac{1}{2}}$ then corresponds to removing the non-integral contribution. In the compact setting, such a picture is difficult to replicate even in the simplest example of 1 point. It may have been therefore better to call $\mO^{^\uvir}$ the \textit{untwisted twisted virtual structure sheaf} pointing out that we are not introducing a new and better definition of a virtual structure sheaf on top of \cite{OT}. But this is somewhat of a mouthful. 

 Because $\Hilb$ admits a natural derived refinement, there exists a virtual structure sheaf $\mO^{\vir}$ on it in the sense of \cite{KCF}. Unlike the virtual fundamental class $[\Hilb]^{\vir}$ and $\hat{\mO}^{\vir}$, it is constructed using the full self-dual obstruction theory, and it is not clear how it is related to either of them. The notation $\mO^{\uvir}$ was used to make a clear distinction between the two classes. 

 Next, we define the Verlinde series for Calabi--Yau fourfolds and its square-root version that is related to the Nekrasov genus.
\begin{definition}
Let $X$ be a Calabi--Yau fourfold, then its \textit{square root DT$_4$ Verlinde series} are defined for all  $\alpha\in G^0(X)$ by
$$
V^{\frac{1}{2}}(\alpha;q) = 1+\sum_{n>0}V^{\frac{1}{2}}_n(\alpha) q^n=1+\sum_{n>0}\hat{\chi};\big(\Hilb, \textnormal{det}^{\frac{1}{2}}(L^{[n]}_\alpha)\otimes E^a \big)q^n\,,
$$
where $L_\alpha=\textnormal{det}(\alpha)$, $a=\textnormal{rk}(\alpha)$. The \textit{DT$_4$ Verlinde series} is defined by
$$
V(\alpha;q)= 1+\sum_{n>0}V_n(\alpha)q^n=1+\sum_{n>0}\chi^{\textnormal{uvir}}\big(\Hilb, \textnormal{det}(\alpha^{[n]})\big)q^n\,.
$$
\end{definition}
\begin{remark}
\label{Verlinde-Nekrasov remark}
Just for the purpose of this remark, let us define \textit{negative square root Verlinde series} by
$$
V^{-\frac{1}{2}}(\alpha;q) = 1+\sum_{n>0}V^{-\frac{1}{2}}_n(\alpha) q^n=1+\sum_{n>0}\hat{\chi}^{\textnormal{uvir}}\big(\Hilb, \textnormal{det}^{-\frac{1}{2}}(L^{[n]}_\alpha)\otimes E^{-a} \big)q^n\,,
$$
for each $\alpha\in G^0(X)$, where $a=\textnormal{rk}(\alpha)$. 
\begin{enumerate}
    \item When $\alpha=\llbracket V\rrbracket$ is a vector bundle of rank $a$, one can show  that 
$$
    [y^{\mp\frac{n}{2}(2a+1)}]\Big(K_n(L_\alpha\oplus E^{\oplus 2a},y)\Big)=V^{\pm \frac{1}{2}}_n(V)\,.
$$

\item 
From the expression $K(L,y;q)=\sqrt{M(qy^{\frac{1}{2}})M(qy^{-\frac{1}{2}})}^{c_1(\alpha)\cdot c_3(X)}$, we obtain that Nekrasov generating series decouples into the positive and negative square-root Verlinde series:
$$
K(L,y;q) = V^{ \frac{1}{2}}(\mu(L)     y^{-1};q)V^{- \frac{1}{2}}(\mu(L) y^{-1};q)\,,
$$
where $\mu(L) = L-\mathcal{O}_X$, as it can be written as a product of series only with positive or negative powers of $y^{\frac{1}{2}}$. Thus
$$
V^{\pm\frac{1}{2}}(\mu(L);q) = M(q)^{\frac{1}{2}c_1(L)\cdot c_3(X)}\,.
$$
\item By applying Proposition \ref{mega theorem}, one can show that
$
   V(\alpha;q) = \big(V^{\frac{1}{2}}(\alpha;q)\big)^2 
$.
 \end{enumerate}
\end{remark}

\begin{theorem}
\label{theorem Segre--Verlinde}
Assuming Conjecture \ref{conjecture WC} holds, we have the following Segre--Verlinde correspondence for any choice of orientations on $\textnormal{Hilb}^n(X)$:
$$
V(\alpha;q)=R(\alpha;-q)\,.
$$
\end{theorem}
\begin{proof}
From Proposition \ref{mega theorem} together with \eqref{squareroottoddgenus} and Definition \ref{defuntwisted}, we see after setting $a=\textnormal{rk}(\alpha)$ that
$$
V(\alpha;q) = U(e^z)^{c_1(\alpha)\cdot c_3(X)}\,,\quad \textnormal{where}\quad q=\frac{(1-e^{-z})}{e^{az}}\,.
$$
Changing variables to $t=e^{z}-1$ we obtain 
$$
V(\alpha;q) = U(1+t)^{c_1(\alpha)\cdot c_3(X)}\,,\quad \textnormal{where}\quad q=t(t+1)^{-(a+1)}\,.
$$
We, therefore, see from Lemma \ref{lemma change} that
\begin{equation}
V(\alpha;q) = 
\left\{\begin{array}{ll}
    \displaystyle  U\big[\mathscr{B}_{a+1}(q)^{-c_1(\alpha)\cdot c_3(X)}\big]  & \textnormal{for }a\geq 0 \\
    &\\
  \displaystyle    U\big[ \mathscr{B}_{-a}(-q)^{c_1(\alpha)\cdot c_3(X)}\big] &\textnormal{for }a< 0 \\
\end{array} 
\right.
 \,. 
  \end{equation}
Comparing with \eqref{segreseries} concludes the proof.
\end{proof}
We can also study the series:
$$
Z(\vec{\alpha},\vec{k};q)=1+\sum_{n>0}q^n\chi^{\textnormal{uvir}}\big(\wedge^{k_1}\alpha_1^{[n]}\otimes\ldots \otimes\wedge^{k_M}\alpha_M^{[n]}\big)\,.
$$
We show that they give rise to interesting formulae. This was motivated by investigating the rationality question as studied in \cite{AJLOP} and their example \cite[Ex. 7]{AJLOP}
\begin{example}
For $\alpha\in G^0(X)$, take the series
$
Z(\alpha;q) = \sum_{n>0}\chi^{\textnormal{uvir}}(\alpha^{[n]})\,,
$
then it can be expressed as 
$$Z(\alpha;q) = \frac{\partial}{\partial y}Z(\alpha,y;q)|_{y=0}\,, \quad \textnormal{where}\quad Z(\alpha,y;q)=1+\sum_{n>0}\chi^{\textnormal{uvir}}(\Lambda^{\bullet}_{-y}\alpha^{[n]})\,.$$
Using Proposition \ref{mega theorem}, we have
$$
Z(\alpha,y;q) =\Big[\prod_{n>0}\prod_{k=1}^n\frac{1+ye^{z(-\omega^k_nq)}}{1+y}\Big]^{c_1(\alpha)\cdot c_3(X)}\,,\quad\textnormal{where}\quad q=\frac{1-e^{-z}}{(1+ye^z)^a}\,.
$$
After changing variables $1+u=e^{z}$, this gives 
$$
Z(\alpha,y;q) =\Big[\prod_{n>0}\prod_{k=1}^n\frac{1+y\big(1+u(-\omega^k_nq)\big)}{1+y}\Big]^{c_1(\alpha)\cdot c_3(X)}\,,\quad\textnormal{where}\quad q=\frac{u}{(1+u)(1+y+yu)^a}
$$
Acting with $\partial/\partial y$ on the last formula, using that the terms under the product are equal to 1 for $y=0$ and that the derivative $(\partial/\partial y)u$ exist we obtain from a product rule for infinite products
$$
Z(\alpha;q) =c_1(\alpha)\cdot c_3(X)\sum_{n>0}\sum_{k=1}^n u(-\omega^k_nq)\quad \textnormal{where} \quad u= \frac{q}{1-q}
$$
We can write this as
\begin{align*}
  Z(\alpha;q) =c_1(\alpha)\cdot c_3(X) \sum_{n>0}\frac{(-q)^n}{1-(-q)^n} 
=c_1(\alpha)\cdot c_3(X) S(-q) ,,  
\end{align*}
where $S(q)$ is the Lambert series as considered by Lambert \cite{lambert}.
\end{example}

\subsection{4D-2D-1D correspondence}
\label{4d2d1dsection}
We obtain a one-to-one correspondence between invariants on compact Calabi--Yau fourfolds, surfaces satisfying $c_1(S)^2 = 0$ and elliptic curves by comparing the invariants that we have already computed in Theorem \ref{mega theorem} with the ones that were obtained in \cite{AJLOP} using an entirely different approach relying on equivariant localization.\\

Recall from \cite[Lem. 1]{MOP1} that the virtual obstruction theory on $\textnormal{Quot}_S(\mathbb{C}^N,n)$ is given by
\begin{equation}
\label{eqClF}
\mathbb{F}=\Big(\tau_{[0,1]}\underline{\textnormal{Hom}}_{\textnormal{Quot}_S(\mathbb{C},n)}(\mathcal{I},\mathcal{F})\Big)^\vee\,,
\end{equation}
where $\mathcal{I}=(\mathcal{O}\to \mathcal{F})$ is the universal complex on $\textnormal{Quot}_S(\mathbb{C},n)$.
 When $N=1$, we have  $\textnormal{Quot}_S(\mathbb{C}^1,n)= \textnormal{Hilb}^n(S)$ and the virtual fundamental classes get identified by Oprea--Pandharipande \cite[eq. (31) ]{OP1} with
$$
\big[\textnormal{Hilb}^n(S)\big]^{\textnormal{vir}} = \big[\textnormal{Hilb}^n(S)\big]\cap c_n(K_\textnormal{Hilb}^{[n]}(S)\,^\vee)\,.
$$
using that $\textnormal{Hilb}^n(S)$ is smooth. Here, $\chi^{\textnormal{vir}}(-)$ denotes the virtual Euler characteristic of Fantechi--Göttsche \cite{GF}.
\begin{theorem}
\label{4d2d1d}
 Let $X$ be a Calabi--Yau fourfold for which Conjecture \ref{conjecture WC} holds and $S$ a surface satisfying $c_1(S)^2=0$. Let $f_1,\ldots,f_M, g$ be power-series, $\alpha_{Y,1},\ldots,\alpha_{Y,M}$ in $G^0(Y)$ for $Y=X,S$ and $\textnormal{rk}(\alpha_{Y,i})=a_i$, then there exist universal series $A_1,\ldots,A_M$ depending on $f_i,\{g\}$ and $a_i$ such that
\begin{align*}
    1+\sum_{n>0}q^n\int_{[\textnormal{Hilb}^n(S)]^\textnormal{vir}}f_1\big(\alpha_{S,1}^{[n]}\big)\cdots f_M\big(\alpha_{S,M}^{[n]}\big)\{g\}\big(T^{\textnormal{vir}}_{\textnormal{Hilb}^n(S)}\big) &=\prod_{i=1}^MA_i^{c_1(\alpha_{S,i})\cdot c_1(S)}\,,  \\
  1+\sum_{n>0}q^n\int_{[\Hilb]^{\textnormal{vir}}}f_1\big(\alpha_{X,1}^{[n]}\big)\cdots f_M\big(\alpha_{X,M}^{[n]}\big)g\big(T^{\textnormal{vir}}_n\big) &=\prod_{i=1}^MU(A_i)^{c_1(\alpha_{X,i})\cdot c_3(X)}\,.
\end{align*}
 Moreover, there are universal generating series $B_i$ depending on $f_i$, $a_i$, such that
\begin{align*}
   1+\sum_{n>0} q^n\chi^{\textnormal{vir}}\Big(f_1\big(\alpha^{[n]}_{S,1}\big)\otimes\ldots \otimes f_M\big(\alpha^{[n]}_{S,M}\big)\Big)&=\prod_{i=1}^MB_i^{c_1(\alpha_{S,i})\cdot c_1(S)}\,,\\
   1+\sum_{n>0} q^n\chi^{\textnormal{uvir}}\Big(f_1\big(\alpha^{[n]}_{X,1}\big)\otimes\ldots \otimes f_M\big(\alpha^{[n]}_{X,M}\big)\Big) &= \prod_{i=1}^MU(B_i)^{c_1(\alpha_{X,i})\cdot c_3(X)}\,.
\end{align*}
where we abuse the notation by thinking of $\mathscr{G}_{f_i}$ as mapping to $G^0(-)\otimes \Q$.
\end{theorem}
\begin{proof}
AJLOP \cite{AJLOP} prove general formulae for generating series
$$
\sum_{n\in \ZZ}q^n\int_{[\textnormal{Quot}_S(\CC^N,\beta,n)]^{\textnormal{vir}}}f_1\big(\alpha_1^{[n]}\big)\cdots f_M\big(\alpha_M^{[n]}\big)h\big(T^{\textnormal{vir}}_{\textnormal{Hilb}^n(S)}\big)\,.
$$
When $\beta=0$, $N=1$ and $K^2_S=0$ the results of \cite[§2.2 \& Eq. (14)]{AJLOP} imply
\begin{align*}
   &1+\sum_{n>0}q^n\int_{[\textnormal{Hilb}^n(S)]^\textnormal{vir}}f_1\big(\alpha_1^{[n]}\big)\cdots f_M\big(\alpha_M^{[n]}\big)h\big(T^{\textnormal{vir}}_{\textnormal{Hilb}^n(S)}\big)=\prod_{i=1}^M \bigg[\frac{f_i\big(H(q)\big)}{f_i(0)}\bigg]^{c_1(\alpha_i)\cdot c_1(S)}\,,\\
  &\textnormal{where}\quad    q=\frac{H}{\prod f_i^{a_i}(H)\,h(H)} \,.
\end{align*}
 Replacing $h$ with $\{g\}$, and comparing to the result of Proposition \ref{mega theorem}, we obtain the first two formulae. 

Using \eqref{squareroottoddgenus}, we see that $\big[\sqrt{Td}\big]\Big(T^{\textnormal{vir}}_{\textnormal{Hilb}^n(S)}\Big)E^{\frac{1}{2}}$ contributes $$
\frac{x}{1+e^{-x}}
$$
to the variable change above. This corresponds precisely to the Todd-genus $\textnormal{Td}\Big(T^{\textnormal{vir}}_{\textnormal{Hilb}^n(S)}\Big) = \frac{x}{1+e^{-x}}\Big(T^{\textnormal{vir}}_{\textnormal{Hilb}^n(S)}\Big)$. The second result for a surface $S$ satisfying $c_1(S)^2=0$ then follows from the virtual Riemann--Roch of Fantechi--Göttsche \cite{GF} together with the definition of $\chi^{\textnormal{uvir}}(-)$ in §\ref{k-theoretic paragraph}. 
\end{proof}
\begin{remark}
By the work of Oprea--Pandharipande \cite[Lem. 34]{OP1} there is a relation between integrals over $[\textnormal{Quot}_C(\mathbb{C}^N,n)]$ and $[\textnormal{Quot}_S(\mathbb{C}^N,n)]^{\textnormal{vir}}$, where the former is a smooth moduli space of dimension $nN$ and $C$ is a smooth anti-canonical curve in $S$ (if it exists). When $\alpha \in K^0(S)$ and $\Theta$ is the theta divisor on $C$, we obtain
$$
1+\sum_{n>0}\int_{C^{[n]}}f\big((\alpha|_C\big)^{[n]})g\big(TC^{[n]}+\big(\Theta^{[n]}\big)^\vee-\Theta^{[n]}\big)
$$
Interestingly, this gives a precise relation between the generating series of three sets of virtual invariants in 3 different dimensions. We will unify these results by applying similar arguments to the ones in §\ref{cao--kool section} and §\ref{VFC section} to extend the results of AJLOP \cite{AJLOP}, \cite{Lim}, and \cite{OP1} in the author's future work \cite{bojkoquot} where we replace $\mathbb{C}^N$ by a general torsion free sheaf $E$. In the Calabi--Yau case, we will assume $E$ additionally to be rigid and simple to construct virtual fundamental classes.
\end{remark}
Using that $[\textnormal{Hilb}^1(X)]^{\textnormal{vir}}=\textnormal{Pd}\big(c_3(x)\big)$ together with Theorem \ref{VRROT} and that we have natural isomorphisms $\Lambda^i(TX|_x)\cong Ext^i(\mathcal{O}_x,\mathcal{O}_x)$ which hold in a family one can show:
\begin{corollary}
\label{corollary n=1}
All of the results of this section hold $\textnormal{mod }q^2$.
\end{corollary}

\subsection{4D-2D correspondence explained by wall-crossing}
\label{sec:4d2d}
When $S$ is a surface with $c_1(S)^2=0$, we will follow the same line of arguments for the classes $[\Quot]^{\vir}$  as we did for $[\Hilb]^{\vir}$ in §\ref{VFC section}. We will again only need the following small class of invariants
$$
I(L,q) = 1+\sum_{n>0}q^n\int_{\big[\textnormal{Quot}_S(\mathbb{C}^1,n)\big]^{\textnormal{vir}}}c_n(L^{[n]})
$$
 to determine $[\textnormal{Quot}_S(\mathbb{C}^N,n)]^{\textnormal{vir}}$ as an element in $H_{nN}(\mathcal{P}_S)$. For this, we will need a different definition of the moduli stack of pairs. For simplicity, we assume that $b_1(S)=0$, but we then drop this requirement in Remark \ref{remark b1}. The result describing descendent integrals over $\Quot$ in Lemma \ref{lemma workhorse} is new but its main consequence noted down in Proposition \ref{mega theorem surfaces} was already shown by AJLOP \cite[§2]{AJLOP} using different methods. In the sequel \cite{bojkoquot}, we are going to obtain these results for $[\textnormal{Quot}_{S}(E,n)]^{\textnormal{vir}}$ for any surface and $E$ a torsion-free sheaf using the current approach. These cases could not be addressed in this generality using the same methods as in \cite{AJLOP}.

Let us for now set up the general framework for $[\textnormal{Quot}_S(\mathbb{C}^N,n)]^{\textnormal{vir}}$ for any smooth projective surface $S$.
\begin{definition}
\begin{itemize}
    \item We consider this time the abelian category $\mathcal{B}_N$ of  triples $(E,V,\phi)$, where $V$ is a complex vector space, $F$ is a zero-dimensional sheaf, and $\phi: V\otimes\mathbb{C}^N\otimes \mathcal{O}_S\xrightarrow{\phi}F$ is a morphism of sheaves.\footnote{If we continued using the category of pairs $V\otimes \mO_S\to F$ without restricting to $\mB_N$, we would wall-cross again into $[\textnormal{Hilb}^n(S)]^\vir$ not into the Quot schemes. At least if we follow the assumptions of \protect\cite[§5]{JoyceWC} carefully.}
    \item The moduli stack $\mathcal{N}^N_1$ is the stack of objects in $\mB_N$. It is constructed in a similar way as in Definition \ref{definition pairs}.
    \item We define $\Theta^{N,\textnormal{pa}}$ by
       \begin{align*} 
 \Theta^{N,\textnormal{pa}}_{(n_1p,d_1),(n_2p,d_2)} =& (\pi_{n_1p,d_1}\times \pi_{n_2p,d_2})^* \Big\{(\Theta_{n_1p,n_2p})_{1,3}\\&
\oplus  \Big((\mathcal{V}^{\oplus N}_{d_1})\boxtimes \pi_{2\,*}(\mathcal{E}_{n_2p})^\vee\Big)_{2,3}[1]\oplus \big(\mV_{d_1}\boxtimes \mV_{d_2}^*\big)_{2,4} \Big\}
     \numberthis
   \end{align*}
    with $\Theta_{n_1p,n_2p}=\underline{\textnormal{Hom}}_{\mathcal{M}_{n_1p}\times \mathcal{M}_{n_2p}}(\mathcal{E}_{n_1p},\mathcal{E}_{n_2p})^\vee$ and the form $\chi^{N,\textnormal{pa}}\big((n_1p,d_1),(n_2p,d_2)\big)=\textnormal{rk}\big(\Theta^{N,\textnormal{pa}}_{(n_1p,d_1),(n_2p,d_2)}\big)=   -Nd_1n_2+d_1d_2$.
\end{itemize}
The rest of the data has an obvious modification, which we do not mention here.
\end{definition}
Note that when working with Behrend--Fantechi virtual fundamental classes (at least in the setting of \cite{JoyceWC}), the correct vertex algebra structure requires the symmetrization of $\Theta^{\textnormal{pa}}$, which leads to the following set of data:
\begin{equation}
\big((\mathcal{N}^N_1)^{\textnormal{top}},\ZZ\times \ZZ,\mu_{\mathcal{N}_1}^{\textnormal{top}},\mu_{\mathcal{N}_1}^{\textnormal{top}}, 0^{\textnormal{top}}, \llbracket \Theta^{N,\textnormal{pa}} +\sigma^*\big(\Theta^{N,\textnormal{pa}}\big)^\vee\rrbracket,\epsilon^N\big)
\end{equation}
where $\epsilon^N_{(n_1p,d_1),(n_2p,d_2)} = (-1)^{Nd_1n_2+d_1d_2}$.
We have again a universal family $\mathbb{C}^N\otimes \mathcal{O}_{S\times \textnormal{Quot}_{S}(\mathbb{C}^N,n)}\to \mathcal{F}$ giving us
$$
    \begin{tikzcd}
       &\mathcal{N}^N_1\arrow[d,"\Pi^{\textnormal{pl}}"]\\
\Quot \arrow[ur,"\iota_{n,N}"]\arrow[r,"\iota^{\textnormal{pl}}_{n,N}"]&(\mathcal{N}^N_1)^{\textnormal{pl}}
    \end{tikzcd}\,.
$$
and $$[\Quot]^{\vir}_{\mN} \in H_*(\mathcal{N}^N_1)\,.$$
We take the obvious modification $\tau^{\textnormal{pa}}_N$ of Joyce--Song stability, such that $\mathbb{C}^N\otimes \mathcal{O}_X\xrightarrow{\phi} F$ is $\tau^{\textnormal{pa}}_N$\textit{-stable} if and only if $\phi$ is surjective. Therefore, we again obtain 
$$\Quot = N^{\textnormal{st}}_{(np,1)}(\tau^{\textnormal{pa}}_N)\,,\quad  [\Quot]^{\textnormal{vir}} = [N^{\textnormal{st}}_{(np,1)}(\tau^{\textnormal{pa}}_N)]^{\textnormal{vir}}\,.$$ 
Once the work of Joyce \cite{JoyceWC} is complete, the following conjecture will be a consequence of a more general theorem after proving that some axioms are satisfied. \footnote{Joyce has made his work \protect\cite{JoyceWC} publicly available and the necessary assumptions from \protect\cite[§5]{JoyceWC} have been checked in \protect\cite[Thm. A.4]{bojkoquot}. This makes the Corollary below and all the results depending on it into a theorem. To emphasize the timeline of the development of the subject, I keep it as a conjecture in this work.} 
\begin{conjecture}
\label{conjecture quot WC}
For any smooth projective surface $S$, in $\check{H}_{*}\big(\mathcal{N}^N_1\big)$  we have for all $n,N$
$$
[\textnormal{Quot}_S(\mathbb{C}^N,n)]^{\vir}_{\mN}=\sum_{\begin{subarray}
a k>0,n_1,\ldots,n_k\\
n_1+\ldots +n_k=n
\end{subarray}}\frac{(-1)^k}{k!}\big[\big[\ldots \big[[\mathcal{N}_{(0,1)}]^{\inv},[\mathcal{M}_{n_1p}^{\textnormal{ss}}]^{\inv}],\ldots ],[\mathcal{M}_{n_kp}^{\textnormal{ss}}]^{\inv}]
$$
for some $[\mathcal{M}_{np}^{\textnormal{ss}}]^{\inv}\in \check{H}_2(\mathcal{N}^N_1)$.
\end{conjecture}
We again construct the vertex algebra on topological pairs and the \textit{$L$-twisted} vertex algebra.
\begin{definition}
\label{def twistedbylquot}
Define the data $\big(\mathcal{P}_S, K(\mathcal{P}_S), \Phi_{\mathcal{P}_S}, \mu_{\mathcal{P}_S}, 0, \theta^{L,N}_{\mathcal{P}_S}, \tilde{\epsilon}^{L,N}\big)$, $\big(\mathcal{P}_S, K(\mathcal{P}_S), \Phi_{\mathcal{P}_S}, \mu_{\mathcal{P}_S}, 0, \theta^N_{\mathcal{P}_S}, \tilde{\epsilon}^{N}\big)$ as follows:

\begin{itemize}
    \item $K(\mathcal{P}_S) = K^0(S)\times \ZZ$.
    \item Set $\mathfrak{L} = \pi_{2\,*}(\pi_S^*(L)\otimes  \mathfrak{E})\in K^0(\mathcal{C}_S)$.  Then on $\mathcal{P}_S\times \mathcal{P}_S$ we define
   $$
   \theta^{N, \textnormal{ob}}=(\theta)_{1,3}-N\Big(\mathfrak{U}\boxtimes \pi_{2\,*}(\mathfrak{E})^\vee\Big)_{2,3} + \big(\mathfrak{U}\boxtimes \mathfrak{U}^{\vee}\big)_{2,4}$$ where $\theta= \pi_{2,3\,*}\big(\pi^*_{1,2}(\mathfrak{E})\cdot\pi^*_{1,3}(\mathfrak{E})^\vee\big)$ and
\begin{align*}
  &\theta^N_{\mathcal{P}_S} =  \theta^{N,\textnormal{ob}} +\sigma^*(\theta^{N,\textnormal{ob}})^\vee\\
   & \theta^{L,N}_{\mathcal{P}_S} =\theta_{\mathcal{P}_S,N}+  N\Big(\mathfrak{U}\boxtimes \mathfrak{L}^\vee\Big)_{2,3}+N\Big(\mathfrak{L}\boxtimes \mathfrak{U}^\vee\Big)_{1,4}\,,
\end{align*}
\item The symmetric forms $\tilde{\chi}^N: (K^0(S)\times \ZZ)\times (K^0(S)\times \ZZ)\to \ZZ$, $\tilde{\chi}^{L,N}: (K^0(S)\times \ZZ)\times (K^0(S)\times \ZZ)\to \ZZ $ are given by 
\hspace{-20mm}\begin{align*}
\label{chiLtilquot}
\tilde{\chi}^N\big((\alpha,d),(\beta,e)\big)& = \chi(\alpha,\beta)+\chi(\beta,\alpha)-dN\chi(\beta)-eN\chi(\alpha)+2de\,,\\
\tilde{\chi}^{L,N}\big((\alpha,d),(\beta,e)\big)&=\chi(\alpha,\beta)  -dN\big(\chi(\beta)-\chi(\beta\cdot L)\big) \\
&-eN\big(\chi(\alpha) - \chi(\alpha\cdot L)\big)+2de\,.
\numberthis
\end{align*}
\item The signs are defined by $\tilde{\epsilon}^N_{(\alpha,d),(\beta,e)} =(-1)^{\chi(\alpha,\beta)+Nd\chi(\beta)+de}$ and $\tilde{\epsilon}^{L,N}_{(\alpha,d),(\beta,e)} =(-1)^{\chi(\alpha,\beta)+Nd\big(\chi(\beta)-\chi(L\cdot\beta)\big)+de}$. 
\end{itemize}
We denote by $(\hat{H}_*(\mathcal{P}_S),\ket{0}, e^{zT}, Y^N)$, resp. $(\hat{H}^L_*(\mathcal{P}_S),\ket{0}, e^{zT}, Y^{L,N})$ the vertex algebras associated to this data and $(\check{H}_*(\mathcal{P}_S),[-,-]^N)$, resp. $(\check{H}^L_*(\mathcal{P}_X),[-,-]^{L,N})$ the corresponding Lie algebras. We now consider the map
\begin{equation}
    \Omega^N=( \Gamma\times\textnormal{id})\circ (\Sigma^N)^{\textnormal{top}}:(\mN_1^N)^{\textnormal{top}}\to  (\mathcal{M}_X)^{\textnormal{top}}\times BU\times\ZZ\to \mathcal{C}_X \times BU\times\ZZ\,,
\end{equation} 
where $\Sigma^N$ maps a point $[E,V,\phi]$ to $[E,V\otimes \mathcal{O}_S]$.
\end{definition}
Let $\mathbb{B}=B\sqcup \{(0,1)\}$, where $B=\bigsqcup_{i=1}^4B_i$, $\textnormal{ch}(B_i)$ basis of $H^i(S)$ with $B_0 =\{\llbracket\mathcal{O}_S\rrbracket \}$, $B_4 = \{p\}$. Combining all the ideas we used for fourfolds, we can state the following:
\begin{proposition}
\label{propmorphismLquot}
Let $\Q[K^0(S)\times \ZZ]\otimes_{\Q}\textnormal{SSym}_{\Q}[u_{\sigma,i}, \sigma\in \mathbb{B},i>0]$ be the generalized super-lattice vertex algebra associated to $\big((K^0(S)\oplus\ZZ)\oplus K^1(S),(\tilde{\chi}^{L,N})^\bullet\big)$, resp. $\big((K^0(S)\oplus\ZZ)\oplus K^1(S),(\tilde{\chi}^N)^\bullet\big)$, where $(\tilde{\chi}^N)^\bullet = \tilde{\chi}^N\oplus \chi^-$, $(\tilde{\chi}^{L,N})^\bullet = \tilde{\chi}^{L,N}\oplus \chi^-$  and 
\begin{align*}
\chi^-&: K^1(S)\times K^1(S)\to \ZZ\,,\\ \chi^-(\alpha,\beta) &= \int_S\textnormal{ch}(\alpha)^\vee\textnormal{ch}(\beta)\textnormal{Td}(S)+\int_S\textnormal{ch}(\beta)^\vee\textnormal{ch}(\alpha)\textnormal{Td}(S)\,.
\end{align*}
Notice that all three pairings $\tilde{\chi}^N$, $\tilde{\chi}^{L,N}$ and $\tilde{\chi}$ are symmetric unlike the case in Proposition \ref{theorem VA}.
The isomorphism \eqref{isohomol} induces an isomorphism of graded vertex algebras for all $N$:
\begin{align*}
\hat{H}_*(\mathcal{P}_S)\cong \Q[K^0(S)\times \ZZ]\otimes_{\Q}\textnormal{SSym}_{\Q}[ u_{\sigma,i}, \sigma\in \mathbb{B},i>0]\,,\\
\hat{H}^L_*(\mathcal{P}_S)\cong \Q[K^0(S)\times \ZZ]\otimes_{\Q}\textnormal{SSym}_{\Q}[u_{\sigma,i}, \sigma\in \mathbb{B},i>0]\,.
\end{align*}
 \sloppy The map $(\Omega^N)_*: H_*(\mathcal{N}^N_1)\to H_*(\mathcal{P}_S)$ induces morphisms of graded vertex algebras $(\hat{H}_*(\mathcal{N}^N_1),\ket{0},e^{zT}, Y^N)\to (\hat{H}_*(\mathcal{P}_S),\ket{0},e^{zT}, Y^N)$, $(\hat{H}^L_*(\mathcal{N}^N_1),\ket{0},e^{zT}, Y^{L,N})\to (\hat{H}^L_*(\mathcal{P}_S),\ket{0},e^{zT}, Y^{L,N})$ and of graded Lie algebras
\begin{align*}
\bar{\Omega}_* :\big(\check{H}_*(\mathcal{N}^N_1),[-,-]^N\big)\longrightarrow \big(\check{H}_*(\mathcal{P}_S),[-,-]_N\big)\,,\\
\bar{\Omega}_* :\big(\check{H}^L_*(\mathcal{N}^N_1),[-,-]^{L,N}\big)\longrightarrow \big(\check{H}^L_*(\mathcal{P}_S),[-,-]^L_N\big)\,.
\end{align*}
\end{proposition}
We will from now on omit writing $N$ in the superscript of $Y^N$, $Y^{L,N}$, $[-,-]^N$ and $[-,-]^{L,N}$ to easy the notation. 

 The following result replaces Theorem \ref{theorem workhorse} and it is noticeably simpler due to canonical orientations. We use the notation
$$
\mathscr{Q}_{N,n}=\bar{\Omega}^N_*\Big(\big[\Quot\big]^{\vir}_{\mN}\Big)\,,\quad\textnormal{and}\quad  \mathscr{M}_{np}= \bar{\Omega}^N_*\big([\mathcal{M}_{np}]^{\inv}\big)\,.
$$
\begin{lemma}
\label{lemma workhorse}
Let $S$ be a smooth projective surface with $b_1(S)=0$. If Conjecture \ref{conjecture quot WC} holds, then 
$$
  \mathscr{M}_{np} = e^{(np,1)}\otimes 1\cdot \mathscr{N}_{np} + \Q T(e^{(np,1)}\otimes 1)\,,
$$
where for the series $\mathscr{N}(q) = \sum_{n>0}\mathscr{N}_{np}q^n$ we have
$$
\textnormal{exp}\big(\mathscr{N}(q)\big) =  \Big(1-e^pq\Big)^{\Big(\sum_{v\in B_2}c_1(S)_vu_{v,1}\Big)}\,.
$$
If $c_1(S)^2=0$, we have 
\begin{equation}
\label{QNn}
 1+ \sum_{n>0}\frac{\mathscr{Q}_{N,n}}{e^{(np,1)}}q^n=\textnormal{exp}\bigg[\sum_{n>0}[z^{nN-1}]\Big\{\sum_{v\in B_2}-\frac{c_{1,v}}{n}U_v(z)\textnormal{exp}\Big[\sum_{k>0}\frac{ny_k}{k}z^k\Big]\Big\}q^n\bigg]\,.
\end{equation}
\end{lemma}
\begin{proof}
We have $[\textnormal{Quot}_S(\mathbb{C}^1,n)]^{\textnormal{vir}}\cap c_n(L^{[n]}) = [\textnormal{Hilb}^n(S)]\cap c_n\big((K_\textnormal{Hilb}^n(S))^\vee\big)\cap c_n\big(L^{[n]}\big)=(-1)^n[\textnormal{Hilb}^n(S)]\cap c_{2n}(K^{[n]}_S\oplus L^{[n]})$ for an algebraic line bundles $L\to S$. Then by \cite[eq. (18)]{MOPhigher}, we see
$$
I(L,q) = \bigg(\frac{1}{1-q}\bigg)^{c_1(L)\cdot c_1(X)}\,.
$$
Using that $H^2(X)=H^{1,1}(X)$ because of $b_1(X)=0$, we have $\textnormal{ch}(B_2)\subset H^{1,1}(X)$ and therefore the above result for algebraic line bundles is sufficient. Relying on Proposition \ref{propmorphismLquot}, we obtain by a similar computation as in the proof of Theorem \ref{theorem workhorse} that
$$
    [e^{(mp,1)}\otimes 1, e^{(np,0)}\otimes N_{np}]^L = -(-1)^ne^{(m+n)p,1}\otimes \sum_{v\in B_2}\int_Xc_1(L)\textnormal{ch}(v)a_{v}(n)\,.
$$
By simpler arguments than in the proof of Theorem \ref{theorem workhorse} because of the absence of orientations, we obtain $\mathscr{N}(np) = \frac{1}{n}\sum_{v\in B_2}c_1(S)_vu_{v,1}$\,. Then an analogous argument as in the proof of Theorem \ref{theoremhilb} leads to \eqref{QNn} because of $c^2_1(S)=0$.
\end{proof}
\begin{remark}
\label{remark b1}
Going through the above computation without the assumption $b_1(S)=0$, we need two modifications. Firstly, we would use the result of Oprea--Pandharipande \cite[Cor. 15]{OP1} instead of Marian--Oprea--Pandharipande  \cite[eq. (18)]{MOP1} in the proof of Lemma \ref{lemma workhorse} to avoid the issue of purely topological line bundles. One can moreover check that under the projection $\Pi_{\textnormal{even}}:\check{H}_*(\mathcal{P}_S)\to \check{H}_{\textnormal{even}}(\mathcal{P}_S)$ we still obtain the same results for a surface with $c_1(S)^2 = 0$. This is sufficient for us, because we never integrate odd cohomology classes, except when integrating polynomials in $\textnormal{ch}_k(T^{\textnormal{vir}})$, but as the only terms $\mu_{v,k}$ for $v\in B_{\textnormal{odd}}$ are given for $v\in B_3$, each such integral will contain a factor of $\chi^-(v,w)=0$ for $v,w\in B_3$. 
\end{remark}
As a consequence, we then obtain the following result which could also be extracted from AJLOP \cite{AJLOP} for a surface satisfying $c_1(S)^2 = 0$.

\begin{proposition}
\label{mega theorem surfaces}
Let $S$ be a smooth projective surface satisfying $c_1(S)^2=0$ and $f_0(\mathfrak{p},\cdot), f_1(\mathfrak{p},\cdot),\ldots, f_m(\mathfrak{p},\cdot)$ be power-series with $f(0,0)=1$, then define 
\begin{align*}
   \textnormal{Inv}_N(\vec{f}, \vec{\alpha},q)=1+\sum_{n>0}\int_{\big[\Quot\big]^{\textnormal{vir}}}f_0\big(T^{\textnormal{vir}}\big)f_1\big(\alpha^{[n]}_1\big)\cdots f_m\big(\alpha^{[n]}_m\big) q^n\,.
\end{align*}
Setting $\textnormal{rk}(\alpha_j)=a_j$, we have
\begin{align*}
      \textnormal{Inv}_N(\vec{f},\vec{a},q)& = \bigg[\prod_{j=1}^N\prod_{i=1}^m\frac{f_i( H_j(q))}{f_i(0)}\bigg]^{c_1(\alpha_j)\cdot c_1(S)}\,,
      \numberthis
\end{align*}
where $H_j(q)$, $j=1,\ldots, N$ are the different solutions for
$$
      q=\frac{H_j^N}{\prod_{i=1}^m f_i^{a_i}(H_j)\,f_0^N(H_j)}\,.
$$
\end{proposition}
\begin{proof}
We can show again that
\begin{align*}
&\int_{[\Quot]^{\textnormal{vir}}}f_0(T^{\textnormal{vir}})f_i(\alpha_i^{[n]})\ldots f_m(\alpha_m^{[n]})\\
&= \int_{\mathscr{Q}_{N,n}}\textnormal{exp}\Big[\sum_{\begin{subarray}a
k>0\\
v\in B_{2,4}
\\
\end{subarray}}\sum_{i=1}^ma_{\alpha_i}(k)\chi(\alpha_i^\vee,v)\mu_{v,k}+Nb_k\chi(v)\mu_{v,k}\Big]\,,
\end{align*}
where $\sum_{k}\frac{a_{\alpha_i(k)}}{k!}q^k = \textnormal{log}\big(f_i(q)\big)$ and $\sum_{k>0}\frac{b_k}{k!}q^k = \textnormal{log}\big(f_0(q)\big)$.
The rest then follows from Lemma \ref{linear exp} and \ref{Gessel-Lagrange} by a similar computation as in §\ref{VFC section}.
\end{proof}
\begin{remark}
For an elliptic curve, $C$ the Quot scheme $\textnormal{Quot}_C(\mathbb{C}^N,n)$ carries the obstruction theory $\mathbb{F}=\Big(\tau_{[0,1]}\underline{\textnormal{Hom}}_{\textnormal{Quot}_C(\mathbb{C}^N,n)}(\mathcal{I},\mathcal{F})\Big)^\vee$ constructed by Marian--Oprea \cite{MO1} which is just a vector bundle of rank $nN$, therefore the construction of the vertex algebra is identical and the same computation applies. We leave it to the reader to check using \cite[Thm. 3]{OP1} that under the projection $\Pi_{\textnormal{even}}:\check{H}_*(\mathcal{P}_C)\to \check{H}_{\textnormal{even}}(\mathcal{P}_C)$ the generating series $1+\sum_{n>0}\frac{\mathscr{Q}_{N,n}}{e^{(np,1)}}$ is given by 
$$
\textnormal{exp}\Big[-\sum_{n>0}\frac{(-1)^n}{n}[z^{nN-1}]\Big\{U_{\llbracket \mathcal{O}_C\rrbracket}(z)\textnormal{exp}\Big[\sum_{k>0}\frac{ny_k}{k}z^k\Big]\Big\}q^n\Big]\,.
$$
\end{remark}

\bibliography{mybib.bib}
\end{document}